\documentclass{amsart}
\usepackage[T1]{fontenc}
\usepackage{amssymb}
\usepackage{amsmath}
\usepackage{amsfonts}
\usepackage{geometry}
\usepackage{hyperref}
\usepackage{color}

\hypersetup{
    colorlinks=true,
    linkcolor=black,
    citecolor=black,
    filecolor=black,
    urlcolor=black,
}
\newtheorem{theorem}{Theorem}
\theoremstyle{plain}

\newtheorem{corollary}[theorem]{Corollary}

\newtheorem{definition}[theorem]{Definition}
\newtheorem{example}[theorem]{Example}

\newtheorem{lemma}[theorem]{Lemma}

\newtheorem{proposition}[theorem]{Proposition}
\newtheorem{remark}[theorem]{Remark}

\numberwithin{equation}{section}
\input{tcilatex}

\begin{document}
\title{General Rough integration, L\'evy Rough paths and a L\'evy--Kintchine
type formula}
\author{Peter K. Friz, Atul Shekhar}

\begin{abstract}
We consider rough paths with jumps. In particular, the analogue of Lyons'
extension theorem and rough integration are established in a jump setting,
offering a pathwise view on stochastic integration against c\'adl\'ag
processes. A class of L\'evy rough paths is introduced and characterized by
a sub-ellipticity condition on the left-invariant diffusion vector fields
and and a certain integrability property of the Carnot--Caratheodory norm
with respect to the L\'evy measure on the group, using Hunt's framework of
Lie group valued L\'evy processes. Examples of L\'evy rough paths include
standard multi-dimensional L\'evy process enhanced with stochastic area as
constructed by D. Williams, the pure area Poisson process and Brownian
motion in a magnetic field. An explicit formula for the expected signature
is given.
\end{abstract}

\thanks{ The first author is supported by European Research Council under
the European Union's Seventh Framework Programme (FP7/2007-2013) / ERC grant
agreement nr. 258237. The second author is supported by Berlin Mathematical
School. Date \today}
\maketitle
\tableofcontents


\part{Introduction and background}

\section{Motivation and contribution of this paper}

An important aspect of \textquotedblleft general" theory of stochastic
processes \cite{Jacod, RogersWilliams, Protter} is its ability to deal with
jumps. On the other hand, the (deterministic) theory of rough paths \cite%
{Lyons, LyonsQian, LyonsStFlour, FrizVictoir, FrizHairer} has been very
successful in dealing with \textit{continuous} stochastic processes (and
more recently random fields arising from SPDEs \cite{Lectures, FrizHairer}).
It is a natural question to what extent there is a \textquotedblleft
general" rough path theory i.e. allowing for jumps. To the best of our
knowledge, there has been only one discussion, by Williams \cite{williams}
in the spirit of Marcus canonical equations, and essentially nothing else.
(We should however Mikosch--Norvai{\v{s}}a \cite{MN00} and Simon \cite{Simon}%
, noting that these works take place in the Young regime of finite $p$%
-variation, $p<2 $.) We will comment in more detail in Section \ref%
{sec:williamsintro} about Williams' work and the relation to ours. 

Postponing the exact definition of ``general" (a.k.a. c\'adl\'ag) rough
path, let us start with a list of desirable properties and natural questions.

\begin{itemize}
\item An analogue of Lyons' fundamental \textit{extension theorem} (Section %
\ref{sec:gRP} below for a recall) should hold true. That is, any general
geometric $p$-rough path $\mathbf{X}$ should admit canonically defined
higher iterated integrals, thereby yielding a group-like element (the
``signature" of $\mathbf{X}$).

\item A general rough path $\mathbf{X}$ should allow the integration of $1$%
-forms, and more general suitable ``controlled rough paths" $Y$ in the sense
of Gubinelli \cite{FrizHairer}, leading to rough integrals of the form 
\begin{equation*}
\int f\left( X^-\right) d\mathbf{X}\text{ and }\int Y^- d\mathbf{X}.
\end{equation*}

\item Every semimartingale $X=X\left( \omega \right) $ with (rough path)
It\^o-lift $\mathbf{X}^I=\mathbf{X}^I\left( \omega \right) $, should give
rise to a (random) rough integral that coincides under reasonable
assumptions with the It\^o-integral, so that a.s. 
\begin{equation*}
\mathrm{(It\hat{o})}\int f\left( X^{-}\right) dX=\int f\left( X\right) d%
\mathbf{X}^I.
\end{equation*}

\item As model case for both semi-martingales and jump Markov process, what
is the precisely rough path nature of L\'evy processes? In particular, it
would be desirable to have a class of ``L\'evy rough paths" that captures
natural (but ``non-canoncial") examples such as the pure area Poisson
process or the Brownian rough path in a magnetic field?

\item To what extent can we compute the \textit{expected signature} of such
processes? And we do we get from it?
\end{itemize}

In essence, we will give reasonable answers to all these points. We have not
tried to push for maximal generality. For instance, in the spirit of
Friz--Hairer \cite[Chapter 3-5]{FrizHairer}, we develop general rough
integration only in the level $2$-setting, which is what matters most for
probability. But that said, the required algebraic and geometric picture to
handle the level $N$-case is still needed in this paper, notably when we
discuss the extension theorem and signatures. 
For the most, we have chosen to work with (both ``canonically" and
``non-canonically" lifted) L\'evy processes as model case for random
c\'adl\'ag rough paths, this choice being similar to choosing Brownian
motion over continuous semimartingales. In the final chapter we give discuss
some extensions, notably to Markov jump diffusions and some simple Gaussian
examples.

\textbf{Applications and outlook.} In his landmark paper \cite[p.220]{Lyons}%
, Lyons gave a long and visionary list of advantages (to a probabilist) of
constructing stochastic objects in a pathwise fashion most of which extend
mutatis mutandis to situations with jump noise. 
We also note that integration against general rough paths can be considered
as a generalization of the F\"{o}llmer integral \cite{F81} and, to
some extent, Karandikar \cite{K95}, (see also Soner et al. \cite{STZ11} and the 
references therein), but now free of implicit semimartingale features. 

\section{Preliminaries}

\subsection{``General" Young integration}

\cite{Young,dudley}

We briefly review Young's integration theory. Consider a path $X:\left[ 0,T%
\right] \rightarrow \mathbb{R}^{d}$ of finite $p$-variation, that is 
\begin{equation*}
\left\Vert X\right\Vert _{p\text{-var;}\left[ 0,T\right] }:=\left( \sup_{%
\mathcal{P}}\sum_{\left[ s,t\right] \in P}\left\vert X_{s,t}\right\vert
^{p}\right) ^{1/p}<\infty
\end{equation*}%
with $X_{s,t}=X_{t}-X_{s}$ and $\sup $ (here and later on) taken over all
for finite partitions $\mathcal{P}$ of $\left[ 0,T\right] $. As is
well-known tsuch paths are \textit{regulated} in the sense of admitting
left- and right-limits. In particular, $X_{t}^{-}:=\lim_{s\uparrow t}X_{s}$
is c\'{a}gl\'{a}d and $X_{t}^{+}:=\lim_{s\downarrow t}X_{s}$ c\'{a}dl\'{a}g
(by convention: $X_{0}^{-}\equiv X_{0},\,X_{T}^{+}\equiv X_{T}$). Let us
write $X\in W^{p}([0,T])$ for the space of c\'{a}dl\'{a}g path of finite $p$%
-variation. A generic c\'{a}gl\'{a}d path of finite $q$-variation is then
given by $Y^{-}$ for $Y\in W^{q}([0,T])$. Any such pair~$(X,Y^{-}$) has no
common points of discontinuity on the same side of a point and the Young
integral of $Y^{-}$ against $X$,%
\begin{equation*}
\int_{0}^{T}Y^{-}dX\equiv \int_{0}^{T}Y_{r}^{-}dX_{r}\equiv
\int_{0}^{T}Y_{s-}dX_{s},
\end{equation*}
is well-defined (see below) provided $1/p+1/q>1$ (or $p<2$, in case $p=q$).
We need

\begin{definition}
\label{def:RRSMRS} Assume $S=S\left( \mathcal{P}\right) $ is defined on the
partitions of $\left[ 0,T\right] $ and takes values in some normed space.%
\newline
(i) Convergence in \textit{Refinement Riemann--Stieltjes} (RRS) sense: we
say (RRS) $\lim_{\left\vert \mathcal{P}\right\vert \rightarrow 0}S\left( 
\mathcal{P}\right) =L$ if for every $\varepsilon >0$ there exists $\mathcal{P%
}_{0}$ such that for every \textquotedblleft refinement" $\mathcal{P}\supset 
\mathcal{P}_{0}$ one has $\left\vert S\left( \mathcal{P}\right)
-L\right\vert <\varepsilon .$\newline
(ii) Convergence in \textit{Mesh Riemann--Stieltjes} (MRS) sense: we say
(MRS) $\lim_{\left\vert \mathcal{P}\right\vert \rightarrow 0}S\left( 
\mathcal{P}\right) =L$ if for every $\varepsilon >0$ there exists $\delta >0$
s.t. $\forall \mathcal{P}$ with mesh $\left\vert \mathcal{P}\right\vert
<\delta $, one has $\left\vert S\left( \mathcal{P}\right) -L\right\vert
<\varepsilon $.
\end{definition}

\begin{theorem}[Young]
If $X\in W^{p}$ and $Y\in W^{q}$ with $\frac{1}{p}+\frac{1}{q}>1$ , then the
Young integral is given by 
\begin{equation}
\int_{0}^{T}Y^{-}dX:=\lim_{\left\vert \mathcal{P}\right\vert \rightarrow
0}\sum_{\left[ s,t\right] \in \mathcal{P}}Y_{s}^{-}X_{s,t}=\lim_{\left\vert 
\mathcal{P}\right\vert \rightarrow 0}\sum_{\left[ s,t\right] \in \mathcal{P}%
}Y_{s}X_{s,t}  \label{equ:YoungIntegralRS}
\end{equation}%
where both limit exist in (RRS) sense. Moreover, Young's inequality holds in
either form%
\begin{eqnarray}
\left\vert \int_{s}^{t}Y^{-}dX-Y_{s}^{-}X_{s,t}\right\vert &\lesssim
&\left\Vert Y^{-}\right\Vert _{q\text{-var;}\left[ s,t\right] }\left\Vert
X\right\Vert _{p\text{-var;}\left[ s,t\right] },  \label{equ:Y1} \\
\left\vert \int_{s}^{t}Y^{-}dX-Y_{s}X_{s,t}\right\vert &\lesssim &\left\Vert
Y\right\Vert _{q\text{-var;}\left[ s,t\right] }\left\Vert X\right\Vert _{p%
\text{-var;}\left[ s,t\right] }.  \label{equ:Y2}
\end{eqnarray}%
At last, if $X,Y$ are continuous (so that in particular $Y^{-}\equiv Y$),
the defining limit of the Young integral exists in (MRS) sense.
\end{theorem}

Everything is well-known here, although we could not find the equality of
the limits in (\ref{equ:YoungIntegralRS}) pointed out explicitly in the
literature. The reader can find the proof in Proposition \ref%
{youngleft=right} below.

\subsection{``General" It\^o stochastic integration}

\cite{Jacod,RogersWilliams,Protter}

Subject to the usual conditions, any semimartingale $X=X\left( \omega
\right) $ may (and will) be taken with c\'{a}dl\'{a}g sample paths. A
classical result of Monroe allows to write any (real-valued) martingale as a
time-change of Brownian motion. As an easy consequence, semimartingales
inherit a.s. finite $2^{+}$ variation of sample paths from Brownian sample
paths. See \cite{lepingle} for much more in this direction, notably a
quantification of $\left\Vert X\right\Vert _{p\text{-var;}\left[ 0,T\right]
} $ for any $p>2$ in terms of a BDG\ inequality. Let now $Y$ be another (c%
\'{a}dl\'{a}g) semimartingale, so that $Y^-$ is previsible. The It\^o
integral of $Y^-$ against $X$ is then well-defined, and one has the
following classical Riemann--Stieltjes type description,

\begin{theorem}[It\^o]
The It\^{o} integral of $Y^{-}$ against $X$ has the presentation, with $%
t_{i}^{n}=\frac{i\,T}{2^{n}}$, \textbf{\ } 
\begin{equation}  \label{equ:ItoIntegralRS}
\int_{0}^{T}Y^{-}dX=\lim_{n}%
\sum_{i}Y_{t_{i-1}^{n}}^{-}X_{t_{i-1}^{n},t_{i}^{n}}=\lim_{n}%
\sum_{i}Y_{t_{i-1}^{n}}X_{t_{i-1}^{n},t_{i}^{n}}
\end{equation}%
where the limits exists in probability, uniformly in $T$ over compacts.
\end{theorem}

Again, this is well-known but perhaps the equality of the limits in (\ref%
{equ:ItoIntegralRS}) which the reader can find in Protter \cite[Chapter 2,
Theorem 21]{Protter}.

\subsection{Marcus canonical integration}

\label{sec:Marcus}

\cite{Marcus1,Marcus2,kpp,applebaum} Real (classical) particles do not jump,
but may move at extreme speed. In this spirit, transform $X\in W^{p}\left( %
\left[ 0,T\right] \right) $ into $\tilde{X}\in C^{p\text{-var}}([0,\tilde{T}%
])$, by "streching" time whenever%
\begin{equation*}
X_{t}-X_{t-}\equiv \Delta _{t}X\neq 0,
\end{equation*}%
followed by replacing the jump by a straight line connecting $X_{s-}$ with $%
X_{s}$, say%
\begin{equation*}
\left[ 0,1\right] \ni \theta \mapsto X_{t-}+\theta \Delta _{t}X.
\end{equation*}%
Implemented in a (c\'{a}dl\'{a}g) semimartingale context , this leads to
Marcus integration 
\begin{eqnarray*}
\int_{0}^{T}f\left( X\right) \diamond dX:= &&\int_{0}^{T}f\left(
X_{t-}\right) dX_{t}+\frac{1}{2}\int_{0}^{T}Df\left( X_{t-}\right)
d[X,X]_{t}^{c} \\
&&+\sum_{t\in (0,T]}\Delta _{t}X\left\{ \int_{0}^{1}f\left( X_{t-}+\theta
\Delta _{t}X\right) -f\left( X_{t-}\right) \right\} d\theta .
\end{eqnarray*}%
(A Young canonical integral, providied $p<2$ and $f\in C^1$, is defined
similarly, it suffices to omit the continuous quadratic variation term.) A
useful consequence, for $f \in C^3(\mathbb{R}^d)$, say, is the chain-rule 
\begin{equation*}
\int_{0}^{t} \partial_i f(X)\diamond X^i = f(X_t) - f(X_0).
\end{equation*}


It is also possible to implement this idea in the context of SDE's, 
\begin{equation}
dZ_{t}=f(Z_{t})\diamond dX_{t}  \label{Marcus SDE}
\end{equation}%
for $f:\mathbb{R}^{d}\rightarrow \mathbb{R}^{d\times k}$ where $X$ is a
semi-martingale, \cite{kpp}. The precise meaning of this Marcus canoncial
equation is given by 
\begin{align*}
Z_{t}=& Z_{0}+\int_{0}^{t}f(Z_{s-})dX_{s}+\frac{1}{2}\int_{0}^{t}f^{\prime
}f(Z_{s})d[X,X]_{s}^{c} \\
& \qquad \qquad +\sum_{0<s\leq t}\{\phi (f\Delta
X_{s},Z_{s-})-Z_{s-}-f(Z_{s-})\Delta X_{s}\} \\
& =Z_{0}+\int_{0}^{t}f(Z_{s-})dX_{s}+\frac{1}{2}\int_{0}^{t}f^{\prime
}f(Z_{s})d[X,X]_{s} \\
& +\sum_{0<s\leq t}\{\phi (f\Delta X_{s},Z_{s-})-Z_{s-}-f(Z_{s-})\Delta
X_{s}-f^{\prime }f(Z_{s})\frac{1}{2}\left( \Delta X_{s}\right) ^{\otimes 2}\}
\end{align*}%
where $\phi (g,x)$ is the time $1$ solution to $\dot{y}=g(y),\hspace{2mm}%
y(0)=x$. As one would expected fromt the afore-mentioned (first order)
chain-rule, such SDEs respect the geometry.

\begin{theorem}[\protect\cite{kpp}]
\label{manifoldstability} If $X$ is a c\'{a}dl\'{a}g semi-martingale and $f$
and $f^{\prime }f$ are globally Lipchitz, then solution to the Marcus
canoncial SDE \eqref{Marcus SDE} exists uniquely and it is a c\'{a}dl\'{a}g
semimartingale. Also, if $M$ is manifold without boundary embedded in $%
\mathbb{R}^{d}$ and $\{f_{i}(x):x\in M\}_{1\leq i\leq k}$ are vector fields
on $M$, then 
\begin{equation*}
\mathbb{P}(Z_{0}\in M)=1\implies \mathbb{P}(Z_{t}\in M\hspace{2mm}\forall
t\geq 0)=1.
\end{equation*}
\end{theorem}

\subsection{``Continuous" rough integration}

\cite{Lyons, Max, FrizHairer}

\label{sec:CRI}

Young integration of (continuous) paths has been the inspiration for the
(continuous) rough integration, elements of which we now recall. Consider $%
p\in \lbrack 2,3)$ and $\mathbf{X}=\left( X,\mathbb{X}\right) \in \mathcal{C}%
^{p\text{-var}}([0,T])$ which in notation of \cite{FrizHairer} means
validity of \textit{Chen's relation} 
\begin{equation}
\mathbb{X}_{s,u}=\mathbb{X}_{s,t}+\mathbb{X}_{t,u}+X_{s,t}\otimes X_{t,u}
\label{Chen}
\end{equation}%
and $\left\Vert \mathbf{X}\right\Vert _{p\text{-var}}:=$ $\left\Vert
X\right\Vert _{p\text{-var}}+\Vert \mathbb{X}\Vert _{p/2\text{-var}%
}^{1/2}<\infty $, where 
\begin{equation*}
\Vert \mathbb{X}\Vert _{p/2\text{-var}}:=\left( \sup_{\mathcal{P}}\sum_{%
\left[ s,t\right] \in \mathcal{P}}\left\vert \mathbb{X}_{s,t}\right\vert
^{p/2}\right) ^{2/p} .
\end{equation*}

For nice enough $F$ (e.g. $F \in C^{2}$ ), both $Y_s:=F(X_s)$ and $Y^{\prime
}:=DF(X_s)$ are in $C^{p\text{-var}}$ and we have%
\begin{equation}
\Vert R\Vert _{p/2\text{-var}}=\left( \sup_{\mathcal{P}}\sum_{\left[ s,t%
\right] \in \mathcal{P}}\left\vert R_{s,t}\right\vert ^{p/2}\right)
^{2/p}<\infty \text{ where }R_{s,t}:=Y_{s,t}-Y_{s}^{\prime }X_{s,t}.
\label{CRP}
\end{equation}

\begin{theorem}[Lyons, Gubinelli]
Write $\mathcal{P}$ for finite partitions of $\left[ 0,T\right] $. Then%
\begin{equation*}
\exists \lim_{\left\vert \mathcal{P}\right\vert \rightarrow 0}\sum_{\left[
s,t\right] \in \mathcal{P}}Y_{s}X_{s,t}+Y_{s}^{\prime }\mathbb{X}%
_{s,t}=:\int_{0}^{T}Yd\mathbf{X}
\end{equation*}
where the limit exists in (MRS) sense, cf. Definition \ref{def:RRSMRS}.
\end{theorem}

Rough integration extends immediately to the integration of so-called 
\textit{controlled rough paths}, that is, pairs $(Y,Y^{\prime })$ subject to
(\ref{CRP}). 
This gives meaning to a rough differential equation (RDE) 
\begin{equation*}
dY=f\left( Y\right) d\mathbf{X}
\end{equation*}%
provided $f \in C^2$, say: A solution is simply as a path $Y$ such that $%
(Y,Y^\prime ) := (Y,f(Y))$ satisfies (\ref{CRP}) and such that the above RDE
is satisfied in the (well-defined!) integral sense, i.e. for all $t\in [0,T]$%
, 
\begin{equation*}
Y_t - Y_0 =\int_0^ t f\left( Y\right) d\mathbf{X}.
\end{equation*}

\subsection{Geometric rough paths and signatures}

\label{sec:gRP}

\cite{Lyons, LyonsQian, LyonsStFlour, FrizVictoir}

A \textit{geometric} rough path $\mathbf{X=}\left( X,\mathbb{X}\right) $ is
rough path with $\mathrm{Sym}\left( \mathbb{X}_{s,t}\right) =\frac{1}{2}%
X_{s,t}\otimes X_{s,t}$ ; 
and we write $\mathbf{X}=\left( X,\mathbb{X}\right) \in \mathcal{C}_{\mathrm{%
g}}^{p\text{-var}}([0,T])$ accordingly. We work with generalized increments
of the form $\mathbf{X}_{s,t}=\left( X_{s,t},\mathbb{X}_{s,t}\right) $ where
we write $X_{s,t}=X_{t}-X_{s}$ for the path increment, while second order
increments $\mathbb{X}_{s,t}$ are determined from $(\mathbf{X}_{0,t})$ by
Chen's relation 
\begin{equation*}
\mathbb{X}_{0,s}+X_{0,s}\otimes X_{s,t}+\mathbb{X}_{s,t}=\mathbb{X}_{0,t}.
\end{equation*}

Behind all this is the picture that $\mathbf{X}_{0,t} := \left( 1,X_{0,t},%
\mathbb{X}_{0,t}\right) $ takes values in a Lie group $T_{1}^{\left(
2\right) }\left( \mathbb{R}^{d}\right) \equiv \left\{ 1\right\} \oplus 
\mathbb{R}^{d}\oplus \mathbb{R}^{d\times d}$, embedded in the (truncated)\
tensor algebra $T^{\left( 2\right) }\left( \mathbb{R}^{d}\right) $, and $%
\mathbf{X}_{s,t}=\mathbf{X}_{0,s}^{-1}\otimes \mathbf{X}_{0,t}$. From the
usual power series in this tensor algebra one defines, for $a+b\in \mathbb{R}%
^{d}\oplus \mathbb{R}^{d\times d}$, 
\begin{eqnarray*}
\log \left( 1+a+b\right) &=&a+b-\frac{1}{2}a\otimes a, \\
\exp \left( a+b\right) &=&1+a+b+\frac{1}{2}a\otimes a.
\end{eqnarray*}%
The linear space $\mathfrak{g}^{(2)}(\mathbb{R}^{d})=\mathbb{R}^{d}\oplus 
\mathfrak{so}\left( d\right) $ is a Lie algebra under%
\begin{equation*}
\left[ a+b,a^{\prime }+b^{\prime }\right] =a\otimes a^{\prime }-a^{\prime
}\otimes a;
\end{equation*}%
its exponential image $G^{(2)}(\mathbb{R}^{d}):=\exp \left( \mathfrak{g}%
^{(2)}(\mathbb{R}^{d})\right) $ is then a Lie (sub) group under 
\begin{equation*}
\left( 1,a,b\right) \otimes \left( 1,a^{\prime },b^{\prime }\right) =\left(
1,a+a^{\prime },b+a\otimes a^{\prime }+b^{\prime }\right) .
\end{equation*}%
At last we recall that $G^{(2)}(\mathbb{R}^{d})$ admits a so-called
Carnot--Caratheodory norm (abbreviated as CC norm henceforth), with infimum
taken over all curves $\gamma :\left[ 0,1\right] \rightarrow \mathbb{R}^{d}$
of finite length $L$,%
\begin{eqnarray*}
\left\Vert1+ a+b\right\Vert _{CC} &:&=\inf \left( L\left( \gamma \right)
:\gamma _{1}-\gamma _{0}=a,\int_{0}^{1}\left( \gamma _{t}-\gamma _{0}\right)
\otimes d\gamma _{t}=b\right) \\
&\asymp &\left\vert a\right\vert +\left\vert b\right\vert ^{1/2} \\
&\asymp &\left\vert a\right\vert +\left\vert \mathrm{Anti}\left( b\right)
\right\vert ^{1/2}.
\end{eqnarray*}%
A left-invariant distance is induced by the group structure,%
\begin{equation*}
d_{CC}\left( g,h\right) =\left\Vert g^{-1}\otimes h\right\Vert _{CC}
\end{equation*}%
which turns $G^{(2)}(\mathbb{R}^{d})$ into a Polish space. Geometric rough
paths with roughness parameter $p\in \lbrack 2,3)$ are \textit{precisely }%
classical paths of finite $p$-variation with values in this metric space.

\begin{proposition}
$\mathbf{X}=\left( X,\mathbb{X}\right) \in \mathcal{C}_{\mathrm{g}}^{p\text{%
-var}}([0,T])$ iff $\mathbf{X}=\left( 1,X,\mathbb{X}\right) \in C^{p\text{%
-var}}\left( \left[ 0,T\right] ,G^{(2)}(\mathbb{R}^{d})\right) $. Moreover, 
\begin{equation*}
\left\Vert \mathbf{X}\right\Vert _{p\text{-var}}\asymp \left( \sup_{\mathcal{%
P}}\sum_{\left[ s,t\right] \in P}\left\Vert \mathbf{X}_{s,t}\right\Vert
_{CC}^{p}\right) ^{1/p}.
\end{equation*}
\end{proposition}

The theory of geometric rough paths extends to all $p\geq 1$, and a
geometric p-rough path is a path with values in $G^{(\left[ p\right])}\left( 
\mathbb{R}^{d}\right) $, the step-$\left[ p\right] $ nilpotent Lie group
with $d$ generators, embedded in $T^{\left( \left[ p\right] \right) }\left( 
\mathbb{R}^{d}\right) $, where%
\begin{equation*}
T^{(m)}\left( \mathbb{R}^{d}\right) =\oplus _{k=0}^{m}\left( \mathbb{R}%
^{d}\right) ^{\otimes k}\subset \oplus _{k=0}^{\infty }\left( \mathbb{R}%
^{d}\right) ^{\otimes k}\subset T((\mathbb{R}^{d}))
\end{equation*}%
(the last inclusion is strict, think polyomials versus power-series) and
again of finite $p$-variation with respect to the Carnot--Caratheodory
distance (now defined on $G^{(\left[ p\right])}$).

\begin{theorem}[Lyons' extension]
\label{Lyonsfirsttheorem} Let $1\leq m:=\left[ p\right] \leq p\leq N<\infty $%
. A (continuous) geometric rough path $\mathbf{X}^{\left( m\right) }\in C^{p%
\text{-var}}\left( \left[ 0,T\right] ,G^{(m)}\right) $ admits an extension
to a path $\mathbf{X}^{\left( N\right) }$ with values $G^{(N)}\subset
T^{(N)} $, unique in the class of $G^{(N)}$-valued path starting from $1$
and of finite $p $-variation with respect to CC metric on $G^{(N)}$. In
fact, 
\begin{equation*}
\left\Vert \mathbf{X}^{\left( N\right) }\right\Vert _{p\text{-var;}\left[ s,t%
\right] }\lesssim \left\Vert \mathbf{X}^{\left( m\right) }\right\Vert _{p%
\text{-var;}\left[ s,t\right] }.
\end{equation*}
\end{theorem}

\begin{remark}
In view of this theorem, any $\mathbf{X}\in C^{p\text{-var}}\left( \left[ 0,T%
\right] ,G^{(m)}\right) $ may be regarded as $\mathbf{X}\in C^{p\text{-var}%
}\left( \left[ 0,T\right] ,G^{(N)}\right) $, any $N\geq m$, and there is no
ambiguity in this notation.
\end{remark}

\begin{definition}
Write $\pi _{(N)}$ resp. $\pi _{M}$ for the projection $T((\mathbb{R}^{d}))%
\rightarrow T^{(N)}\left( \mathbb{R}^{d}\right) $ resp. $\left( \mathbb{R}%
^{d}\right) ^{\otimes M}$. Call $g\in T((\mathbb{R}^{d}))$ group-like, if $\pi
_{(N)}\left( g\right) \in $ $G^{(N)}$ for all $N$. Consider a geometric
rough path $\mathbf{X}\in C^{p\text{-var}}\left( \left[ 0,T\right] ,G^{\left[
p\right] }\right) $. Then, thanks to the extension theorem, 
\begin{equation*}
S\left( \mathbf{X}\right) _{0,T}:=\left( 1,\pi _{1}\left( \mathbf{X}\right)
,\dots ,\pi _{m}\left( \mathbf{X}\right) ,\pi _{m+1}\left( \mathbf{X}\right)
,\dots \right) \in T((\mathbb{R}^{d}))
\end{equation*}%
defines a group-like element, called the signature of $\mathbf{X}$.
\end{definition}

The signature solves a rough differential equation (RDE, ODE if $p=1$) in
the tensor-algebra,%
\begin{equation}
d {S}= {S}\otimes d \mathbf{X},\, {S}_{0}=1.  \label{dSequals}
\end{equation}

To a significant extent, the signature determines the underlying path $X$,
if of bounded variation, cf. \cite{HL10}. (The rough path case was recently
obtained in \cite{2014arXiv1406.7871B}). A basic, yet immensely useful fact
is that multiplication in $T((\mathbb{R}^{d}))$, if restricted to group-like
elements, can be linearized.

\begin{proposition}
{(Shuffle product formula)} Consider two multiindices $v=\left( i_{1},\dots
,i_{m}\right) ,w=\left( j_{1},\dots ,j_{n}\right) $%
\begin{equation*}
\mathbf{X}^{v}\mathbf{X}^{w}=\sum \mathbf{X}^{z}
\end{equation*}%
where the (finite) sum runs over all shuffles $z$ of $v,w$.
\end{proposition}

\subsection{Checking $p$-variation}

\cite{LLQ02, FV06var, Manstavicius}

\label{sec:Man}

(i) As in whenever $\gamma >p-1>0$, with $t_{k}^{n}=k2^{-n}T$, one has%
\begin{equation}
\left\Vert X\right\Vert _{p\text{-var;}\left[ 0,T\right] }^{p}\lesssim
\sum_{n=1}^{\infty }n^{\gamma }\sum_{k=1}^{2^{n}}\left\vert
X_{t_{k}^{n}}-X_{t_{k-1}^{n}}\right\vert ^{p}.  \label{PVarEstimateLLQ}
\end{equation}%
This estimate has been used \cite{LLQ02} to verify finite (sample path) $p$%
-variation, simply by taking expectation, e.g. for the case of Brownian
motion by using that $\mathbb{E}\left[ \left\vert B_{s,t}\right\vert ^{p}%
\right] =|t-s|^{1+\epsilon }$ for $\epsilon >0$, provided $p>2$.
Unfortunately, this argument does not work for jump processes. Even for the
standard Poisson process one only has $\mathbb{E}\left[ \left\vert
N_{s,t}\right\vert ^{p}\right] \sim C_{p}\left\vert t-s\right\vert $ as $%
t-s\rightarrow 0$, so that the \textit{expected value} of the right-hande
side of (\ref{PVarEstimateLLQ}) is infinity. 
An extension of (\ref{PVarEstimateLLQ}) to rough path is%
\begin{equation*}
\left\Vert \mathbf{X}\right\Vert _{p\text{-var;}\left[ 0,T\right]
}^{p}\lesssim \sum_{n=1}^{\infty }n^{\gamma }\sum_{k=1}^{2^{n}}\left\{
\left\vert X_{t_{k-1}^{n},t_{k}^{n}}\right\vert ^{p}+\left\vert \mathbb{X}%
_{t_{k-1}^{n},t_{k}^{n}}\right\vert ^{p/2}\right\}
\end{equation*}%
and we note that for a \textit{geometric} rough path $\mathbf{X=}\left( X,%
\mathbb{X}\right) $, i.e. when $\mathrm{Sym}\left( \mathbb{X}_{s,t}\right) =%
\frac{1}{2}X_{s,t}\otimes X_{s,t}$, we may replace $\mathbb{X}$ on the
right-hand side by the area $\mathbb{A}=\mathrm{Anti}\left( \mathbb{X}%
\right) $. This has been used in \cite{LyonsQian}, again by taking
expecations, to show that Brownian motion $B$ enhanced with $\mathbb{B}%
_{s,t}:=\int_{s}^{t}B_{s,r}\otimes \circ dB_{r}$ constitutes a.s. an element
in the rough path space $\mathcal{C}^{p\text{-var}}([0,T],G^{(2)})$, for $%
p\in \left( 2,3\right) $.

(ii) In \cite{FV06var} an embedding result $W^{\delta ,q}\hookrightarrow C^{p%
\text{-var}}$ is shown, more precisely%
\begin{equation*}
\left\Vert X\right\Vert _{p\text{-var;}\left[ 0,T\right] }^{q}\lesssim
\int_{0}^{T}\int_{0}^{T}\frac{\left\vert X_{t}-X_{s}\right\vert ^{q}}{%
\left\vert t-s\right\vert ^{1+\delta q}}dsdt,
\end{equation*}%
provided $1<p<q<\infty $ with $\delta =1/p\in (0,1)$. The extension to rough
paths reads%
\begin{equation*}
\left\Vert \mathbf{X}\right\Vert _{p\text{-var;}\left[ 0,T\right]
}^{q}\lesssim \int_{0}^{T}\int_{0}^{T}\left\{ \frac{\left\vert
X_{s,t}\right\vert ^{q}}{\left\vert t-s\right\vert ^{1+\delta q}}+\frac{%
\left\vert \mathbb{X}_{s,t}\right\vert ^{q/2}}{\left\vert t-s\right\vert
^{1+\delta q}}\right\} dsdt.
\end{equation*}
Since elements in $W^{\delta ,q}$ are also $\alpha $-H\"{o}lder, with $%
\alpha =\delta -1/q>0$, these embeddings are not suitable for non-continuous
paths.

(iii) In case of a strong Markov process $X$ with values in some Polish
space $\left( E,d\right) $, a powerful criterion has been established by
Manstavicius \cite{Manstavicius}. Define%
\begin{equation*}
\alpha \left( h,a\right) :=\sup \left\{ \mathbb{P}\left( d\left(
X_{t}^{s,x},x\right) \geq a\right) \right\}
\end{equation*}%
with $\sup $ taken over all $x\in E$, and $s<t$ in$\left[ 0,T\right] $ with $%
t-s\leq h$. Under the assumption%
\begin{equation*}
\alpha \left( h,a\right) \lesssim \frac{h^{\beta }}{a^{\gamma }},
\end{equation*}%
uniformly for $h,a$ in a right neighbourhood of zero, the process $X$ has
finite $p$-variation for any $p>\gamma /\beta $. In the above Poisson
example, noting $\mathbb{E}\left[ \left\vert N_{s,t}\right\vert \right]
=O\left( h\right) $ whenever $t-s\leq h$, \ Chebychev inequality immediately
gives $\alpha \left( h,a\right) \leq h/a$, and we find finite $p$-variation,
any $p>1$. (Of course $p=1$ here, but one should not expect this borderline
case from a general criterion.) The Manstavicius criterion will play an
important role for us.

\subsection{Expected signatures}

\cite{fawcett, HL10, LeJQ13, chevyrev}

Recall that for a smooth path $X:\left[ 0,T\right] \rightarrow \mathbb{R}%
^{d} $, its signature $\mathbf{S=S}\left( X\right) $ is given by the
group-like element

\begin{equation*}
\left( 1,\int_{0<t_{1}<T}dX_{r_1},\int_{0<t_{1}<t_{2}<T}dX_{r_1}\otimes
dX_{r_2},\dots \right) \in T((\mathbb{R}^{d})).
\end{equation*}%
The signature solves an ODE in the tensor-algebra,%
\begin{equation}
d\mathbf{S}=\mathbf{S}\otimes dX,\,\mathbf{S}_{0}=\mathbf{1.}
\label{dSequals}
\end{equation}%
Generalizations to semimartingales are immediate, by interpretation of (\ref%
{dSequals}) as It\^{o}, Stratonovich or Marcus stochastic differential
equation. In the same spirit $X$ can be replaced by a generic (continuous)
geometric \ rough path with the according interpretation of (\ref{dSequals})
as (linear) rough differential equation.

Whenever $X=X\left( \omega \right) $, or $\mathbf{X=X}\left( \omega \right) $
is granted sufficient integrability, we may consider the expected signature,
that is%
\begin{equation*}
\mathbb{E}\mathbf{S}_{T}\in T((\mathbb{R}^{d}))
\end{equation*}%
defined in the obvious componentwise fashion. To a significant extent, this
object behaves like a moment generating function. In a recent work \cite%
{chevyrev}, it is shown that under some mild condition, the expected
signature determines the law of the $S_T(\omega)$.

\subsection{L\'{e}vy Processes}

\label{sec:Levy}

\cite{sato, bertoin, applebaum, Hunt}

Recall that a $d$-dimensional L\'{e}vy process $\left( X_{t}\right) $ is a
stochastically continuous process such that (i) for all $0<s<t<\infty $, the
law of $X_{t}-X_{s}$ depends only on $t-s$; (ii) for all $t_{1},\dots ,t_{k}$
such that $0<t_{1}<\dots <t_{k}$ the random variables $X_{t_{i+1}}-X_{t_{i}}$
are independent. L\'{e}vy process can (and will) be taken with c\'{a}dl\'{a}%
g sample paths and are characterized by the L\'{e}vy triplet $\left(
a,b,K\right) $, where $a=(a^{i,j})$ is a positive semidefinite symmetric
matrix, $b=(b^{i})$ a vector and $K(dx)$ a L\'{e}vy measure on $\mathbb{R}%
^{d}$ (no mass at $0$, integrates $\min (|x|^{2},1)$) so that 
\begin{equation}
\mathbb{E}\left[ e^{i\left\langle u,X_{t}\right\rangle }\right] =\exp \left(
-\frac{1}{2}\left\langle u,au\right\rangle +i\left\langle u,b\right\rangle
+\int_{\mathbb{R}^{d}}(e^{iu\,y}-1-iu\,y1_{\left\{ \left\vert y\right\vert
<1\right\} })K\left( dy\right) \right) .  \label{LKclassic}
\end{equation}%
The It\^{o}--L\'{e}vy decomposition asserts that any such L\'{e}vy process
may be written as,%
\begin{equation}
X_{t}=\sigma B_{t}+bt+\int_{(0,t]\times \left\{ \left\vert y\right\vert
<1\right\} }y\tilde{N}\left( ds,dy\right) +\int_{(0,t]\times \left\{
\left\vert y\right\vert \geq 1\right\} }yN\left( ds,dy\right)
\label{ItoLevyDecomp}
\end{equation}%
where $B$ is a $d$-dimensional Browbian motion, $\sigma \sigma ^{T}=a$, and $%
N$ (resp. $\tilde{N}$) is the Poisson random meausre (resp. compensated PRM)
with intensity $ds\,K\left( dy\right) $. A Markovian description of a L\'{e}%
vy process is given in terms of its generator 
\begin{equation}
\left( \mathcal{L}f\right) \left( x\right) =\frac{1}{2}%
\sum_{i,j=1}^{d}a^{i,j}\partial _{i}\partial
_{j}f+\sum_{i=1}^{d}b^{i}\partial _{i}f+\int_{\mathbb{R}^{d}}\left( f\left(
x+y\right) -f\left( x\right) -1_{\left\{ \left\vert y\right\vert <1\right\}
}\sum_{i=1}^{d}y^{i}\partial _{i}f\right) K\left( dy\right) .
\label{LKclassicMarkov}
\end{equation}%
By a classical result of Hunt \cite{Hunt}, this characterization extends to L%
\'{e}vy process with values in a Lie group $G$, defined as above, but with $%
X_{t}-X_{s}$ replaced by $X_{s}^{-1}X_{t}$. Let $\left\{ \mathfrak{u}%
_{1},\dots ,\mathfrak{u}_{m}\right\} $ be a basis of the Lie algebra $%
\mathfrak{g}$, thought of a left-invariant first order differential
operators. In the special case of exponential Lie groups, meaning that $\exp
:\mathfrak{g}\rightarrow G$ is an analytical diffeomorphism (so that $g=\exp
\left( x^{i}\mathfrak{u}_{i}\right) $ for all $g\in G$, with canonical
coordinates $x^{i}=x^{i}\left( g\right) $ of the first kind) the generator
reads%
\begin{equation}
\left( \mathcal{L}f\right) \left( x\right) =\frac{1}{2}%
\sum_{v,w=1}^{m}a^{v,w}\mathfrak{u}_{v}\mathfrak{u}_{w}f+\sum_{v=1}^{m}b^{v}%
\mathfrak{u}_{v}f+\int_{G}\left( f\left( xy\right) -f\left( x\right)
-1_{\left\{ \left\vert y\right\vert <1\right\} }\sum_{v=1}^{m}y^{v}\mathfrak{%
u}_{v}f\right) K\left( dy\right) .  \label{LKLieMarkov}
\end{equation}%
As before the L\'{e}vy triplet $\left( a,b,K\right) $ consists of $\left(
a^{v,w}\right) $ positive semidefinite symmetric, $b=(b^{v})$ and $K(dx)$ a L%
\'{e}vy measure on $G$ (no mass at the unit element, integrates $\min
(|x|^{2},1)$, with $\left\vert x\right\vert ^{2}:=\sum_{i=v}^{m}\left(
x^{v}\right) ^{2}$.)

\subsection{The work of D. Williams}

\label{sec:williamsintro}

\cite{williams}

Williams first considers the Young regime $p\in \lbrack 1,2)$ and shows that
every $X\in W^{p}\left( [0,T]\right) $ may be turned into $\tilde{X}\in C^{p%
\text{-var}} ( [0,\tilde{T}] ) $, by replacing jumps by segments of straight
lines (in the spirit of Marcus canonical equations, via some time change $%
\left[ 0,T\right] \rightarrow [ 0,\tilde{T} ] $) Crucially, this can be done
with a uniform estimate $||\tilde{X}||_{p\text{-var}}\lesssim ||X||_{p\text{%
-var}}$. In the rough regime $p\geq 2$, Williams considers a generic $d$%
-dimensional L\'{e}vy process $X$ enhanced with stochastic area%
\begin{equation*}
\mathbb{A}_{s,t}:=\mathrm{Anti}\int_{(s,t]}(X_{r}^{-}-X_{s})\otimes dX_{r}
\end{equation*}%
where the stochastic integration is understood in It\^{o}-sense. On a
technical level his main results \cite[p310-320]{williams} are summarized in

\begin{theorem}[Williams]
\label{theo:will} Assume $X$ is a $d$-dimensional L\'{e}vy process $X$ with
triplet $\left( a,b,K\right) $. \newline
(i) Assume $K$ has compact support. Then%
\begin{equation*}
\mathbf{E}\left[ |\mathbb{A}_{s,t}|^{2}\right] \lesssim \left\vert
t-s\right\vert ^{2}.
\end{equation*}%
(ii)\ For any $p>2$, with $\sup $ taken over all partitions of $\left[ 0,T%
\right] ,$%
\begin{equation*}
\sup_{\mathcal{P}}\sum_{\left[ s,t\right] \in \mathcal{P}}\left\vert \mathbb{%
A}_{s,t}\right\vert ^{p/2}<\infty \text{ a.s.}
\end{equation*}
\end{theorem}

Clearly, $(X,\mathbb{A)}\left( \omega \right) $ is all the information one
needs to have a (in our terminology) cadlag geometric $p$-rough path $%
\mathbf{X=X}\left( \omega \right) $, any $p\in (2,3)$. However, Williams
does not discuss rough integration, nor does he give meaning (in the sense
of an integral equation) to a rough differential equations driven by cadlag $%
p$-rough paths. Instead he constructs, again in the spirit of Marcus, $%
\mathbf{\tilde{X}}\in \mathcal{C}^{p\text{-var}}( [0,\tilde{T}]) $, and then
goes on to \textit{define} a solution $Y$ to an RDE driven by $\mathbf{X}%
\left( \omega \right) $ as reverse-time change of a (classical) RDE solution
driven by the (continuous) geometric $p$-rough path $\mathbf{\tilde{X}}$.
While this construction is of appealing simplicity, the time-change depends
in a complicated way on the jumps of $X\left( \omega \right) $ and the
absence of quantitative estimates, makes any local analysis of so-defined
RDE solution difficult (starting with the identification of $Y$ as solution
to the corresponding Marcus canonical equation). We shall not rely on any of
Williams' result, although his ideas will be visible at various
places in this paper.  A simplified proof of Theorem \ref{theo:will} will be
given below.

\newpage

\part{Rough paths in presence of jumps: deterministic theory}

\section{General rough paths: definition and first examples}

The following definitions are fundamental.

\begin{definition}
\label{definitionofRP} Fix $p\in \lbrack 2,3)$. We say that $\mathbf{X}%
=\left( X,\mathbb{X}\right) $ is a \textit{general (c\'{a}dl\'{a}g) rough
path over } $\mathbb{R}^d$ if

\begin{itemize}
\item[(i)] Chen's relation holds, i.e. for all $s \leq u \leq t $, $\mathbb{X%
}_{s,t} - \mathbb{X}_{s,u} - \mathbb{X}_{u,t} = X_{s,u}\otimes X_{u,t}$;

\item[(ii)] the following map 
\begin{equation*}
\left[ 0,T\right] \ni t\mapsto \mathbf{X}_{0,t}=\left( X_{0,t},\mathbb{X}%
_{0,t}\right) \in \mathbb{R}^{d}\oplus \mathbb{R}^{d\times d}
\end{equation*}%
is c\'{a}dl\'{a}g;

\item[(iii)] $p$-variation regularity in rough path sense holds, that is 
\begin{equation*}
\left\Vert X\right\Vert _{p\text{-var;}\left[ 0,T\right] }+\Vert \mathbb{X}%
\Vert _{p/2\text{-var;}\left[ 0,T\right] }^{1/2}<\infty \text{.}
\end{equation*}
\end{itemize}

We then write%
\begin{equation*}
\mathbf{X}\in \mathcal{W}^{p}=\mathcal{W}^{p}\left( \left[ 0,T\right] ,%
\mathbb{R}^{d}\right) .
\end{equation*}
\end{definition}

\begin{definition}
We call $\mathbf{X}\in \mathcal{W}^{p}$\underline{ \textit{g}}\textit{%
eometric if it takes values in }$G^{(2)}(\mathbb{R}^{d})$, in symbols $%
\mathbf{X}\in \mathcal{W}_{\text{\textrm{g}}}^{p}$. If, in addition, 
\begin{equation*}
(\Delta _{t}X,\Delta _{t}\mathbb{A}):=\log \Delta _{t}\mathbf{X}\in \mathbb{R%
}^{d}\oplus \left\{ 0\right\} \subset \mathfrak{g}^{(2)}(\mathbb{R}^{d})
\end{equation*}%
we call $\mathbf{X}$ \ \underline{\textit{M}}\textit{arcus-like, in symbols }%
$\mathbf{X\in }$ $\mathcal{W}_{\text{\textrm{M}}}^{p}$.
\end{definition}

As in the case of (continuous) rough paths, cf. Section \ref{sec:gRP},%
\begin{equation*}
\mathcal{W}_{\text{\textrm{g}}}^{p}:=\mathcal{W}_{\text{\textrm{g}}%
}^{p}\left( \left[ 0,T\right] ,\mathbb{R}^{d}\right) =W^{p}\left(
[0,T],G^{(2)}(\mathbb{R}^{d})\right)
\end{equation*}%
so that general geometric $p$-rough paths are precisely paths of finite $p$%
-variation in $G^{(2)}(\mathbb{R}^{d})$ equipped with CC\ metric. We can
generalize the definition to general $p\in \lbrack 1,\infty )$ at the price
of working in the step-$\left[ p\right] $ free nilpotent group, 
\begin{equation*}
\mathcal{W}_{\text{\textrm{g}}}^{p}=W^{p}\left( \left[ 0,T\right]
,G^{([p])}\right) .
\end{equation*}

As a special case of Lyons' extension theorem (Theorem \ref%
{Lyonsfirsttheorem}), for a given continuous path $X\in $ $W^{p}$ for $p\in
\lbrack 1,2)$, there is a \textit{unique} rough path $\mathbf{X}=\left( X,%
\mathbb{X}\right) \in \mathcal{W}^{p}$. (Uniqueness is lost when $p\geq 2$,
as seen by the perturbation $\mathbb{\bar{X}}_{s,t}=\mathbb{X}_{s,t}+a\left(
t-s\right) $, for some matrix $a$.)

The situation is different in presence of jumps and Lyons' First Theorem
fails, even when $p=1$. Essentially, this is due to the fact that there are
non-trivial pure jump paths of finite $q$-variation with $q<1$.

\begin{proposition}[Canonical lifts of paths in Young regime]
\label{CanLiftBV}Let $X\in W^{p}\left( \left[ 0,T\right] ,\mathbb{R}%
^{d}\right) $ be a c\'{a}dl\'{a}g path of finite p-variation for $p\in
\lbrack 1,2)$. (i)\ It is lifted to a (in general, non-geometric) rough path 
$\mathbf{X}=\left( X,\mathbb{X}\right) \in $ $\mathcal{W}^{p}$ by enhancing $%
X$ with 
\begin{equation*}
\mathbb{X}_{s,t}=\mathrm{(Young)}\int_{(s,t]}X_{s,r-}\otimes dX_{r}
\end{equation*}%
(ii) It is lifted to a Marcus-like c\'{a}dl\'{a}g rough path $\mathbf{X}^{%
\mathrm{M}}=\left( X,\mathbb{X}^{\mathrm{M}}\right) \in $ $\mathcal{W}_{%
\text{M}}^{p}$ by enhancing $X$ with 
\begin{equation*}
\mathbb{X}_{s,t}^{\mathrm{M}}=\mathbb{X}_{s,t}+\frac{1}{2}\sum_{r\in
(s,t]}\left( \Delta _{r}X\right) \otimes \left( \Delta _{r}X\right) .
\end{equation*}
\end{proposition}

\begin{proof}
As an application of Young's inequality, it is easy to see that 
\begin{equation*}
|\mathbb{X}_{s,t}|\lesssim ||X||_{p\text{-var;}[s,t]}^{2}
\end{equation*}%
Note that $\omega (s,t):=||X||_{p\text{-var;}[s,t]}^{p}$ is superadditive,
i.e. for all $s<u<t$, $\omega (s,u)+\omega (u,t)\leq \omega (s,t)$, so that 
\begin{equation*}
\sum_{\lbrack s,t]\in \mathcal{P}}|\mathbb{X}_{s,t}|^{\frac{p}{2}}\lesssim
\sum_{\lbrack s,t]\in \mathcal{P}}||X||_{p\text{-var;}[s,t]}^{p}\lesssim
||X||_{p\text{-var;}[0,T]}^{p}
\end{equation*}%
Taking sup over $\mathcal{P}$, $\mathbb{X}$ has $\frac{p}{2}$ variation. %
We then note that

\begin{equation*}
\left(\sum_{r\in (s,t]}\left( \Delta _{r}X\right) \otimes \left( \Delta
_{r}X\right)\right)^{\frac{p}{2}} \leq \left(\sum_{r\in (s,t]}|\Delta _{r}X
|^2\right)^{\frac{p}{2}} \leq \sum_{r\in (s,t]}|\Delta _{r}X |^p
\end{equation*}

Since the jumps of $X$ are $p$-summable, we immediately conclude that $%
\mathbb{X}^{\mathrm{M}}$ also is of finite $\frac{p}{2}$ variation.

Also, from \textquotedblleft integration by parts formula for sums", it can
be easily checked that $Sym(\mathbb{X}_{s,t}^{\mathrm{M}})=\frac{1}{2}%
X_{s,t}\otimes X_{s,t}$. The fact that $\left( X,\mathbb{X}^{\mathrm{M}%
}\right) $ forms a Marcus-like rough path comes from the underlying idea of
the Marcus integral replaces jumps by straight lines which do not create
area. Precisely, 
\begin{equation*}
\lim_{s\uparrow t}\mathbb{X}_{s,t}^{\mathrm{M}}=:\Delta _{t}\mathbb{X}^{%
\mathrm{M}}=\frac{1}{2}\left( \Delta _{t}X\right) ^{\otimes 2}
\end{equation*}%
which is symmetric. Thus $\Delta _{t}\mathbb{A}=Anti(\Delta _{t}\mathbb{X}^{%
\mathrm{M}})$ = 0.
\end{proof}

\bigskip

Clearly, in the continuous case every geometric rough path is Marcus-like
and so there is need to distinguish them. The situation is different with
jumps and there are large classes of Marcus-like as well as non-Marcus-like
geometric rough paths. We give some examples.

\begin{example}[\textbf{Pure area jump rough path}]
\label{ex:PAJRP}Consider a $\mathfrak{so}\left( d\right) $-valued path $%
\left( A_{t}\right) $ of finite $1$-variation, started at $A_{0}=0$. Then%
\begin{equation*}
\mathbf{X}_{0,t}:=\exp \left( A_{t}\right)
\end{equation*}%
defines a geometric rough path, for any $p\geq 2$, i.e. $\mathbf{X}\left(
\omega \right) \mathbf{\in }$ $\mathcal{W}_{\text{g}}^{p}$ but, unless $A$
is continuous,%
\begin{equation*}
\mathbf{X}\left( \omega \right) \mathbf{\notin }\mathcal{W}_{\text{\textrm{M}%
}}^{p}.
\end{equation*}
\end{example}

It is not hard to randomize the above non-Marcus $M$ rough path example.

\begin{example}[\textbf{Pure area Poisson process}]
\label{PACPRP} Consider an i.i.d. sequence of a $\mathfrak{so}\left(
d\right) $-valued r.v. $\left( \mathfrak{a}^{n}\left( \omega \right) \right) 
$ and a standard Poisson process $N_{t}$ with rate $\lambda >0$. Then, with
probability one,%
\begin{equation*}
\mathbf{X}_{0,t}\left( \omega \right) :=\exp \left( \sum_{n=1}^{N_{t}}%
\mathfrak{a}^{n}\left( \omega \right) \right)
\end{equation*}%
yields a geometric, non-Marcus like c\'{a}dl\'{a}g rough path for any $p\geq
2$.
\end{example}

It is instructive to compare the last examples with the following two
classical examples from (continuous) rough path theory.

\begin{example}[\textbf{Pure area rough path}]
\label{ex:pacCPP}Fix $\mathfrak{a\in so}\left( d\right) $. Then 
\begin{equation*}
\mathbf{X}_{0,t}:=\exp \left( \mathfrak{a\,}t\right) ,\,\,\,\,
\end{equation*}%
yields a geometric rough path, $\mathbf{X}\in \mathcal{C}_{g}^{p}\left( %
\left[ 0,T\right] ,\mathbb{R}^{d}\right) $, above the trivial path $X\equiv
0 $, for any $p\in \lbrack 2,3)$.
\end{example}

\begin{example}[\textbf{Brownian rough path in magnetic field}]
\label{ex:ncBRP}Write%
\begin{equation*}
\mathbf{B}^{\mathrm{S}}_{s,t}=\left( B_{s,t},\int_s^t B_{s,r}\otimes \circ
dB_r\right)
\end{equation*}%
for the Brownian rough path based on iterated \underline{S}tratonovich
integration.\ If one considers the (zero-mass) limit of a physical Brownian
particle, with non-zero charge, in a constant \underline{m}agnetic field 
\cite{Friz--Gassiat--Lyons} one finds the (non-canonical) Brownian rough
path 
\begin{equation*}
\mathbf{B}_{0,t}^{\mathrm{m}}:=\mathbf{B}^{\mathrm{S}}{}_{0,t}+\left( 0,%
\mathfrak{a\,}t\right) ,
\end{equation*}%
for some $\mathfrak{a} \in \mathfrak{so}(d)$. This yields a continuous,
non-canonical geometric rough path lift of Brownian motion.\ More precisely, 
$\mathbf{B}^{\mathrm{m}}\in \mathcal{C}_{g}^{p}\left( \left[ 0,T\right] ,%
\mathbb{R}^{d}\right) $ a.s, for any $p\in (2,3)$.
\end{example}


As is well-known in rough path theory, it is not trivial to construct suitable $\mathbb{X}$
given some (irregular) path $X$, and most interesting constructions are of stochastic
nature. At the same time, $X$ does not determine $\mathbb{X}$, as was seen 
in the above examples.  That said, 
once in possession of a (c\'{a}dl\'{a}g) rough path, there are immediate
ways to obtain further rough paths, of which we mention in particular
perturbation of $\mathbb{X}$ by increments of some $p/2$-variation path,
and, secondly, subordination of $(X, \mathbb{X})$ by some
increasing (c\'{a}dl\'{a}g) path. For instance, in a stochastic setting, any time change of the
(canonical)  Brownian rough path, by some L\'evy subordinator for instance, will
yield a general random rough path, corresponding to the 
(c\'{a}dl\'{a}g) rough path associated to a specific semimartingale. 

For Brownian motion, as for (general) semimartingales, there are two ``canoncial"
candidates for $\mathbb{X}$, obtained by It\^o- and Marcus canonical
(=Stratonovich in absence of jumps) integration, respectively. We have

\begin{proposition}
\label{prop:semimartlift} Consider a $d$-dimensional (c\'{a}dl\'{a}g)
semimartingale $X$ and let $p\in \left( 2,3\right) $. Then the following
three statements are equivalent.\newline
(i)$\ \mathbf{X}^{\text{\textrm{I}}}\left( \omega \right) \in \mathcal{W}%
^{p} $ a.s where $\mathbf{X}^{\text{\textrm{I}}}=(X,\mathbb{X}^{\text{%
\textrm{I}}})$ and%
\begin{equation*}
\mathbb{X}_{s,t}^{\text{\textrm{I}}}:=\int_{s}^{t}X_{s,r-}\otimes dX_{r}\,\ 
\mathrm{(\underline{I}t\hat{o})}
\end{equation*}%
(ii) $\mathbf{X}^{\text{\textrm{M}}}\left( \omega \right) \in \mathcal{W}_{%
\text{\textrm{M}}}^{p}\,(\subset \mathcal{W}_{\text{\textrm{g}}}^{p})$ a.s.
where $\mathbf{X}^{\text{\textrm{M}}}=(X,\mathbb{X}^{\text{\textrm{M}}})$ and%
\begin{equation*}
\mathbb{X}_{s,t}^{\text{\textrm{M}}}:=\int_{s}^{t}X_{s,r-}\diamond \otimes
dX_{r}\,\ \mathrm{(\underline{M}arcus).}
\end{equation*}%
(iii)\ The stochastic area \ (identical for both It\^{o}- and Marcus lift)%
\begin{equation*}
\mathbb{A}_{s,t}:=\mathrm{Anti}\left( \mathbb{X}_{s,t}^{\text{\textrm{I}}%
}\right) =\mathrm{Anti}\left( \mathbb{X}_{s,t}^{\text{M}}\right)
\end{equation*}%
has a.s. finite $p/2$-variation.
\end{proposition}

\begin{proof}
Clearly%
\begin{equation*}
\mathrm{Sym}(\mathbb{X}_{s,t}^{\text{\textrm{M}}})=\frac{%
{\acute{}}%
1}{2}X_{s,t}\otimes X_{s,t}
\end{equation*}%
is of finite $p/2$-variation, a consequence of $X\in W^{p}$ a.s., for any $%
p>2$. Note that $\mathbb{X}^{\text{\textrm{M}}}-\mathbb{X}^{\text{\textrm{I}}%
}$ is symmetric, 
\begin{equation*}
(\mathbb{X}_{s,t}^{\text{\textrm{M}}})^{i,j}-(\mathbb{X}_{s,t}^{\text{%
\textrm{I}}})^{i,j}=\frac{1}{2}[X^{i},X^{j}]_{s,t}^{c}+\frac{1}{2}\sum_{r\in
(s,t]}\Delta _{r}X^{i}\Delta _{r}X^{j}.
\end{equation*}%
and is of finite $\frac{p}{2}$ variation as $[X^{i},X^{j}]^{c}$ is of
bounded variation, while

\begin{equation*}
\left\vert \sum_{r\in (s,t]}\Delta _{r}X^{i}\Delta _{r}X^{j}\right\vert^{%
\frac{p}{2}} \leq\left\vert \frac{1}{2}\sum_{r\in (s,t]}|\Delta
_{r}X|^{2}\right\vert^{\frac{p}{2}} \lesssim \sum_{r \in (s,t]}
|\Delta_rX|^p <\infty \text{ a.s.}
\end{equation*}
because jumps of semimartingale is square summable and thus $p \geq2$
summable.
\end{proof}

We now given an elegant criterion which allows to check finite $2^+$%
-variation of $G^{(2)}$-valued processes.

\begin{proposition}
\label{prop:SMRP} Consider a $G^{(2)}(\mathbb{R}^{d})$-valued strong Markov
process $\mathbf{X}_{s,t}:=\mathbf{X}_{s}^{-1}\otimes \mathbf{X}_{t}=\exp
\left( X_{s,t},\mathbb{A}_{s,t}\right) $. Assume%
\begin{eqnarray*}
\mathbb{E}\left\vert X_{s,t}\right\vert ^{2} &\lesssim &\left\vert
t-s\right\vert , \\
\mathbb{E}\left\vert \mathbb{A}_{s,t}\right\vert ^{2} &\lesssim &\left\vert
t-s\right\vert ^{2},
\end{eqnarray*}%
uniformly in $s,t\in \left[ 0,T\right] $. Then, for any $p>2$,%
\begin{equation*}
\left\Vert X\right\Vert _{p\text{-var}}+\left\Vert \mathbb{A}\right\Vert
_{p/2\text{-var}}<\infty \text{ a.s.}
\end{equation*}%
Equivalently, $||\mathbf{X}||_{p\text{-var}}<\infty $ a.s.
\end{proposition}

\begin{proof}
Consider $s,t\in \left[ 0,T\right] $ with $\left\vert t-s\right\vert \leq h$%
. Then%
\begin{eqnarray*}
\mathbb{P}\left( \left\vert X_{s,t}\right\vert \geq a\right) &\leq &\frac{1}{%
a^{2}}\mathbb{E}\left\vert X_{s,t}\right\vert ^{2}\lesssim \frac{h}{a^{2}},
\\
\mathbb{P}\left( \left\vert \mathbb{A}_{s,t}\right\vert ^{1/2}\geq a\right)
&=&\mathbb{P}\left( \left\vert \mathbb{A}_{s,t}\right\vert \geq a^{2}\right)
\leq \frac{1}{a^{2}}\mathbb{E}\left\vert \mathbb{A}_{s,t}\right\vert \\
&\leq &\frac{1}{a^{2}}\left( E\left\vert \mathbb{A}_{s,t}\right\vert
^{2}\right) ^{1/2}=\frac{h}{a^{2}}.
\end{eqnarray*}%
From properties of the Carnot--Caratheodory metric $d_{CC}(\mathbf{X}_{s},%
\mathbf{X}_{t})\asymp \left\vert X_{s,t}\right\vert +\left\vert \mathbb{A}%
_{s,t}\right\vert ^{1/2}$ and the above estimates yield%
\begin{equation*}
\mathbb{P}\left( d(\mathbf{X}_{s},\mathbf{X}_{t})\geq a\right) \lesssim 
\frac{h}{a^{2}}.
\end{equation*}%
Applying the result of Manstavicius (cf. Section \ref{sec:Man}) with $\beta
=1,\gamma =2$ 
we obtain a.s. finite $p$-variation of $\mathbf{X}$, any $p>\gamma /\beta =
2 $, with respect to $d_{CC}$ and the statement follows.
\end{proof}

As will be detailed in Section \ref{sec:LKformula_and_RPreg} below, this
criterion, combined with the expected signature of a $d$-dimensional L\'evy
process, provides an immediate way to recover Williams' rough path
regularity result on L\'evy process (Theorem \ref{theo:will}) and then
significantly larger classes of jump diffusions. With the confidence that
there exists large classes of random c\'adl\'ag rough paths, we continue to
developt the deterministic theory. 

\bigskip

\section{The minimal jump extension of cadlag rough paths}

In view of Theorem \ref{Lyonsfirsttheorem}, it is natural to ask for such
extension theorem for c\'{a}dl\'{a}g rough paths. (For
continuous paths in Young regime, extension is uniquely given by $n$-fold
iterated young integrals.) However, in presence of jumps the uniqueness
part of Lyons' extension theorem fails, as already seen by elementary examples 
of  finite variation
paths.

\begin{example}
Let $p=1,N=2$ and consider the trivial path $X\equiv 0\in W^{1}\left( [0,1],{%
\mathbb{R}}^{d}\right) $, identified with $\mathbf{X}\equiv \left(
1,0\right) \in W^{1}\left( [0,1],G^{(1)}\right) $. Consider a non-trivial $%
so\left( d\right) $-valued cadlag path $a(t)$, of pure finite jump type,
i.e. 
\begin{equation*}
a_{0,t}=\sum_{\substack{ s\in (0,t]  \\ \text{(finite)}}}\Delta a_{s}.
\end{equation*}%
Then two possible lifts of\textbf{\ }$\mathbf{X}$ are given by 
\begin{equation*}
\mathbf{X}^{\left( 2\right) }\equiv \left( 1,0,0\right) ,\,\mathbf{\tilde{X}}%
_{t}^{\left( 2\right) }\equiv \left( 1,0,a_{t}-a_{0}\right) \in \mathcal{W}_{%
\mathrm{g}}^{1\text{-var}}=W^{1}\left( [0,1],G^{(2)}\right) .
\end{equation*}
\end{example}

We can generalize this example as follows. 

\begin{example}
Again $p=1,N=2$ and consider $X\in W^{1\text{-var}}$. Then%
\begin{equation*}
\mathbf{X}_{t}^{\left( 2\right) }:=\left( 1,X_{t},\mathbb{X}_{t}^{\mathrm{M}%
}\right) \in \mathcal{W}_{\mathrm{g}}^{1\text{-var}}
\end{equation*}%
and another choice is given by 
\begin{equation*}
\mathbf{\tilde{X}}_{t}^{\left( 2\right) }\equiv \left( 1,X_{t},\mathbb{X}%
_{t}^{M}+a_{t}-a_{0}\right) \in \mathcal{W}_{\mathrm{g}}^{1\text{-var}},
\end{equation*}%
whenever, $a_{t}\in \mathfrak{so}(d)$ is piecewise constant, with finitely
many jumps $\Delta a_{t}\neq 0$.
\end{example}

Note that, among all such lifts $\mathbf{\tilde{X}}_{t}^{\left( 2\right) }$,
the $\mathbf{X}_{t}^{\left( 2\right) }$ is minimal in the sense that $\log
^{(2)}\Delta \mathbf{X}_{t}^{\left( 2\right) }$ has no $2$-tensor component,
and in fact, 
\begin{equation*}
\log ^{(2)}\Delta \mathbf{X}_{t}^{\left( 2\right) }=\Delta X_{t}.
\end{equation*}%
We have the following far-reaching extension of this example. Note that we
consider $\mathfrak{g}^{n}\supset \mathfrak{g}^{m}$ in the obvious way
whenever $n\geq m$.

\begin{theorem}[Minimal jump extension]
\label{minimaljumpextension}Let $1\leq p<\infty $ and $\mathbb{N}\ni n>m:=%
\left[ p\right] $. A cadlag rough path $\mathbf{X}^{\left( m\right) }\in 
\mathcal{W}_{\mathrm{g}}^{p}=W^{p}\left( \left[ 0,T\right] ,G^{(m)}\right) $
admits an extension to a path $\mathbf{X}^{\left( n\right) }$ of with values 
$G^{(n)}\subset T^{(n)}$, unique in the class of $G^{(n)}$-valued path
starting from $1$ and of finite $p$-variation with respect to CC metric on $%
G^{(n)}$ subject to the additional constraint

\begin{equation}
\log ^{\left( n\right) }\Delta \mathbf{X}_{t}^{\left( n\right) }=\log
^{\left( m\right) }\Delta \mathbf{X}_{t}^{\left( m\right) }.
\label{MinJumpEx}
\end{equation}
\newline
\end{theorem}

For the proof, we will adopt the Marcus / Willliams idea of introducing an
artificial additional time interval at each jump times of $\mathbf{X}^m$,
during which the jump will be suitably traversed. Since $\mathbf{X}^m$ has
countably infinite many jumps, we number the jumps as follows. Let $t_1$ is
such that 
\begin{equation*}
||\Delta_{t_1}\mathbf{X}^{(m)}||_{CC} = \sup_{t \in [0,T]}\{ ||\Delta_{t}%
\mathbf{X}^{(m)}||_{CC} \}
\end{equation*}
Similarly, define $t_2$ with 
\begin{equation*}
||\Delta_{t_2}\mathbf{X}^{(m)}||_{CC} = \sup_{t \in [0,T], t\neq t_1}\{
||\Delta_{t}\mathbf{X}^{(m)}||_{CC} \}
\end{equation*}%
and so on. Note that the suprema are always attained and if $||\Delta_t 
\mathbf{X}^{(m)}||_{CC} \neq 0 $, then $t = t_k $ for some $k$. Indeed, it
readily follows from the c\'adl\'ag (or regulated) property that for any $%
\epsilon > 0$, there are only finitely many jumps with $||\Delta_t \mathbf{X}%
^{(m)} ||_{CC} > \epsilon $.

Choose any sequence $\delta_k > 0$ such that $\sum_k \delta_k < \infty$. 
Starting from $t_1$, we recursively introduce an interval of length $%
\delta_k $ at $t_n$, during which the jump $\Delta_{t_k}\mathbf{X}^{(m)}$ is
traversed suitably, to get a continuous curve $\tilde{\mathbf{X}}^{(m)}$ on
the (finite) interval $[0, \tilde{T}]$ where 

\begin{equation*}
\tilde T = T+ \sum_{k=1}^\infty \delta_k < \infty.
\end{equation*}


Taking motivation from simple examples, in order to get minimal jump
extensions, we choose the \textquotedblleft best possible" curve traversing
the jump, so that it doesn't create additional terms in $\log ^{\left(
n\right) }\Delta \mathbf{X}_{t}^{\left( n\right) }$. If $[a,a+\delta
_{k}]\subset \lbrack 0,\tilde{T}]$, is the jump segment corresponding to the 
$k^{\text{th}}$ jump, define 
\begin{equation*}
\gamma _{t}^{k}=\exp ^{(m)}\left( \frac{a+\delta _{k}-t}{\delta _{k}}\log
^{(m)}\mathbf{X}_{t_{k}-}^{(m)}+\frac{t-a}{\delta _{k}}\log ^{(m)}\mathbf{X}%
_{t_{k}}^{(m)}\right)
\end{equation*}

\begin{lemma}
\label{traversebound} $\gamma ^{k}:[a,a+\delta _{k}]\rightarrow G^{(m)}$ is
a continuous path of finite $p$ variation w.r.t. the CC metric and we have
the bound 
\begin{equation}
||\gamma ^{k}||_{p\text{-var;}[a,a+\delta _{k}]}^{p}\lesssim ||\Delta
_{t_{k}}\mathbf{X}^{(m)}||^{p}.
\end{equation}
\end{lemma}

\begin{proof}
Omit $k$. W.l.o.g. we can assume that $\gamma _{t}=\exp ^{(m)}((1-t)\log
^{(m)}x+t\log ^{(m)}y)$ for $t\in \lbrack 0,1]$ for some $x,y\in G^{(m)}$.
Also, as an application of Campbell-Baker-Hausdorff formula, 
\begin{equation*}
\exp ^{(m)}(\log ^{(m)}x)\otimes \exp ^{(m)}(t\log ^{(m)}(x^{-1}\otimes
y))=\exp ^{(m)}((1-t)\log ^{(m)}x+t\log ^{(m)}y)
\end{equation*}%
so that we can assume $x=1$. At this point, we have 
\begin{equation*}
\gamma _{s,t}=\exp ^{(m)}((t-s)\log ^{(m)}y)
\end{equation*}%
Also, since $p\geq m$, it is easy to check that for $z\in \mathfrak{g}^{m}$
and $\lambda \in \lbrack 0,1]$, 
\begin{equation*}
||\exp ^{(m)}(\lambda x)||^{p}\lesssim \lambda ||\exp ^{(m)}(z)||^{p}
\end{equation*}%
So, 
\begin{equation*}
||\gamma _{s,t}||^{p}\lesssim (t-s)||y||^{p}
\end{equation*}%
which finishes the claim.
\end{proof}

\begin{lemma}
The curve $\tilde{\mathbf{X}}^{(m)}:[0,\tilde{T}]\rightarrow G^{(m)}$
constructed as above from $\mathbf{X}^{\left( m\right) }\in W^{p}\left( %
\left[ 0,T\right] ,G^{(m)}\right) $ is a continuous path of finite $p$
variation w.r.t. the CC metric and we have the bound 
\begin{equation}
||\tilde{\mathbf{X}}^{(m)}||_{p\text{-var;}[0,\tilde{T}]}\lesssim ||{\mathbf{%
X}}^{(m)}||_{p\text{-var;}[0,T]}
\end{equation}
\end{lemma}

\begin{proof}
For simpler notation, omit $m$ and write $\tilde{\mathbf{X}},\mathbf{X}$.
The curve $\tilde{\mathbf{X}}$ is continuous by construction. To see the
estimate, %
%
introduce $\omega (s,t)=||\mathbf{X}||_{p-var,[s,t]}^{p}$ with the notation $%
\omega (s,t-):=||\mathbf{X}||_{p\text{-var,}[s,t)}^{p}$. Note that $\omega
(s,t)$ is superadditive. Call $t^{\prime }$ the preimage of $t\in \lbrack 0,%
\tilde{T}]$ under the time change from $\left[ 0,T\right] \rightarrow
\lbrack 0,\tilde{T}]$. Note that $[0,\tilde{T}]$ contains (possibly
countably many) jump segements $I_{n}$ of the form $[a,a+\delta _{k})$. 
Let us agree that point in these jump segments are \textquotedblleft red"
and all remaining points are \textquotedblleft blue". 
Note that jump segements correspond to one point in the pre-image. For $%
0\leq s<t\leq \tilde{T}$, there are following possiblities;

\begin{itemize}
\item Both $s,t$ are blue, in which case $||\tilde{\mathbf{X}}_{s,t}||^{p}=||%
\mathbf{X}_{s^{\prime },t^{\prime }}||^{p}\leq ||\mathbf{X}||_{p\text{-var,}%
[s^{\prime },t^{\prime }]}^{p}=\omega (s^{\prime },t^{\prime })$

\item Both $s,t$ are red and in same jump segment $[a,a+\delta _{k})$, in
which case 
\begin{equation*}
||\tilde{\mathbf{X}}_{s,t}||^{p}\leq ||\gamma ||_{p\text{-var;}[a,a+\delta
_{k}]}^{p}
\end{equation*}

\item Both $s,t$ are red but in different jump segment $s\in \lbrack
a,a+\delta _{k})$ and $t\in \lbrack b,b+\delta _{l})$, in which case $%
s^{\prime }=(a+\delta _{k})^{\prime }$, $t^{\prime }=(b+\delta _{l})^{\prime
}$ and

\begin{eqnarray*}
\left\Vert \mathbf{\tilde{X}}_{s,t}\right\Vert ^{p} &\leq &3^{p-1}\left(
\left\Vert \mathbf{\tilde{X}}_{s,a+\delta _{k}}\right\Vert ^{p}+\left\Vert 
\mathbf{\tilde{X}}_{a+\delta _{k},b}\right\Vert ^{p}+\left\Vert \mathbf{%
\tilde{X}}_{b,t}\right\Vert ^{p}\right) . \\
&\leq &3^{p-1}\left( ||\gamma ^{k}||_{p\text{-var;}[a,a+\delta _{k}]}^{p}+||%
\mathbf{X}_{s^{\prime },t^{\prime }-}||^{p}+||\gamma ^{l}||_{p\text{-var,}%
[b,b+\delta _{l}]}^{p}\right) \\
&\leq &3^{p-1}\left( ||\gamma ^{k}||_{p\text{-var;}[a,a+\delta
_{k}]}^{p}+\omega (s^{\prime },t^{\prime }-)+||\gamma ^{l}||_{p\text{-var,}%
[b,b+\delta _{l}]}^{p}\right)
\end{eqnarray*}

\item $s$ is blue and $t\in \lbrack a,a+\delta _{k})$ is red, in which case 
\begin{eqnarray*}
\left\Vert \mathbf{\tilde{X}}_{s,t}\right\Vert ^{p} &\leq &2^{p-1}\left(
\left\Vert \mathbf{\tilde{X}}_{s,a}\right\Vert ^{p}+\left\Vert \mathbf{%
\tilde{X}}_{a,t}\right\Vert ^{p}\right) \\
&\leq &2^{p-1}\left( \omega (s^{\prime },t^{\prime }-)+||\gamma ^{k}||_{p%
\text{-var;}[a,a+\delta _{k}]}^{p}\right)
\end{eqnarray*}

\item $s\in \lbrack a,a+\delta _{k})$ is red and $t$ is blue, then

\begin{eqnarray*}
\left\Vert \mathbf{\tilde{X}}_{s,t}\right\Vert ^{p} &\leq &2^{p-1}\left(
\left\Vert \mathbf{\tilde{X}}_{s,a+\delta _{k}}\right\Vert ^{p}+\left\Vert 
\mathbf{\tilde{X}}_{a+\delta _{k},t}\right\Vert ^{p}\right) \\
&\leq &2^{p-1}\left( ||\gamma ^{k}||_{p\text{-var;}[a,a+\delta
_{k}]}^{p}+\omega (s^{\prime },t^{\prime })\right)
\end{eqnarray*}
\end{itemize}

In any case, by using Lemma \ref{traversebound}, we see that 
\begin{equation*}
\left\Vert \mathbf{\tilde{X}}_{s,t}\right\Vert ^{p} \lesssim
\omega(s^{\prime },t^{\prime }) + ||\Delta_{s^{\prime }} \mathbf{X}||^p +
||\Delta_{t^{\prime }} \mathbf{X}||^p
\end{equation*}%
which implies for any partition $\mathcal{P}$ of $[0, \tilde{T}]$, 
\begin{eqnarray*}
\sum_{[s,t] \in \mathcal{P}} \left\Vert \mathbf{\tilde{X}}_{s,t}\right\Vert
^{p} &\lesssim & \sum_{[s^{\prime },t^{\prime }] } \omega(s^{\prime
},t^{\prime }) + ||\Delta_{s^{\prime }} \mathbf{X}||^p + ||\Delta_{t^{\prime
}} \mathbf{X}||^p \\
&\lesssim & \omega(0,T) + \sum_{0 < s \leq T } ||\Delta_{s} \mathbf{X}||^p
\end{eqnarray*}

Finally, note that 
\begin{equation*}
\sum_{0<s\leq T}||\Delta _{s}\mathbf{X}||^{p}\leq ||\mathbf{X}||_{p\text{%
-var;}[0,T]}^{p}
\end{equation*}%
which proves the claim.
\end{proof}

\begin{proof}[Proof of Theorem \protect\ref{minimaljumpextension}]
Since $\tilde{\mathbf{X}}^{(m)}$ is continuous path of finite $p$-variation
on $[0,\tilde{T}]$, from Theorem \ref{Lyonsfirsttheorem}, it admits an
extension $\tilde{\mathbf{X}}^{(n)}$ taking values in $G^{(n)}$ starting
from $1$ for all $n>m$. We emphasize that $S = \tilde{\mathbf{X}}^{(n)}$ can
be obtained as linear RDE solution to 
\begin{equation}
d {S}= {S}\otimes d \mathbf{X}^{(m)},\, {S}_{0}=1 \in T^{(n)}.
\label{equ:sigviaTC}
\end{equation}
We claim that for each jump segment $[a,a+\delta _{k}]$,%
\begin{equation*}
\tilde{\mathbf{X}}_{a,a+\delta _{k}}^{(n)}=\exp ^{(n)}(\log ^{(m)}(\Delta
_{t_{k}}\mathbf{X}^{(m)}))
\end{equation*}%
which amounts to proving that if $\gamma _{t}=\exp ^{(m)}((1-t)\log
^{(m)}x+t\log ^{(m)}y)$ for $t\in \lbrack 0,1]$ for some $x,y\in G^{(m)}$,
then its extension $\gamma ^{(n)}$ to $G^{(n)}$ satisfies 
\begin{equation*}
\gamma _{0,1}^{(n)}=\exp ^{(n)}(\log ^{(m)}(x^{-1}\otimes y))
\end{equation*}%
By Campbell-Baker-Hausdorff formula, 
\begin{equation*}
\exp ^{(m)}(\log ^{(m)}x)\otimes \exp ^{(m)}(t\log ^{(m)}(x^{-1}\otimes
y))=\exp ^{(m)}((1-t)\log ^{(m)}x+t\log ^{(m)}y)
\end{equation*}%
Thus, 
\begin{equation*}
\gamma _{s,t}=\exp ^{m}((t-s)\log ^{m}(x^{-1}\otimes y))
\end{equation*}%
Here we have used our crucial construction that $\log ^{(m)}\gamma _{t}$ is
linear in $t$. Now by guessing and uniqueness of Theorem \ref%
{Lyonsfirsttheorem}, 
\begin{equation*}
\gamma _{s,t}^{(n)}=\exp ^{(n)}((t-s)\log ^{(m)}(x^{-1}\otimes y))
\end{equation*}%
which proves that claim and defining $\mathbf{X}_{t^{\prime }}^{(n)}=\tilde{%
\mathbf{X}}_{t}^{(n)}$ finishes the existence part of Theorem \ref%
{minimaljumpextension}. \newline

For uniqueness, w.l.o.g., assume $n=m+1$. Let $\mathbf{Z}_{t}^{(n)}=\mathbf{X%
}_{t}^{(m)}+M_{t}$ and $\mathbf{Y}_{t}^{(n)}=\mathbf{X}_{t}^{(m)}+N_{t}$ are
two extension of $\mathbf{X}_{t}^{(m)}$ as prescribed of Theorem \ref%
{minimaljumpextension}, where $M_{t},N_{t}\in (\mathbb{R}^{d})^{\otimes n}$. 
\newline
Consider 
\begin{equation*}
S_{t}=\mathbf{Z}_{t}^{(n)}\otimes \{\mathbf{Y}_{t}^{(n)}\}^{-1}=(\mathbf{X}%
_{t}^{(m)}+M_{t})\otimes (\mathbf{X}_{t}^{(m)}+N_{t})^{-1}=1+M_{t}-N_{t}
\end{equation*}%
where the last equality is due to truncation in the (truncated) tensor
product. This in particular implies $S_{t}$ is in centre of the group $%
G^{(n)}$ (actually group $T_{1}^{n}$) and thus so is $S_{s}^{-1}\otimes
S_{t} $. So, by using symmetry and subadditivity of CC norm, 
\begin{equation*}
||S_{s}^{-1}\otimes S_{t}||=||\mathbf{Y}_{s}^{(n)}\otimes \mathbf{Z}%
_{s,t}^{(n)}\otimes \{\mathbf{Y}_{t}^{(n)}\}^{-1}||=||\mathbf{Z}%
_{s,t}^{(n)}\otimes \{\mathbf{Y}_{s,t}^{(n)}\}^{-1}||\leq ||\mathbf{Z}%
_{s,t}^{(n)}||+||\mathbf{Y}_{s,t}^{(n)}||
\end{equation*}%
which implies $S_{t}$ is of finite $p$-variation. Also, 
\begin{equation*}
\ \Delta _{t}S=\mathbf{Y}_{t-}^{(n)}\otimes \Delta _{t}\mathbf{Z}%
^{(n)}\otimes \mathbf{Y}_{t}^{(n)}
\end{equation*}%
Since $\Delta _{t}\mathbf{Z}^{(n)}=\Delta _{t}\mathbf{Y}^{(n)}$, we see that 
$\log ^{\left( n\right) }\Delta _{t}S=0$, i.e. $S_{t}$ is continuous. Thus, $%
M-N$ is a continuous path in $(\mathbb{R}^{d})^{\otimes n}$ with finite $%
\frac{p}{n}<1$ variation, which implies $M_{t}=N_{t}$ concluding the proof.
\end{proof}

\begin{remark}
\label{uniqueextensionremark}In the proof of uniqueness of minimal jump
extension, we didn't use the structure of group $G^{(n)}$. The fact that the
minimal jump extension takes value in $G^{(n)}$ follows by construction.
That said, if $\mathbf{Z}^{(n)}$ and $\mathbf{Y}^{(n)}$ are two extensions
of $\mathbf{X}^{(m)}$ taking values in $T^{(n)}(\mathbb{R}^{d})$, of finite $p$%
-variation w.r.t. norm 
\begin{equation*}
||1+g||:=|g^{1}|+|g^{2}|^{\frac{1}{2}}+..+|g^{n}|^{\frac{1}{n}}
\end{equation*}%
and 
\begin{equation*}
\Delta _{t}\mathbf{Z}^{(n)}=\Delta _{t}\mathbf{Y}^{(n)}=\exp ^{(n)}(\log
^{(m)}(\Delta _{t}\mathbf{X}^{(m)}))
\end{equation*}%
then same argument asw above implies 
\begin{equation*}
\mathbf{Z}_{t}^{(n)}=\mathbf{Y}_{t}^{(n)}
\end{equation*}
\end{remark}

\begin{definition}[Signature of a cadlag rough path]
Given $\mathbf{X}\in W^{p}\left( \left[ 0,T\right] ,G^{([p])}\right) $ call $%
\mathbf{X}^{\left( n\right) }$ constructed above the step-$n$ signature of $%
\mathbf{X}$. The $T(({\mathbb{R}}^{d}))$-valued projective limit of $\mathbf{%
X}_{0,T}^{\left( n\right) }$ as $n\rightarrow \infty $ is called signature
of $\mathbf{X}$ over $[0,T]$.
\end{definition}

\section{Rough integration with jumps}

\label{roughintegrationwithjumps}


In this section, we will define rough integration for c\'{a}dl\'{a}g rough
paths in the spirit of \cite{Young, Max} and apply this for pathwise
understanding of stochastic integral. We restrict ourselves to case $p<3$.
For $p\in \lbrack 1,2)$, Young integration theory is well established and
interesting case is for $p\in \lbrack 2,3)$. Recall the meaning of
convergence in (RRS) sense, cf. Definition \ref{def:RRSMRS}. In order to
cause no confusion between following two choices of Riemann sum
approximation 
\begin{equation*}
S(\mathcal{P}):=\sum_{[s,t]\in \mathcal{P}}Y_{s}X_{s,t}
\end{equation*}%
and 
\begin{equation*}
S^{\prime }(\mathcal{P}):=\sum_{[s,t]\in \mathcal{P}}Y_{s-}X_{s,t}
\end{equation*}%
we add that, if $X$ and $Y$ are regulated paths of finite $p$-variation for $%
p<2$, then 
\begin{equation*}
C:=\mathrm{(RRS)}\,\lim\limits_{\left\vert \mathcal{P}\right\vert
\rightarrow 0}\sum_{[s,t]\in \mathcal{P}}Y_{s}X_{s,t}
\end{equation*}%
exist if either $Y$ is c\'{a}dl\'{a}g or $Y$ is c\'{a}gl\'{a}d (left
continuous with right limit) and $X$ is c\'{a}dl\'{a}g.

This can be easily verified by carefully reviewing the proof of existence of
Young integral as in \cite{dudley}. Note that we have restricted ourselves
to left point evaluation in Riemann sums. Thus if $Y$ is a c\'{a}dl\'{a}g
path then, 
\begin{equation*}
C_{1}:=\mathrm{(RRS)}\lim\limits_{|\mathcal{P}|\rightarrow 0}\sum_{[s,t]\in 
\mathcal{P}}Y_{s}X_{s,t}
\end{equation*}%
and 
\begin{equation*}
C_{2}:=\mathrm{(RRS)}\lim\limits_{|\mathcal{P}|\rightarrow 0}\sum_{[s,t]\in 
\mathcal{P}}Y_{s-}X_{s,t}
\end{equation*}%
both exists. But it doesn't cause any ambiguity because in fact they are
equal.

\begin{proposition}
\label{youngleft=right}If $X$ and $Y$ are c\'{a}dl\'{a}g paths of finite $p$%
-variation for $p<2$, then $C_{1}=C_{2}$.
\end{proposition}

\begin{proof}
For each $\epsilon >0$, 
\begin{equation*}
S(\mathcal{P}):=\sum_{[s,t]\in \mathcal{P}}\Delta
Y_{s}X_{s,t}=\sum_{[s,t]\in \mathcal{P}}\Delta Y_{s}{1}_{|\Delta
Y_{s}|>\epsilon }X_{s,t}+\sum_{[s,t]\in \mathcal{P}}\Delta Y_{s}{1}_{|\Delta
Y_{s}|\leq \epsilon }X_{s,t}
\end{equation*}%
Since there are finitely many jumps of size bigger than $\epsilon $ and $X$
is right continuous, 
\begin{equation*}
\lim\limits_{|\mathcal{P}|\rightarrow 0}\sum_{[s,t]\in \mathcal{P}}\Delta
Y_{s}{1}_{|\Delta Y_{s}|>\epsilon }X_{s,t}=0
\end{equation*}%
On the other hand,%
\begin{eqnarray*}
\left\vert \sum_{\lbrack s,t]\in \mathcal{P}}\Delta Y_{s}{1}_{|\Delta
Y_{s}|\leq \epsilon }X_{s,t}\right\vert ^{2} &\leq &\sum_{[s,t]\in \mathcal{P%
}}(\left\vert \Delta Y_{s}\right\vert ^{2}{1}_{|\Delta Y_{s}|\leq \epsilon
})\sum_{[s,t]\in \mathcal{P}}{|}X_{s,t}|^{2} \\
&\leq &\varepsilon ^{2-p}\sum_{[s,t]\in \mathcal{P}}\left\vert \Delta
Y_{s}\right\vert ^{p}\sum_{[s,t]\in \mathcal{P}}{|}X_{s,t}|^{2} \\
&\leq &\varepsilon ^{2-p}\left\Vert Y\right\Vert _{p\text{-var}%
}^{p}\left\Vert X\right\Vert _{2\text{-var}}^{2}
\end{eqnarray*}%
where we used $p < 2$ in the step. It thus follows that 
\begin{equation*}
\lim\limits_{\epsilon \rightarrow 0}\lim\limits_{|\mathcal{P}|\rightarrow
0}\sum_{[s,t]\in \mathcal{P}}\Delta Y_{s}{1}_{|\Delta Y_{s}|\leq \epsilon
}X_{s,t}=0
\end{equation*}%
which proves the claim.
\end{proof}

One fundamental difference between continuous and c\'adl\'ag cases is
absence of uniform continuity which implies small oscillation of a path in
small time interval. This becomes crucial in the construction of integral,
as also can be seen in construction of Young integral (see \cite{dudley})
when the integrator and integrand are assumed to have no common
discontinuity on the same side of a point. This guarantees at least one of
them to have small oscillation on small time intervals.

\begin{definition}
A pair of functions $(X_{s,t}, Y_{s,t})$ defined for $\{0 \leq s \leq t \leq
T\}$ is called compatible if for all $\epsilon > 0$, there exist a partition 
$\tau = \{0 = t_0 < t_1 \cdots < t_n = T \}$ such that for all $0 \leq i
\leq n-1$, 
\begin{equation*}
Osc(X, [t_i,t_{i+1}]) \leq \epsilon \hspace{2mm}\mbox{\textbf{OR}} \hspace{%
2mm} Osc(Y, [t_i,t_{i+1}]) \leq \epsilon
\end{equation*}
where $Osc(Z, [s,t]) := \sup\{ |Z_{u,v}| \big\vert s \leq u \leq v \leq t \}$%
.
\end{definition}

\begin{proposition}
If $X$ is a c\'adl\'ag path and $Y$ is c\'agl\'ad path, then $(X,Y)$ is a
compatible pair.
\end{proposition}

\begin{proof}
See \cite{dudley}
\end{proof}

Ler $\mathbf{X} = (X, \mathbb{X})$ be c\'adl\'ag rough path in the sense of
Definition \ref{definitionofRP}. For the purpose of rough integration we
will use a different enhancement 
\begin{equation*}
\tilde{\mathbb{X}}_{s,t} = \mathbb{X}_{s,t} + \Delta_s X \otimes X_{s,t}
\end{equation*}%
Note clearly that $\tilde{\mathbb{X}}$ is also of finite $\frac{p}{2}$
variation, $\tilde{\mathbb{X}}_{0,t}$ is c\'adl\'ag path and for $s \leq u
\leq t $, 
\begin{equation}  \label{chen}
\tilde{\mathbb{X}}_{s,t} - \tilde{\mathbb{X}}_{s,u} - \tilde{\mathbb{X}}%
_{u,t} = X_{s,u}^{-}\otimes X_{u,t}
\end{equation}
where $X_t^- := X_{t-}$, $X_0^- = X_0 = 0$.

\begin{lemma}
For any $\epsilon > 0$, there exist a partition $\tau = \{0 = t_0 < t_1
\cdots < t_n = T \}$ such that for all $0 \leq i \leq n-1$, 
\begin{equation*}
Osc(\tilde{\mathbb{X}} ,(t_i , t_{i+1})) \leq \epsilon
\end{equation*}
\end{lemma}

\begin{proof}
Since $\tilde{\mathbb{X}}_{0,t}$ is c\'adl\'ag, from \eqref{chen}, it
follows that for each $y \in (0, T) $, there exist a $\delta_y > 0 $ such
that 
\begin{equation*}
Osc (\tilde{\mathbb{X}}, (y - \delta_y , y)) \leq \epsilon \hspace{2mm} %
\mbox{and} \hspace{2mm} Osc (\tilde{\mathbb{X}}, (y, y + \delta_y )) \leq
\epsilon
\end{equation*}
Similarly there exist $\delta_0 $ and $\delta_T $ such that $Osc (\tilde{%
\mathbb{X}}, (0, \delta_0 )) \leq \epsilon $ and $Osc (\tilde{\mathbb{X}},
(T - \delta_T, T)) \leq \epsilon $. Now family of open sets 
\begin{equation*}
[0, \delta_0) , (y - \delta_y , y+ \delta_y ),... , (T - \delta_T, T]
\end{equation*}
form a open cover of interval $[0,T]$, so it has a finite subcover $[0,
\delta_0 ), (y_1- \delta_{y_1}, y_1 + \delta_{y_1}), .. , (y_n-
\delta_{y_n}, y_n + \delta_{y_n}), (T - \delta_T, T]$. Without loss of
generality, we can assume that each interval in the finite subcover is the
first interval that intersects its previous one and the claim follows by
choosing 
\begin{equation*}
t_0 =0, t_1 \in ( y_1- \delta_{y_1}, \delta_0), t_2 = y_1, t_3 \in (y_2
-\delta_{y_2}, y_1 + \delta_{y_1} ), ..., t_{2n+1} = T
\end{equation*}
\end{proof}

\begin{lemma}
For any c\'agl\'ad path $Y$, the pair $(Y,\tilde{\mathbb{X}} )$ is a
compatible pair.
\end{lemma}

\begin{proof}
Choose a partition $\tau$ such that for all $[s,t] \in \tau $, 
\begin{equation*}
Osc( Z, (s,t)) \leq \epsilon
\end{equation*}
for $Z = Y, X, X^-, \tilde{\mathbb{X}} $. We refine the partition $\tau $ by
adding a common continuity point of $Y_t$, $X_t$, $X_t ^- $ and $\tilde{%
\mathbb{X}}_{0,t}$ in each interval $(s,t)$. Note that such common
continuity points will exist because a regulated paths can have only
countably many discontinuities. With this choice of partition, we observe
that on every odd numbered $[s,t] \in \tau$, 
\begin{equation*}
Osc( \tilde{\mathbb{X}}, [s, t]) \leq \epsilon
\end{equation*}
and on every even numbered $[s,t] \in \tau $, 
\begin{equation*}
Osc(Y, [s, t]) \leq \epsilon
\end{equation*}
\end{proof}

\begin{definition}
Given $X\in W^{p}$, a pair of c\'adl\'ag paths $(Y, Y^{\prime })$ of finite $p$-variation is
called controlled rough path if $R_{s,t} = Y_{s,t} - Y_s^{\prime }X_{s,t} 
$ has finite $\frac{p}{2}$-variation, in the sense
\begin{equation*}
||R||_{\frac{p}{2}} := \sup_{\mathcal{P}} \bigl\{\sum_{[s,t]\in \mathcal{P}%
}|R_{s,t}|^{\frac{p}{2}}\bigl \}^{\frac{2}{p}} < \infty.
\end{equation*}
\end{definition}

It is easy to see that $1$-forms $(Y_t, Y_t^{\prime }) := (f(X_t), f^{\prime
}(X_t))$ for $f \in C^2$ is a controlled rough path. Also 
\begin{equation*}
\tilde{R}_{s,t} := Y_{s,t}^- - Y_{s-}^{\prime }X_{s,t}^-
\end{equation*}
is also of finite $\frac{p}{2}$-variation and pair $(\tilde{R}, X )$ is a
compatible pair.

\begin{theorem}
\label{roughintegration} Let $\mathbf{X}=(X,\mathbb{X})$ be a c\'{a}dl\'{a}g
rough path and $(Y,Y^{\prime })$ a controlled rough path, then 
\begin{equation*}
\int_{0}^{T}Y_{r-}d\mathbf{X}_{r}:=\lim\limits_{|\mathcal{P}|\rightarrow 0}S(%
\mathcal{P})=\lim\limits_{|\mathcal{P}|\rightarrow 0}S^{\prime }(\mathcal{P})
\end{equation*}%
where both limits exist in (RRS) sense, as introduced in Definition \ref%
{def:RRSMRS} and 
\begin{eqnarray*}
S(\mathcal{P}) &:=&\sum_{[s,t]\in \mathcal{P}}Y_{s-}X_{s,t}+Y_{s-}^{\prime }%
\tilde{\mathbb{X}}_{s,t}=\sum_{[s,t]\in \mathcal{P}}Y_{s-}X_{s,t}+Y_{s-}^{%
\prime }(\mathbb{X}_{s,t}+\Delta _{s}X\otimes X_{s,t}) \\
S^{\prime }(\mathcal{P}) &=&\sum_{[s,t]\in \mathcal{P}}Y_{s}X_{s,t}+Y_{s}^{%
\prime }\mathbb{X}_{s,t}.
\end{eqnarray*}%
Furthermore, we have the following rough path estimates: there exist a
constant $C$ depending only on $p$ such that 
\begin{eqnarray}  \label{young1}
\left\vert \int_{s}^{t}Y_{r-}d\mathbf{X}_{r}-Y_{s-}X_{s,t}-Y_{s-}^{\prime }%
\tilde{\mathbb{X}}_{s,t}\right\vert &\leq &C\left( ||\tilde{R}||_{\frac{p}{2}%
,[s,t]}||X||_{p,[s,t]}+||Y^{^{\prime }-}||_{p,[s,t]}||\tilde{\mathbb{X}}||_{%
\frac{p}{2},[s,t]}\right)
\end{eqnarray}
\begin{eqnarray}  \label{young2}
\left\vert \int_{s}^{t}Y_{r-}d\mathbf{X}_{r}-Y_{s}X_{s,t}-Y_{s}^{\prime }%
\mathbb{X}_{s,t}\right\vert &\leq &C\left( ||R||_{\frac{p}{2}%
,[s,t]}||X||_{p,[s,t]}+||Y^{^{\prime }}||_{p,[s,t]}||\mathbb{X}||_{\frac{p}{2%
},[s,t]}\right) .
\end{eqnarray}
\end{theorem}

\begin{proof}
We first consider the approximations given by $S(\mathcal{P})$. We first
note that if $\omega $ is a superadditive function defined on intervals,
i.e. for all $s\leq u\leq t$%
\begin{equation*}
\omega \lbrack s,u]+\omega \lbrack u,t]\leq \omega \lbrack s,t]
\end{equation*}%
then, for any partition $\mathcal{P}$ of $[s,t]$ into $r\geq 2$ intervals,
there exist intervals $[u_{-},u]$ and $[u,u_{+}]$ such that 
\begin{equation}
\omega \lbrack u_{-},u_{+}]\leq \frac{2}{r-1}\omega \lbrack s,t]
\label{superadd}
\end{equation}%
Also, we can immediately verify that for $Z$ of finite $p$-variation, 
\begin{equation*}
\omega \lbrack s,t]:=||Z||_{p,[s,t]}^{p}
\end{equation*}%
defines and superadditive function and if $\omega _{1}$ and $\omega _{2}$
are two positive superaddtive functions, then for $\alpha ,\beta \geq
0,\alpha +\beta \geq 1$, 
\begin{equation*}
\omega :=\omega _{1}^{\alpha }\omega _{2}^{\beta }
\end{equation*}%
is also a superadditive function.\newline
Now, it is enough to prove that for any $\epsilon >0$, there exist a
partition $\tau $ (to be chosen properly) such that for all refinement
partition $\mathcal{P}$ of $\tau $, 
\begin{equation*}
|S(\mathcal{P})-S(\tau )|\leq \epsilon
\end{equation*}%
Choose $p<p^{\prime }<3$ and let $[s,t]\in \tau $ and $\mathcal{P}_{s,t}$ be
the partition of $[s,t]$ by refinement points of $\mathcal{P}$. Note that 
\begin{equation*}
\omega \lbrack s,t]:=||\tilde{R}||_{\frac{p^{\prime }}{2},[s,t]}^{\frac{%
p^{\prime }}{3}}||X||_{p^{\prime },[s,t]}^{\frac{p^{\prime }}{3}%
}+||Y^{^{\prime }-}||_{p^{\prime },[s,t]}^{\frac{p^{\prime }}{3}}||\tilde{%
\mathbb{X}}||_{\frac{p^{\prime }}{2},[s,t]}^{\frac{p^{\prime }}{3}}
\end{equation*}%
is a superadditive and there exist $u_{-}<u<u_{+}\in \mathcal{P}_{s,t}$ such
that \eqref{superadd} holds. Using \eqref{chen} 
\begin{align*}
|S(\mathcal{P}_{s,t})-S(\mathcal{P}_{s,t}\setminus u)|=& |\tilde{R}%
_{u_{-},u}X_{u,u_{+}}+Y_{u_{-},u}^{^{\prime }-}\tilde{\mathbb{X}}_{u,u_{+}}|
\\
& \leq ||\tilde{R}||_{\frac{p^{\prime }}{2},[u_{-},u_{+}]}||X||_{p^{\prime
},[u_{-},u_{+}]}+||Y^{^{\prime }-}||_{p^{\prime },[u_{-},u_{+}]}||\tilde{%
\mathbb{X}}||_{\frac{p^{\prime }}{2},[u_{-},u_{+}]} \\
& \leq (||\tilde{R}||_{\frac{p^{\prime }}{2},[u_{-},u_{+}]}^{\frac{p^{\prime
}}{3}}||X||_{p^{\prime },[u_{-},u_{+}]}^{\frac{p^{\prime }}{3}%
}+||Y^{^{\prime }-}||_{p^{\prime },[u_{-},u_{+}]}^{\frac{p^{\prime }}{3}}||%
\tilde{\mathbb{X}}||_{\frac{p^{\prime }}{2},[u_{-},u_{+}]}^{\frac{p^{\prime }%
}{3}})^{\frac{3}{p^{\prime }}} \\
& \leq \frac{C}{(r-1)^{\frac{3}{p^{\prime }}}}(||\tilde{R}||_{\frac{%
p^{\prime }}{2},[s,t]}^{\frac{p^{\prime }}{3}}||X||_{p^{\prime },[s,t]}^{%
\frac{p^{\prime }}{3}}+||Y^{^{\prime }-}||_{p^{\prime },[s,t]}^{\frac{%
p^{\prime }}{3}}||\tilde{\mathbb{X}}||_{\frac{p^{\prime }}{2},[s,t]}^{\frac{%
p^{\prime }}{3}})^{\frac{3}{p^{\prime }}} \\
& \leq \frac{C}{(r-1)^{\frac{3}{p^{\prime }}}}(||\tilde{R}||_{\frac{%
p^{\prime }}{2},[s,t]}||X||_{p^{\prime },[s,t]}+||Y^{^{\prime
}-}||_{p^{\prime },[s,t]}||\tilde{\mathbb{X}}||_{\frac{p^{\prime }}{2}%
,[s,t]})
\end{align*}%
where $C$ is a generic constant. Iterating this, since $p^{\prime }<3$, we
get that 
\begin{equation*}
|S(\mathcal{P}_{s,t})-Y_{s-}X_{s,t}+Y_{s-}^{\prime }\tilde{\mathbb{X}}%
_{s,t}|\leq C(||\tilde{R}||_{\frac{p^{\prime }}{2},[s,t]}||X||_{p^{\prime
},[s,t]}+||Y^{^{\prime }-}||_{p^{\prime },[s,t]}||\tilde{\mathbb{X}}||_{%
\frac{p^{\prime }}{2},[s,t]})
\end{equation*}%
Thus, 
\begin{equation*}
|S(\mathcal{P})-S(\tau )|\leq C\sum_{[s,t]\in \tau }||\tilde{R}||_{\frac{%
p^{\prime }}{2},[s,t]}||X||_{p^{\prime },[s,t]}+||Y^{^{\prime
}-}||_{p^{\prime },[s,t]}||\tilde{\mathbb{X}}||_{\frac{p^{\prime }}{2},[s,t]}
\end{equation*}%
Note that $(\tilde{R},X)$ and $(Y^{^{\prime }-},\tilde{\mathbb{X}})$ are
compatible pairs. Properly choosing $\tau $, 
\begin{equation*}
|S(\mathcal{P})-S(\tau )|\leq C\epsilon \sum_{\lbrack s,t]\in \tau }||\tilde{%
R}||_{\frac{p}{2},[s,t]}^{\frac{p}{p^{\prime }}}||X||_{p,[s,t]}^{\frac{p}{%
p^{\prime }}}+||Y^{^{\prime }-}||_{p,[s,t]}^{\frac{p}{p^{\prime }}}||\tilde{%
\mathbb{X}}||_{\frac{p}{2},[s,t]}^{\frac{p}{p^{\prime }}}
\end{equation*}%
Finally, the term under summation sign is superadditive which thereby
implies 
\begin{equation*}
|S(\mathcal{P})-S(\tau )|\leq C\epsilon
\end{equation*}%
Also, the estimate \eqref{young1} follows immediately as a by product of the
analysis above.

At last, let us deal with the case of \ Riemann sum approximations 
\begin{equation*}
S^{\prime }(\mathcal{P})=\sum_{[s,t]\in \mathcal{P}}Y_{s}X_{s,t}+Y_{s}^{%
\prime }\mathbb{X}_{s,t}.
\end{equation*}%
It suffices to consider the difference 
\begin{equation*}
S^{\prime }(\mathcal{P})-S(\mathcal{P})=\sum_{[s,t]\in \mathcal{P}%
}R_{s-,s}X_{s,t}+\Delta Y_{s}^{\prime }\mathbb{X}_{s,t}
\end{equation*}%
and then use arguments similar as those in the proof of proposition \ref%
{youngleft=right} to see tgat 
\begin{equation*}
\mathrm{(RRS)}\lim\limits_{|\mathcal{P}|\rightarrow 0}(S^{\prime }(\mathcal{P%
})-S(\mathcal{P}))=0.
\end{equation*}%
The rest is then clear.
\end{proof}

As an immediate corollary of \eqref{young1} and \eqref{young2}, we have

\begin{corollary}
\label{roughintegraliscontrolled} For a controlled rough path $(Y,Y^{\prime })$, 
\begin{equation*}
(Z_{t},Z_{t}^{\prime }):=\left( \int_{0}^{t}Y_{r-}d\mathbf{X}%
_{r},Y_{t}\right)
\end{equation*}%
is also a controlled rough path.
\end{corollary}

\begin{corollary}
\label{jumpofRI}If $(Y, Y^{\prime })$ is a controlled rough path and $Z_t =
\int_0^t Y_{r-}d\mathbf{X}_r$, then 
\begin{equation*}
\Delta_tZ = \lim\limits_{s\uparrow t} \int_s^t Y_{r-}d\mathbf{X}_r =
Y_{t-}\Delta_t X + Y_{t-}^{\prime }\Delta_t\mathbb{X}
\end{equation*}%
where $\Delta_t\mathbb{X}= \lim\limits_{s\uparrow t} \mathbb{X}_{s,t}$.
\end{corollary}

Though we avoid to write down the long expression for the bounds of $%
||Z||_p, ||Z^{\prime }||_p$ and $||R^Z||_{\frac{p}{2}}$, it can be easily
derived from \eqref{young2}. The important point here is that we can again,
for $Z$ taking value in suitable spaces, readily define 
\begin{equation*}
\int_{0}^{t} Z_{r-}d\mathbf{X}_r
\end{equation*}%
The rough integral defined above is also compatible with Young integral. If $%
X$ is a finite $p$-variation path for $p < 2$, we can construct c\'adl\'ag
rough path $\mathbf{X}$ by 
\begin{equation*}
\mathbb{X}_{s,t} := \int_s^t( X_{r-}- X_{s})\otimes dX_r
\end{equation*}%
where right hand side is understood as a Young integral.

\begin{proposition}
\label{youngapplication} If $X,Y$ are c\'adl\'ag path of finite $p$ and $q$
variation respectively with $\frac{1}{p} + \frac{1}{q} > 1$, then for any $%
\theta > 0$ with $\frac{1}{p} + \frac{1}{q} \geq \frac{1}{\theta}$,%
\begin{equation*}
Z_{s,t}:= \int_s^t (Y_{r-}-Y_{s})dX_r
\end{equation*}%
has finite $\theta$ variation. In particular, $\mathbb{X} $ has finite $%
\frac{p}{2}$ variation.
\end{proposition}

\begin{proof}
From Young's inequality, 
\begin{equation*}
|Z_{s,t}|^\theta \leq C ||X||_{p,[s,t]}^\theta||Y||_{q,[s,t]}^\theta
\end{equation*}
If $\frac{1}{p} + \frac{1}{q} \geq \frac{1}{\theta}$, right hand side is
superaddtive, which implies $||Z||_\theta < \infty$
\end{proof}

\begin{theorem}
\label{youngcompatible} If $X,Y$ are c\'adl\'ag paths of of finite $p$%
-variation for $p< 2$, then 
\begin{equation*}
\int_{0}^{t} Y_{r-}d\mathbf{X}_r = \int_{0}^{t} Y_{r-}dX_r
\end{equation*}
\end{theorem}

\begin{proof}
The difference between Riemann sum approximation of corresponding integrals
can be written as 
\begin{equation*}
S(\mathcal{P})= \sum_{[s,t]\in \mathcal{P}} Y_{s-}^{\prime }\tilde{\mathbb{X}%
}_{s,t}
\end{equation*}
Choose $p< p^{\prime }<2$. From Young's inequality,%
\begin{equation*}
|\tilde{\mathbb{X}}_{s,t}| = |\int_s^t (X_{r-} - X_{s-})\otimes dX_r| \leq C
||X^-||_{p^{\prime },[s,t]}||X||_{p^{\prime },[s,t]}
\end{equation*}
$(X^-, X)$ is a compatible pair, which implies for properly chosen $\mathcal{%
P}$, 
\begin{equation*}
|S(\mathcal{P})| \leq C\epsilon \sum_{[s,t]\in \mathcal{P}}
||X^-||_{p,[s,t]}^{\frac{p}{p^{\prime }}}||X||_{p,[s,t]}^{\frac{p}{p^{\prime
}}}
\end{equation*}%
Noting again that the term under summation sign is superadditive,%
\begin{equation*}
(RRS)\lim\limits_{|\mathcal{P}|\rightarrow 0 } S(\mathcal{P}) = 0
\end{equation*}
\end{proof}

\section{Rough differential equations with jumps}

%

In the case of \textit{continuous} RDEs the difference between non-geometric
(It\^o-type) and geometric situations, is entirely captured in one's choice
of the second order information $\mathbb{X}$, so that both cases are handled
with the \textit{same} notion of (continuous) RDE solution. In the jump
setting, the situation is different and a geometric notion of RDE solution
requires additional terms in the equation in the spirit of Marcus' canonical
(stochastic) equations \cite{Marcus1,Marcus2,kpp,applebaum}. We now define
both solution concepts for RDEs with jumps, or course they coincide in
absence of jumps, $(\Delta X_s, \Delta \mathbb{X}_{s}) \equiv (0,0)$.

\begin{definition}
(i) For suitable $f$ and a c\'{a}dl\'{a}g geometric $p$-rough path $\mathbf{X%
}=(X,\mathbb{X})\in \mathcal{W}_{\mathrm{g}}^{p}$, call a path $Z$ (or
better: controlled rough path $(Z,f(Z))$) solution to the \textit{rough canonical
equation} 
\begin{equation*}
dZ_{t}=f\left( Z_{t}\right) \diamond d\mathbf{X}_{t}
\end{equation*}%
if, by definition,%
\begin{equation*}
{Z}_{t}=Z_{0}+\int_{0}^{t}f({Z}_{s-})d\mathbf{X}_{s}+\sum_{0<s\leq t}\phi
\left( f\Delta X_{s}+\frac{1}{2}\left[ f,f\right] \Delta \mathbb{X}%
_{s};Z_{s-}\right) -Z_{s-}-f(Z_{s-})\Delta X_{s}-f^{\prime }f\left(
Z_{s-}\right) \Delta \mathbb{X}_{s}\}
\end{equation*}%
where, as in Section \ref{sec:Marcus}, $\phi (g,x)$ is the time $1$ solution
to $\dot{y}=g(y),\hspace{2mm}y(0)=x$. \ When $\mathbf{X}$ is Marcus like,
i.e. $\mathbf{X}\in \mathcal{W}_{\mathrm{M}}^{p}$ so that $\Delta \mathbb{X}%
_{s}=$ $\left( \Delta X_{s}\right) ^{\otimes 2}/2$, this becomes%
\begin{equation*}
{Z}_{t}=Z_{0}+\int_{0}^{t}f({Z}_{s-})d\mathbf{X}_{s}+\sum_{0<s\leq t}\{\phi
(f\Delta X_{s},Z_{s-})-Z_{s-}-f(Z_{s-})\Delta X_{s}-f^{\prime }f\left(
Z_{s-}\right) \frac{1}{2}\left( \Delta X_{s}\right) ^{\otimes 2}\}.
\end{equation*}%
(ii) For suitable $f$ and a c\'{a}dl\'{a}g $p$-rough path call a path $Z$
(or better: controlled rough path $(Z,f(Z))$) solution to the \textit{(general)
rough differential equation} 
\begin{equation*}
dZ_{t}=f\left( Z_{t-}\right) d\mathbf{X}_{t}
\end{equation*}%
if, by definition,%
\begin{equation*}
{Z}_{t}=Z_{0}+\int_{0}^{t}f({Z}_{s-})d\mathbf{X}_{s}.
\end{equation*}
\end{definition}

We shall not consider the solution type (ii) further here.

\begin{theorem}
Fix initial data $Z_{0}$. Then $Z$ is a solution to $dZ_{t}=f\left(
Z_{t}\right) \diamond d\mathbf{X}_{t}$ if and only if $\tilde{Z}$ is a
solution to the (continuous)\ RDE%
\begin{equation*}
d\tilde{Z}_{t}=f\left( \tilde{Z}_{t}\right) d\mathbf{\tilde{X}}_{t}
\end{equation*}%
where $\mathbf{\tilde{X}}\in \mathcal{C}_{g}^{p}$ is constructed from $%
\mathbf{X}\in \mathcal{W}_{g}^{p}$ as in Theorem \ref{minimaljumpextension}.
\end{theorem}

\begin{proof}
We illustrate the idea by considering $X$ of finite $1$-variation, with 
\textit{one} jump at $\tau \in \lbrack 0,T]$. This jump time becomes an
interval $\tilde{I}=\left[ a,a+\delta \right] \subset \lbrack 0,\tilde{T}%
]=[0,T+\delta ]$ in the stretched time scale. Now 
\begin{equation*}
\tilde{Z}_{0,\tilde{T}}\approx \sum_{\left[ s,t\right] \in \mathcal{P}%
}f\left( \tilde{Z}_{s}\right) \tilde{X}\,_{s,t}
\end{equation*}%
in the sense of (MRS) convergence, as $\left\vert \mathcal{P}\right\vert
\rightarrow 0$. In particular, noting that $\tilde{X}\,_{s,t}=\frac{\left(
t-s\right) }{\delta }\Delta X_{\tau }$ whenever $\left[ s,t\right] \subset
\lbrack a,a+\delta ]$%
\begin{eqnarray*}
\tilde{Z}_{a,a+\delta } &=&\lim_{\left\vert \mathcal{\tilde{P}}\right\vert
\rightarrow 0}\sum_{\left[ s,t\right] \in \mathcal{\tilde{P}}}f\left( \tilde{%
Z}_{s}\right) \tilde{X}\,_{s,t}=\frac{1}{\delta }\int_{a}^{a+\delta }f\left( 
\tilde{Z_r}\right) \Delta X_{\tau }dr \\
&\implies &\tilde{Z}_{a,a+\delta }=\phi (f\Delta X_{\tau },\tilde{Z}_{a})-%
\tilde{Z}_{a}
\end{eqnarray*}%
On the other hand, by refinement of $\mathcal{P}$, we may insist that the
end-point of $\tilde{I}$ are contained in $\mathcal{P}$ which thus has the
form 
\begin{equation*}
\mathcal{P=P}_{1}\cup \mathcal{\tilde{P}}\cup \mathcal{P}_{2}
\end{equation*}%
and so 
\begin{equation*}
\tilde{Z}_{0,\tilde{T}}\approx \sum_{\left[ s,t\right] \in \mathcal{P}%
_{1}}f\left( \tilde{Z}_{s}\right) \tilde{X}\,_{s,t}+\sum_{\left[ s,t\right]
\in \mathcal{\tilde{P}}}f\left( \tilde{Z}_{s}\right) \tilde{X}\,_{s,t}+\sum_{%
\left[ s,t\right] \in \mathcal{P}_{2}}f\left( \tilde{Z}_{s}\right) \tilde{X}%
\,_{s,t}
\end{equation*}%
from which we learn, by sending $\left\vert \mathcal{\tilde{P}}\right\vert
\rightarrow 0$, that%
\begin{equation*}
\tilde{Z}_{0,\tilde{T}}\approx \sum_{\left[ s,t\right] \in \mathcal{P}%
_{1}}f\left( \tilde{Z}_{s}\right) \tilde{X}\,_{s,t}+\phi (f\Delta X_{\tau },%
\tilde{Z}_{a})-\tilde{Z}_{a}+\sum_{\left[ s,t\right] \in \mathcal{P}%
_{2}}f\left( \tilde{Z}_{s}\right) \tilde{X}\,_{s,t}.
\end{equation*}%
We now switch back to the original time scale. Of course, $Z\equiv \tilde{Z}$
on $[0,\tau )$ while $Z_{t}=\tilde{Z}_{t+\delta }$ on $[\tau ,T]$ and in
particular 
\begin{eqnarray*}
Z_{0,T} &=&\tilde{Z}_{0,\tilde{T}} \\
Z_{\tau -} &=&\tilde{Z}_{a} \\
Z_{\tau } &=&\tilde{Z}_{a+\delta }.
\end{eqnarray*}
\ But then, with $\mathcal{P}_{1}^{\prime }$ and $\mathcal{P}_{2}^{\prime }$
partitions of $\left[ 0,\tau \right] $ and $\left[ \tau ,T\right] $,
respectively, 
\begin{eqnarray*}
Z_{0,T} &\approx &\sum_{\substack{ \left[ s^{\prime },t^{\prime }\right] \in 
\mathcal{P}_{1}^{\prime }  \\ t^{\prime }<\tau }}f\left( Z_{s^{\prime
}}\right) X\,_{s^{\prime },t^{\prime }}+\sum_{\left[ s^{\prime },t^{\prime }%
\right] \in \mathcal{P}_{2}^{\prime }}f\left( Z_{s^{\prime }}\right)
X\,_{s^{\prime },t^{\prime }}+\phi (f\Delta X_{\tau },Z_{\tau -})-Z_{\tau -}
\\
&\approx &\sum_{\left[ s^{\prime },t^{\prime }\right] \in \mathcal{P}%
^{\prime }}f\left( Z_{s^{\prime }}\right) X\,_{s^{\prime },t^{\prime }}+\phi
(f\Delta X_{\tau },\tilde{Z}_{a})+\left\{ \phi (f\Delta X_{\tau },Z_{\tau
-})-Z_{\tau -}-f\left( Z_{\tau -}\right) \Delta X_{\tau }\right\}
\end{eqnarray*}%
since $f\left( Z_{s^{\prime }}\right) X\,_{s^{\prime },\tau }\rightarrow $ $%
f\left( Z_{\tau -}\right) \Delta X_{\tau }$ as $\left\vert \mathcal{P}%
^{\prime }\right\vert \rightarrow 0$, with $[s^{\prime },\tau ]\in \mathcal{P%
}^{\prime }$. By passing to the \textrm{(RRS)}\ limit, find%
\begin{equation*}
Z_{0,T}=\int_{0}^{T}f\left( Z_{s}^{-}\right) dX+\left\{ \phi (f\Delta
X_{\tau },Z_{\tau -})-Z_{\tau -}-f\left( Z_{\tau -}\right) \Delta X_{\tau
}\right\} .
\end{equation*}%
This argument extends to countable many jumps. We want to show that%
\begin{eqnarray*}
{Z}_{T} &=&Z_{0}+\int_{0}^{T}f({Z}_{s-})dX_{s}+\sum_{0<s\leq T}\left\{
...\right\} \\
&=&Z_{0}+\mathrm{(RRS)}\,\lim\limits_{\left\vert \mathcal{P}\right\vert
\rightarrow 0}\sum_{\left[ s,t\right] \in \mathcal{P}}f({Z}%
_{s})X_{s,t}+\lim_{\eta \downarrow 0}\sum_{\substack{ s\in (0,t]:  \\ %
\left\vert \Delta X_{s}\right\vert >\eta }}\left\{ ...\right\} .
\end{eqnarray*}%
What we know is $\mathrm{(MRS)}$-convergence of the time-changed problem.
That is, given $\varepsilon >0$, there exists $\delta $ s.t. $\left\vert 
\mathcal{P}\right\vert <\delta $ implies%
\begin{equation*}
\tilde{Z}_{0,\tilde{T}}\approx _{\varepsilon }\sum_{\left[ s,t\right] \in 
\mathcal{P}}f\left( \tilde{Z}_{s}\right) \tilde{X}\,_{s,t}
\end{equation*}%
where $a\approx _{\varepsilon }b$ means $\left\vert a-b\right\vert \leq
\varepsilon $. For fixed $\eta >0$, include all (but only finitely many, say 
$N$) points $s\in (0,t]:$ $\left\vert \Delta X_{s}\right\vert >\eta $ in $%
\mathcal{P}$, giving rise to $\left( \mathcal{\tilde{P}}_{j}:1\leq j\leq
N\right) $. Sending the mesh of these to zero gives, as before,%
\begin{equation*}
Z_{0,T}\approx _{\varepsilon }\sum_{\left[ s^{\prime },t^{\prime }\right]
\in \mathcal{P}^{\prime }}f\left( Z_{s^{\prime }}\right) X\,_{s^{\prime
},t^{\prime }}+\sum_{\substack{ s\in (0,t]:  \\ \left\vert \Delta
X_{s}\right\vert >\eta }}\left\{ \phi (f\Delta X_{s},Z_{s-})-Z_{s-}-f\left(
Z_{s-}\right) \Delta X_{s}\right\}
\end{equation*}%
In fact, due to summability of $\sum_{s\in (0,t]}\left\{ ...\right\} $, we
can pick $\eta >0$ such that%
\begin{equation*}
Z_{0,T}\approx _{2\varepsilon }\sum_{\left[ s^{\prime },t^{\prime }\right]
\in \mathcal{P}^{\prime }}f\left( Z_{s^{\prime }}\right) X\,_{s^{\prime
},t^{\prime }}+\sum_{s\in (0,t]:}\left\{ \phi (f\Delta
X_{s},Z_{s-})-Z_{s-}-f\left( Z_{s-}\right) \Delta X_{s}\right\}
\end{equation*}
and this is good enough to take the $\mathrm{(RRS)}\,\lim $ as $\left\vert 
\mathcal{P}^{\prime }\right\vert \rightarrow 0$. Going from finite variation 
$X$ to the rough case, is just more notational effort. After all, the rough
integral is a sort of (abstract) Riemann integral.
\end{proof}

\bigskip 

\bigskip We will need to the following corollary.

\begin{corollary}
\label{DEforminimaljumpextension}For a c\'{a}dl\'{a}g rough path $\mathbf{X}%
=1+X+\mathbb{X}=\exp \left( X+\mathbb{A}\right) \in \mathcal{W}_{\mathrm{g}%
}^{p}$ for $p\in \lbrack 2,3)$, the minimal jump extenstion $\mathbf{X}%
^{(n)} $ taking values in $G^{(n)}(\mathbb{R}^{d})$ satisfies the
Marcus-type differential equation 
\begin{equation}
\mathbf{X}_{t}^{(n)}=1+\int_{0}^{t}\mathbf{X}_{r-}^{(n)}\otimes d\mathbf{X}%
_{r}+\sum_{0<s\leq t}\mathbf{X}_{s-}^{(n)}\otimes \{\exp ^{(n)}(\log
^{(2)}\Delta \mathbf{X}_{s})-\Delta \mathbf{X}_{s}\}  \label{marcusRDE}
\end{equation}%
where the integral is understood as a rough integral and summation term is
well defined as absolutely summable series.
\end{corollary}

\begin{proof}
This follows from \eqref{equ:sigviaTC}.
\end{proof}

\section{Rough path stability}

We briefly discuss stability of rough integration and rough differential
equations. In the context of cadlag rough integration, Section \ref%
{roughintegrationwithjumps}, it is a natural to estimate $Z^{1}-Z^{2}$, in $p
$-variation norm, where%
\begin{equation*}
Z^{i}=\int_{0}Y^{i}d\mathbf{X}^{i}\text{ for }i=1,2\text{.}
\end{equation*}%
Now, the analysis presented in Section \ref{roughintegrationwithjumps}
adapts without difficulties to this situation. For instance, when $%
Y^{i}=F\left( X^{i}\right) $, one easily find%
\begin{equation*}
\left\vert Z^{1}-Z^{2}\right\vert _{p\text{-var}}\leq C_{F,M}(\left\vert
X_{0}^{1}-X_{0}^{2}\right\vert +\left\Vert X^{1}-X^{2}\right\Vert _{p\text{%
-var}}+\left\Vert \mathbb{X}^{1}\mathbb{-X}^{2}\right\Vert _{p\text{-var}})
\end{equation*}%
provided $F\in C^{2}$ and $\left\vert X_{0}^{i}\right\vert +\left\Vert
X^{i}\right\Vert _{p\text{-var}}+\left\Vert \mathbb{X}^{i}\right\Vert _{p%
\text{-var}}\leq M$. (The situation can be compared with \cite[Sec 4.4]%
{FrizHairer} where analogous estimate in the $\alpha $-H\"{o}lder setting.)

The situation is somewhat different in the case of Marcus type RDEs, $%
dY_{t}^{i}=f\left( Y_{t}^{i}\right) \diamond d\mathbf{X}_{t}^{i}$. The
observation here, quite possibly already contained implicitly in the works
of Williams \cite{williams}, is simply that the difference $Y^{1}-Y^{2}$, in
$p$-variation norm, is controlled, as above uniformly on bounded sets, by%
\begin{equation*}
\left\Vert \tilde{X}^{1}-\tilde{X}^{2}\right\Vert _{p\text{-var}}+\left\Vert 
\mathbb{\tilde{X}}^{1}\mathbb{-\tilde{X}}^{2}\right\Vert _{p\text{-var}}
\end{equation*}%
where $\mathbf{\tilde{X}}^{i}=(\tilde{X}^{i},\mathbb{\tilde{X}}^{i})\in 
\mathcal{C}_{g}^{p}$ is constructed from $\mathbf{X}^{i}\in \mathcal{W}%
_{g}^{p}$ as in Theorem \ref{minimaljumpextension}. (It should be possible
to rederive the convergence results of \cite{kpp} within this framework.)

\section{Rough versus stochastic integration}

Consider a $d$-dimensiona L\'{e}vy process $X_{t}$ enhanced with 
\begin{equation*}
\mathbb{X}_{s,t}:=\mathrm{(It\hat{o})}\int_{(s,t]}(X_{r-}-X_{s})\otimes
dX_{r}.
\end{equation*}

We show that rough integration against the It\^{o} lift actually yields
standard stochastic integral in It\^{o} sense. \ An immediate benefit, say
when taking $Y=f\left( X\right) $ with $f\in C^{2}$, is the universality of
the resulting stochastic integral, defined on a set of full measure
simultanously for all such integrands.

\begin{theorem}
\label{rough=ito} Let $X$ be a $d$-dimensional L\'{e}vy process, and
consider adapted processes $Y$ and $Y^{\prime }$ such that $(Y,Y^{\prime })$
is controlled rough path. Then It\^{o}- and rough integral coincide, 
\begin{equation*}
\int_{(0,T]}Y_{s-}dX_{s}=\int_{0}^{T}Y_{s-}d\mathbf{X}_{s}\hspace{2mm}a.s.
\end{equation*}
\end{theorem}

\begin{proof}[Proof of Theorem \protect\ref{rough=ito}]
By Theorem \ref{roughintegration}, there exist partitions $\mathcal{P}_{n}$
with 
\begin{equation*}
|S(\mathcal{P}_{n})-\int_{0}^{T}Y_{s-}d\mathbf{X}_{s}|\leq \frac{1}{n}
\end{equation*}%
where 
\begin{equation*}
S(\mathcal{P}_{n}):=\sum_{[s,t]\in \mathcal{P}_{n}}Y_{s}X_{s,t}+Y_{s}^{%
\prime }\mathbb{X}_{s,t}.
\end{equation*}%
Let $X_{t}=M_{t}+V_{t}$ be the L\'{e}vy Ito decomposition with martingale $M$
and bounded variation part $V$. Since $(V^{-},V)$, $(V^{-},M)$, $(M^{-},V)$
are compatible pairs we can choose the corresponding $\tau _{n}$ for $%
\epsilon =\frac{1}{n}$ from their compatibility. W.l.o.g., we can assume $%
\mathcal{P}_{n-1}\cup \tau _{n}\cup D_{n}\subset \mathcal{P}_{n}$, where $%
D_{n}$ is the $n$-th dyadic partition. We know from general stochastic
integration theory that, possibly along some subsequence, almost surely, 
\begin{equation*}
S^{\prime }(\mathcal{P}_{n})=\sum_{[s,t]\in \mathcal{P}_{n}}Y_{s}X_{s,t}%
\rightarrow \int_{(0,T]}Y_{s-}dX_{s}\hspace{3mm}%
\mbox{as $n \rightarrow
\infty $}
\end{equation*}%
Thus it suffices to prove that almost surely, along some subsequence, 
\begin{equation*}
S^{^{\prime \prime }}(\mathcal{P}_{n})=\sum_{[s,t]\in \mathcal{P}%
_{n}}Y_{s}^{\prime }\mathbb{X}_{s,t}\rightarrow 0\hspace{3mm}
\end{equation*}%
Now, 
\begin{equation*}
\mathbb{X}_{s,t}=\mathbb{M}_{s,t}+\mathbb{V}_{s,t}+%
\int_{(s,t]}(M_{r-}-M_{s})\otimes dV_{r}+\int_{(s,t]}(V_{r-}-V_{s})\otimes
dM_{r}
\end{equation*}%
Using a similar argument as in Theorem \ref{youngcompatible}, 
\begin{equation*}
\sum_{\lbrack s,t]\in \mathcal{P}_{n}}Y_{s}^{\prime }\biggl(\mathbb{V}%
_{s,t}+\int_{(s,t]}(M_{r-}-M_{s})\otimes
dV_{r}+\int_{(s,t]}(V_{r-}-V_{s})\otimes dM_{r}\biggl)\rightarrow 0
\end{equation*}

We are left to show that%
\begin{equation*}
\sum_{\lbrack s,t]\in \mathcal{P}_{n}}Y_{s}^{\prime }\mathbb{M}%
_{s,t}\rightarrow 0
\end{equation*}%
By the very nature of It\^{o} lift%
\begin{equation*}
\mathrm{Sym}\left( \mathbb{M}_{s,t}\right) =\frac{1}{2}M_{s,t}\otimes
M_{s,t}-\frac{1}{2}\left[ M,M\right] _{s,t}
\end{equation*}%
and it follows from standard (convergence to) quadratic variation results
for semimartingale (due to F\"{o}llmer \cite{F81}) that one is left with 
\begin{equation*}
\sum_{\lbrack s,t]\in \mathcal{P}_{n}}Y_{s}^{\prime }\mathbb{A}_{s,t}
\rightarrow 0
\end{equation*}%
where $\mathbb{A}_{s,t}=\mathrm{Anti}\left( \mathbb{M}_{s,t}\right) $. At
this point, let us first assume that $\left\vert Y^{\prime }\right\vert
_{\infty }\leq K$ uniformly in $\omega $. We know from Theorem \ref%
{theo:will} (or Corollary \ref{cor:will3} below) 
\begin{equation}
\mathbb{E}[|\mathbb{A}_{s,t}|^{2}]\leq C|t-s|^{2}  \label{equ:areaL2est}
\end{equation}%
and using standard martingale argument (orthogonal increment property), 
\begin{equation*}
\mathbb{E}[|\sum_{[s,t]\in \mathcal{P}_{n}}Y_{s}^{\prime }\mathbb{A}%
_{s,t}|^{2}]=\sum_{[s,t]\in \mathcal{P}_{n}}\mathbb{E}[|Y_{s}^{\prime }%
\mathbb{A}_{s,t}|^{2}]\leq K^{2}C\sum_{[s,t]\in \mathcal{P}_{n}}|t-s|^{2}=%
\mathcal{O}(|\mathcal{P}_{n}|)
\end{equation*}%
which implies, along some subsequence, almost surely, 
\begin{equation}
\sum_{\lbrack s,t]\in \mathcal{P}_{n}}Y_{s}^{\prime }\mathbb{A}%
_{s,t}\rightarrow 0  \label{equ:Ypzero}
\end{equation}%
Finally, for unbounded $Y^{\prime }$, introduce stopping times 
\begin{equation*}
T_{K}=\inf \{t\in \lbrack 0,T]:\sup_{s\in \lbrack 0,t]}|Y_{s}^{\prime }|\geq
K\}
\end{equation*}%
Similarly as in the previous case, 
\begin{equation*}
\mathbb{E}[|\sum_{[s,t]\in \mathcal{P}_{n},s\leq T_{K}}Y_{s}^{\prime }%
\mathbb{A}_{s,t}|^{2}]=\mathcal{O}(\mathcal{P}_{n})
\end{equation*}%
Thus almost surely on the event $\{T_{K}>T\}$ 
\begin{equation*}
\sum_{\lbrack s,t]\in \mathcal{P}_{n}}Y_{s}^{\prime }\mathbb{A}%
_{s,t}\rightarrow 0
\end{equation*}%
and sending $K\rightarrow \infty $ concludes the proof.
\end{proof}

\bigskip

We remark that the identification of rough with stochastic integrals is by
no means restricted to L\'{e}vy processes, and the method of proof here
obviously applies to semi-martingale situation. As a preliminary remark, one
can always drop the bounded variation part (and thereby gain integrability).
Then, with finite $p$-variation rough path regularity, for some $p<3$, of (It%
\^{o} -, by Proposition \ref{prop:semimartlift} equivalently: Stratonovich)
lift, see Section \ref{sec:semi}, the proof proceeds along the same lines
until the moment where one shows (\ref{equ:Ypzero}). For the argument then
to go through, one only needs 
\begin{equation*}
\sum_{\lbrack s,t]\in \mathcal{P}}\mathbb{E}[|\mathbb{A}_{s,t}|^{2}]%
\rightarrow 0\text{ as }|\mathcal{P}|\rightarrow 0,
\end{equation*}%
which follows from (\ref{equ:areaL2est}), an estimate which will be extended
to general classes of Markov jump processes in Section \ref{sec:Markov}.
That said, we note that much less than (\ref{equ:areaL2est}) is necessary
and clearly this has to exploited in a general semimartingale context.

\part{Stochastic processes as rough paths and expected signatures}

\section{L\'{e}vy processes}

\subsection{A L\'{e}vy--Kintchine formula and rough path regularity}

\label{sec:LKformula_and_RPreg}

In this section, we assume $\left( X_{t}\right) $ is a $d$-dimensional L\'{e}%
vy process with triplet $\left( a,b,K\right) $. The main insight of section
is that the expected signature is well-suited to study rough path
regularity. More precisely, we consider the Marcus canoncial signature $%
S=S\left( X\right) $, given as solution to%
\begin{eqnarray*}
dS &=&S\otimes \diamond dX \\
S_{0} &=&\left( 1,0,0\dots \right) \in T((\mathbb{R}^{d})).
\end{eqnarray*}%
With $S_{s,t}=S_{s}^{-1}\otimes S_{t}$ as usual this givens random
group-like elements 
\begin{equation*}
S_{s,t}=\left( 1,\mathbf{X}_{s,t}^{1},\mathbf{X}_{s,t}^{2},\dots \right)
=\left( 1,X_{s,t},\mathbb{X}_{s,t}^{\mathrm{M}},\dots \right)
\end{equation*}%
and then the step-$n$ signature of $X|_{[s,t]}$ by projection,%
\begin{equation*}
\mathbf{X}_{s,t}^{\left( n\right) }=\left( 1,\mathbf{X}_{s,t}^{1},\mathbf{%
\dots ,X}_{s,t}^{n}\right) \in G^{(N)}\left( \mathbb{R}^{d}\right) .
\end{equation*}%
The expected signature is obtained by taking component-wise expectation and
exists under a natural assumption on the tail behaviour of the L\'{e}vy
measure $K=K\left( dy\right) $. In fact, it takes ``L\'evy--Kintchine" form
as detailed in the followig theorem. We stress that fact that expected
signature contains significant information about \textit{the process} $%
(X_t:0\le t \le T)$, where a classical moment generating function of $X_T$
only carries information about the \textit{random variable} $X_T$.

\begin{theorem}[L\'evy--Kintchine formula]
\label{lk_formula}If the measure $K1_{|y|\geq 1}$ has moments up to order $N$%
, then 
\begin{equation*}
\mathbb{E}[\mathbf{X}_{0,T}^{\left( N\right) }]=\exp (CT)
\end{equation*}%
with tensor algebra valued exponent%
\begin{equation*}
C=\left( 0,b+\int_{\left\vert y\right\vert \geq 1}yK\left( dy\right) ,\frac{a%
}{2}+\int \frac{y^{\otimes 2}}{2!}K\left( dy\right) ,\dots ,\int \frac{%
y^{\otimes N}}{N!}K\left( dy\right) \right) \in T^{(N)}\left( \mathbb{R}%
^{d}\right) .
\end{equation*}%
In particular, if $KI_{|y|\geq 1}$ has finite moments of all orders, the
expected signature is given by 
\begin{equation*}
\mathbb{E}[S\left( X\right) _{0,T}]=\exp \left[ T\left( b+\frac{1}{2}a+\int
(\exp (y) -1-y1_{|y|<1})K\left( dy\right) \right) \right] \in T((\mathbb{R}%
^{d})).
\end{equation*}
\end{theorem}

The proof is based on the Marcus SDE $dS=S\otimes \diamond dX$ in $%
T^{(N)}\left( \mathbb{R}^{d}\right) $, so that $\mathbf{X}_{0,T}^{\left(
n\right) }=S$ and will be given in detail below. 
We note that Fawcett's formula \cite{fawcett, MR2052260, BaudoinBook} for
the expected value of iterated Stratonovich integrals of of $d$-dimensional
Brownian motion (with covariance matrix $a=I$ in the afore-mentioned
references)%
\begin{equation*}
\mathbb{E}[S\left( B\right) _{0,T}]=\mathbb{E}\left[ \left(
1,\int_{0<s<T}\circ dB,\int_{0<s<t<T}\circ dB\otimes \circ dB,\dots \right) %
\right] =\exp \left[ \frac{T}{2}a\right]
\end{equation*}%
is a special case of the above formula. Let us in fact give a (novel)
elementary argument for the validity of Fawcett's formula. The form $\mathbb{%
E}[S\left( B\right) _{0,T}]=\exp \left( TC\right) $ for some $C\in T((%
\mathbb{R}^{d}))$ is actually an easy consequence of independent increments
of Brownian motion. But Brownian scaling implies the $k^{\text{th}}$ tensor
level of $S(B)_{0,T}$ scales as $T^{k/2}$, which alreay implies that $C$
must be a pure $2$-tensor. The identification $C=a/2$ is then an immediate
computation. Another instructive case which allows for an elementary proof
is the case of when If $X$ is a compound Poisson process, i.e. $%
X_{t}=\sum_{i=1}^{N_{t}}J_{i}$ for some i.i.d. $d$-dimensional random
variables $J_{i}$ and $N_{t}$ a Poisson process with intensity $\lambda $.
In L\'evy terminology, one has triplet $(0,0,K)$ where $K$ is $\lambda$
times the law of $J_i$. Since jumps are to be traversed along straight
lines, Chen's rule implies 
\begin{equation*}
\mathbb{E}[S^{N}(X)_{0,1}|N_{1}=n]=\mathbb{E}[\exp (J_{1})\otimes ..\otimes
\exp (J_{n})]=\mathbb{E}[\exp (J_{1})]^{\otimes n}
\end{equation*}%
and thus 
\begin{equation*}
\mathbb{E}[S^{N}(X)_{0,1}]=\exp [(\lambda (\mathbb{E}\exp (J)-1)]
\end{equation*}%
which gives, with all integrations over $\mathbb{R}^{d}$, 
\begin{eqnarray*}
C &=&\int (\exp (y)-1)K(dy).
\end{eqnarray*}%
%
%
%
%
%
%
%
%
%
%
%

Before turning to the proof of Theorem \ref{lk_formula} we give the
following application. It relies on the fact that the expected signature
allows to extract easily information about stochastic area.

\begin{corollary}
\label{cor:will3}Let $X$ be a $d$-dimensional L\'{e}vy process. Then, for
any $p>2$, a.s. 
\begin{equation*}
\left( X,\mathbb{X}^{\mathrm{M}}\right) \in \mathcal{W}_{\mathrm{M}%
}^{p}\left( \left[ 0,T\right] ,\mathbb{R}^{d}\right) \text{ a.s. }
\end{equation*}
We call the resulting Marcus like (geometric) rough path the Marcus lift of $%
X$.
\end{corollary}

\begin{proof}
W.l.o.g. all jumps have size less than $1$. (This amounts to drop a bounded
variation term in the It\^{o}-L\'{e}vy decomposition. This does not affect
the $p$-variation sample path properties of $X$, nor - in view of basic
Young (actually Riemann--Stieltjes) estimates - those of $\mathbb{X}^{%
\mathrm{M}}$). We establish the desired rough path regularity as application
of Proposition \ref{prop:SMRP} which requires as to show%
\begin{eqnarray*}
\mathbb{E}\left\vert X_{s,t}\right\vert ^{2} &\lesssim &\left\vert
t-s\right\vert \\
\mathbb{E}\left\vert \mathbb{A}_{s,t}\right\vert ^{2} &\lesssim &\left\vert
t-s\right\vert ^{2}.
\end{eqnarray*}%
While the first estimate is immediate from the $L^{2}$-isometry of
stochastic integrals against Poisson random measures\ (drift and Brownian
component obviously pose no problem), the second one is more subtle in
nature and indeed fails - in presence of jumps - when $\mathbb{A}$ is
replaced by the full second level $\mathbb{X}^{\mathrm{M}}$. (To see this
take, $d=1$ so that $\mathbb{X}_{s,t}^{\mathrm{M}}=X_{s,t}^{2}/2$ and note
that even for standard Poisson process $\mathbb{E}\left\vert
X_{s,t}\right\vert ^{4}\lesssim \left\vert t-s\right\vert $ but not $%
\left\vert t-s\right\vert ^{2}.$)

It is clearly enough to consider $\mathbb{A}_{s,t}^{i,j}$ for indices $i\neq
j$. It is enough to work with $S^{4}\left( X\right) =:\mathbf{X}$. Using the
geometric nature of $\mathbf{X}$, by using shuffle product formula,%
\begin{eqnarray*}
\left( \mathbb{A}_{s,t}^{i,j}\right) ^{2} &=&\frac{1}{4}\left( \mathbb{X}%
_{s,t}^{i,j}-\mathbb{X}_{s,t}^{j,i}\right) \left( \mathbb{X}_{s,t}^{i,j}-%
\mathbb{X}_{s,t}^{j,i}\right) \\
&=&\mathbf{X}_{s,t}^{iijj}-\mathbf{X}_{s,t}^{ijji}-\mathbf{X}_{s,t}^{jiij}+%
\mathbf{X}_{s,t}^{jjii}
\end{eqnarray*}%
On the other hand, 
\begin{equation*}
\mathbb{E}\mathbf{X}_{s,t}=\exp \left[ \left( t-s\right) C\right] =1+\left(
t-s\right) C+O\left( t-s\right) ^{2}
\end{equation*}%
so that it is enough to check that $C^{iijj}-C^{ijji}-C^{jiij}+C^{jjii}=0$.
But this is obvious from the symmetry of 
\begin{equation*}
\pi _{4}C=\frac{1}{4!}\int y^{\otimes 4}K\left( dy\right) .
\end{equation*}
\end{proof}

We now given the proof of the L\'evy--Kintchine formula for the expected
signature of L\'evy--processes. We first state some lemmas required. \newline

The following lemma, a generalization of \cite[Ch. 1, Thm. 38]{Protter}, is
surely well-known but since we could not find a precise reference we include
the short proof.

\begin{lemma}
\label{cruciallemma} Let $F_s $ be a c\'agl\'ad adapted process with $%
\sup_{0< s\leq t }\mathbb{E}[|F_s|] < \infty $ and $g$ be a measurable
function with $|g(x)| \leq C|x|^k $ for some $C> 0, k\geq 2$ and $g \in
L^1(K)$. Then 
\begin{equation*}
\mathbb{E}[\sum_{0< s\leq t }F_s g(\Delta X_s) ] = \int_0^t \mathbb{E}%
[F_s]ds\int_{\mathbb{R}^d}g(x)K(dx)
\end{equation*}
\end{lemma}

\begin{proof}
At first we prove the following, 
\begin{equation}  \label{bound}
\mathbb{E}[\sum_{0 <s \leq t }|F_s||g(\Delta X_s)|] \leq t ||g||_1 \sup_{0<
s \leq t} \mathbb{E}[|F_s|]
\end{equation}
To this end, w.l.o.g, we can assume $g $ vanishes in a neighbourhood of
zero. The general case follows by an application of Fatou's lemma. Also, it
is easy to check the inequality when $F_s $ is a simple predictable process.
For general $F_s$, we choose a sequence of simple predictable process $F_s^n
\rightarrow F_s$ pointwise. Since there are only finitely many jumps away
from zero, we see that 
\begin{equation*}
\sum_{0< s \leq t }|F_s^n||g(\Delta X_s)| \rightarrow \sum_{0< s \leq t
}|F_s||g(\Delta X_s)| \hspace{2mm} a.s.
\end{equation*}
and the claim follows again by Fatou's lemma.

Now, define $\bar{g}= \int_{\mathbb{R}^d \setminus 0}g(x)K(dx) $ and $M_t =
\sum_{0< s \leq t}g(\Delta X_s) - t\bar{g} $. Then it is easy to check that $%
M_t$ is a martingale. Also, 
\begin{equation*}
N_t := \int_{(0,t]} F_sdM_s = \sum_{0< s \leq t }F_sg(\Delta X_s) - \bar{g}%
\int_0^t F_sds
\end{equation*}
is a local martingale. From \eqref{bound}, $\mathbb{E}[\sup_{0< s\leq t }
|N_s|] < \infty$. So, $N_t$ is a martingale, which thereby implies that $%
\mathbb{E}[N_t] = 0 $ finishing the proof.
\end{proof}

\begin{lemma}
\label{finiteexpectation} If the measure $K\mathbb{I}_{|y|\geq 1}$ has
moments upto order $N$ then with $S_{t}=S^{N}(X)_{0,t}$, 
\begin{equation*}
\mathbb{E}[\sup_{0<s\leq t}|S_{s}|]<\infty
\end{equation*}
\end{lemma}

\begin{proof}
We will prove it by induction on $N$. For $N=1$, $S_{t}=\mathbf{1}+X_{t}$,
and the claim follows from the classical result that $\mathbb{E}%
[\sup_{0<s\leq t}|X_{s}|]<\infty $ iff $KI_{|y|\geq 1}$ has finite first
moment. Now, note that 
\begin{equation*}
S_{t}=1+\int_{0}^{t}\pi _{N,N-1}(S_{r-})\otimes dX_{r}+\frac{1}{2}%
\int_{0}^{t}\pi _{N,N-1}(S_{r})\otimes adr+\sum_{0<s\leq t}\pi
_{N,N-1}(S_{r-})\otimes \{e^{\Delta X_{s}}-\Delta X_{s}-1\}
\end{equation*}%
where $\pi_{N, N-1} : T_1^N(\mathbb{R}^{d}) \rightarrow T_1^{N-1}(\mathbb{R}^{d}) $ is the
projection map. From induction hypothesis and lemma \eqref{cruciallemma},
last two terms on right hand side has finite expectation in supremum norm.
Using L\'{e}vy-Ito decomposition, 
\begin{align*}
\int_{0}^{t}\pi _{N,N-1}(S_{r-})\otimes dX_{r}=\int_{0}^{t}\pi
_{N,N-1}(S_{r-})\otimes dM_{r}+& \int_{0}^{t}\pi _{N,N-1}(S_{r-})\otimes bdr
\\
+& \sum_{0<s\leq t}\pi _{N,N-1}(S_{r-})\otimes \Delta X_{s} {1}_{|\Delta
X_{s}|\geq 1}
\end{align*}

where $M$ is the martingale. Again by induction hypothesis and Lemma \ref%
{cruciallemma}, last two terms are of finite expectation in supremum norm.
Finally,%
\begin{equation*}
L_t = \int_{0}^{t}\pi _{N,N-1}(S_{r-})\otimes dM_{r}
\end{equation*}%
is a local martingale. By Burkholder-Davis-Gundy inequality and noting that 
\begin{equation*}
[M]_t = at + \sum_{0<s\leq t } (\Delta X_s)^2 {1}_{|\Delta X_s| < 1}
\end{equation*}%
we see that 
\begin{align*}
\mathbb{E}[\sup_{0 < s \leq t } |L_s|] \lesssim & \mathbb{E}[ \biggl\{ %
\int_0 ^t|\pi _{N,N-1}(S_{r-})|^2d[M]_r \biggr\}^{\frac{1}{2}} ] \\
\lesssim & \mathbb{E}[ \biggl\{ \int_0 ^t|\pi _{N,N-1}(S_{r-})|^2dr \biggr\}%
^{\frac{1}{2}} ]+ \mathbb{E}[ \bigg\{\sum_{0 < r \leq t}{|\pi
_{N,N-1}(S_{r-})|^2 |\Delta X_r|^2 1_{|\Delta X_r | <1 }} \biggr\}^{\frac{1}{%
2}} ] \\
\lesssim & \mathbb{E}[ \sup_{r \leq t}| \pi _{N,N-1}(S_{r-})| ] + \mathbb{E}
[ \biggl\{ \sup_{r \leq t } |\pi _{N,N-1}(S_{r-}|\biggr\}^{\frac{1}{2}} %
\biggl\{ \sum_{0 < r \leq t}|\pi _{N,N-1}(S_{r-})| |\Delta X_r|^2 1_{|\Delta
X_r | <1} \biggr\}^{\frac{1}{2}}] \\
\lesssim & \mathbb{E}[ \sup_{r \leq t}| \pi _{N,N-1}(S_{r-})| ] + \mathbb{E}
[ \sup_{r \leq t } |\pi _{N,N-1}(S_{r-}|] + \mathbb{E}[ \sum_{0 < r \leq
t}|\pi _{N,N-1}(S_{r-})| |\Delta X_r|^2 1_{|\Delta X_r | <1} ]
\end{align*}
where in the last line, we have used $\sqrt{ab} \lesssim a + b$. Again by
induction hypothesis and \eqref{cruciallemma}, we conclude that 
\begin{equation*}
\mathbb{E}[\sup_{0 < s \leq t } |L_s|] < \infty
\end{equation*}
finishing the proof.
\end{proof}

\begin{proof}
(Theorem \ref{lk_formula}) As before, 
\begin{equation*}
S_{t}=1+\int_{0}^{t}S_{r-}\otimes dM_{r}+\int_{0}^{t}S_{r}\otimes (b+\frac{a%
}{2})dr+\sum_{0<s\leq t}S_{s-}\otimes \{e^{\Delta X_{s}}-\Delta
X_{s}1_{|\Delta X_{s}|<1}-1\}
\end{equation*}%
By Lemma \ref{finiteexpectation} below $\int_{0}^{t}S_{r-}\otimes dM_{r}$ is
indeed a martingale. Also note that $S_{t}$ has a jump iff $X_{t}$ has a
jump, so that almost surely $S_{t-}=S_{t}$. Thanks to Lemma \ref%
{cruciallemma} below 
\begin{equation*}
\mathbb{E}S_{t}=1+\int_{0}^{t}\mathbb{E}S_{r}\otimes (b+\frac{a}{2}%
)dr+\int_{0}^{t}\mathbb{E}S_{r}dr\otimes \int (e^{y}-y1_{|y|<1}-1)K(dy)
\end{equation*}%
and solving this linear ODE in $T_{1}^{N}(\mathbb{R}^{d})$ completes the
proof.
\end{proof}

\subsection{L\'{e}vy rough paths}

Corollary \ref{cor:will3} tells us that the Marcus lift of some $d$%
-dimensional L\'{e}vy process $X$ has sample paths of finite $p$-variation
with respect to the CC\ norm on $G^{(2)}$, that is 
\begin{equation*}
\mathbf{X}^{\mathrm{M}}:=(1,X,\mathbb{X}^{\mathrm{M}})\in W_{\mathrm{g}%
}^{p}\left( \left[ 0,T\right] ,G^{(2)}\left( \mathbb{R}^{d}\right) \right) .
\end{equation*}%
It is clear from the nature of Marcus integration that $\mathbf{X}_{s,t}^{%
\mathrm{M}}$ is $\sigma \left( X_{r}:r<s\leq t\right) $-measurable. It
easily follows that $\mathbf{X}^{\mathrm{M}}$ is a Lie group valued L\'{e}vy
process, with values in the Lie group $G^{(2)}\left( \mathbb{R}^{d}\right) $%
, and in fact a L\'{e}vy rough path in the following sense

\begin{definition}
\label{def:LRP}Let $p\in \lbrack 2,3)$. A $G^{(2)}(\mathbb{R}^{d})$-valued
process $\left( \mathbf{X}\right) $ with (c\'{a}dl\'{a}g) rough sample paths 
$\mathbf{X}\left( \omega \right) \in W_{\text{\textrm{g}}}^{p}$ a.s. (on any
finite time horizon) is called \textbf{L\'{e}vy }$p$-\textbf{rough path} iff
it has stationary independent left-increments\ (given by $\mathbf{X}%
_{s,t}\left( \omega \right) =\mathbf{X}_{s}^{-1}\otimes \mathbf{X}_{t}$).
\end{definition}

Not every L\'{e}vy rough path arises as Marcus lift of some $d$-dimensional L%
\'{e}vy process. For instance, the \textit{pure area Poisson process }from
Example\textit{\ \ref{ex:pacCPP}} and then \textit{the non-canonical
Brownian rough path (\textquotedblleft Brownian motion in a magnetic field")}
from Example \ref{ex:ncBRP} plainly do not arise from iterated Marcus
integration.\textit{\ }

Given any L\'{e}vy rough path $\mathbf{X}=\left( 1,X,\mathbb{X}\right) $, it
is clear that its projection $X=\pi _{1}\left( \mathbf{X}\right) $ is a
classical L\'{e}vy process on $\mathbb{R}^{d}$ which then admits, thanks to
Corollary \ref{cor:will3}, a L\'{e}vy rough path lift $\mathbf{X}^{\mathrm{M}%
} $. This suggests the following terminology. We say that $\mathbf{X}$ is a 
\textbf{canoncial} \textbf{L\'{e}vy rough path }if $\mathbf{X}$ and $\mathbf{%
X}^{\mathrm{M}}$ are indistinguisable, call $\mathbf{X}$ a \textbf{%
non-canonical} \textbf{L\'{e}vy rough path }otherwise.

Let us also note that there are $G^{(2)}(\mathbb{R}^{d})$-valued L\'{e}vy
processes which are not \textbf{L\'{e}vy }$p$-\textbf{rough path} in the
sense of the above definition, for the may fail to have finite $p$-variation
for $p\in \lbrack 2,3)$ (and thereby missing the in rough path theory
crucial link between regularity and level of nilpotency, $\left[ p\right] =2$%
.) To wit, \textit{area-valued Brownian motion} 
\begin{equation*}
\mathbf{X}_{t}:=\exp ^{\left( 2\right) }\left( B_{t}[e_{1},e_{2}]\right) \in
G^{(2)}(\mathbb{R}^{d})
\end{equation*}%
is plainly a $G^{(2)}(\mathbb{R}^{d})$-valued L\'{e}vy processes, but 
\begin{equation*}
\sup_{\mathcal{P}}\sum_{[s,t]\in \mathcal{P}}\left\vert \left\vert \mathbf{X}%
_{s,t}\right\vert \right\vert _{CC}^{p}\sim \sup_{\mathcal{P}}\sum_{[s,t]\in 
\mathcal{P}}\left\vert B_{s,t}\right\vert ^{p/2}<\infty
\end{equation*}%
if and only if $p>4$.

\begin{remark}
One could define $G^{4}(\mathbb{R}^{d})$-valued L\'{e}vy rough
paths, with $p$-variation regularity where $p\in \lbrack 4,5)$, an example
of which is given by area-valued Brownian motion. But then again not every $%
G^{4}(\mathbb{R}^{d})$-valued L\'{e}vy process will be a $G^{4}(\mathbb{R}%
^{d})$-valued L\'{e}vy rough path and so on. In what follows we remain in
the step-$2$ setting of Definition \ref{def:LRP}.
\end{remark}

We now characterize L\'{e}vy rough paths among $G^{(2)}(\mathbb{R}^{d})$%
-valued L\'{e}vy processes, themselves characterized by Hunt's theory of Lie
group valued L\'{e}vy porcesses, cf. Section \ref{sec:Levy}. To this end,
let us recall $G^{(2)}\left( \mathbb{R}^{d}\right) =\exp \left( \mathfrak{g}%
^{2}\left( \mathbb{R}^{d}\right) \right) $, where%
\begin{equation*}
\mathfrak{g}^{(2)}\left( \mathbb{R}^{d}\right) =\mathbb{R}^{d}\oplus
so\left( d\right) .
\end{equation*}%
For $g\in G^{(2)}\left( \mathbb{R}^{d}\right) $, let $\left\vert
g\right\vert $ be the Euclidean norm of $\log g\in \mathfrak{g}^{(2)}\left( 
\mathbb{R}^{d}\right) $. With respect to the canoncial basis, any element in 
$\mathfrak{g}^{(2)}\left( \mathbb{R}^{d}\right) $ can be written as in
coordinates as $\left( x^{v}\right) _{v\in J}$ where 
\begin{equation*}
J:=\{i:1\leq i\leq d\}\cup \{jk:1\leq j<k\leq d\}.
\end{equation*}%
Write also%
\begin{equation*}
I:=\{i:1\leq i\leq d\}.
\end{equation*}

\begin{theorem}
\label{theo:UnifiedLK_RP}Every $G^{(2)}\left( \mathbb{R}^{d}\right) $-valued
L\'{e}vy process $\left( \mathbf{X}\right) $ is characterized by a triplet $%
\left( \mathbf{a},\mathbf{b},\mathbf{K}\right) $ with 
\begin{eqnarray*}
\mathbf{a} &=&\left( a^{v,w}:v,w\in J\right) , \\
\mathbf{b} &\mathbf{=}&\left( b^{v}:v\in J\right) , \\
\mathbf{K} &\mathbf{\in }&\mathcal{M}\left( G^{(2)}\left( \mathbb{R}%
^{d}\right) \right) :\int_{G^{(2)}(\mathbb{R}^{d})}\left( \left\vert
g\right\vert ^{2}\wedge 1\right) \mathbf{K}\left( dg\right) .
\end{eqnarray*}%
The projection $X:=\pi _{1}\left( \mathbf{X}\right) $ is a standard $d$%
-dimensional L\'{e}vy process, with triplet%
\begin{equation*}
\left( a,b,K\right) :=\left( \left( a^{i,j}:i,j\in I\right) ,\left(
b^{k}:k\in I\right) ,\left( \pi _{1}\right) _{\ast }\mathbf{K}\right)
\end{equation*}%
where $K$ is the pushforward of $\mathbf{K}$ under the projection map. Call $%
\left( \mathbf{a},\mathbf{b},\mathbf{K}\right) $ \textbf{enhanced L\'{e}vy
triplet}, and $\mathbf{X}$ \textbf{enhanced L\'{e}vy process}.
\end{theorem}

\begin{proof}
This is really a special case of Hunt's theory. Let us detail, however, an
explicit construction which we will be useful later on:\ every $%
G^{(2)}\left( \mathbb{R}^{d}\right) $-valued L\'{e}vy process $\mathbf{X}$
(started at $1$) can be written as in terms of a $\mathfrak{g}^{(2)}\left( 
\mathbb{R}^{d}\right) $-valued (standard) L\'{e}vy process $(X,Z)$, started
at $0$, as \ 
\begin{equation*}
\mathbf{X}_{t}=\exp \left( X_{t},\mathbb{A}_{t}+Z_{t}\right)
\end{equation*}%
where $\mathbb{A}_{t}=\mathbb{A}_{0,t}$ is the stochastic area associated to 
$X$. \ Indeed, for $v,w\in J$, write $x=\left( x^{v}\right) $ for a generic
element in $\mathfrak{g}^{(2)}$ and then 
\begin{equation*}
\left( \left( a^{v,w}\right) ,\left( b^{v}\right) ,\mathfrak{K}\right)
\end{equation*}%
for the L\'{e}vy-triplet of $\left( X,Z\right) $. Of course, $X$ and $Z$ are
also ($\mathbb{R}^{d}$- and $\mathfrak{so}\left( d\right) $-valued) L\'{e}vy
process with triplets 
\begin{equation*}
\left( \left( a^{i,j}\right) ,\left( b^{i}\right) ,K\right) \text{ and }%
\left( \left( a^{jk,lm}\right) ,\left( b^{jk}\right) ,\mathbb{K}\right) ,
\end{equation*}%
respectively, where $K$ and $\mathbb{K}$ are the image measures of $%
\mathfrak{K}$ under the obvious projection maps, onto $\mathbb{R}^{d}$ and $%
\mathfrak{so}\left( d\right) $, respectively. Define also the image measure
under $\exp $, that is $\mathbf{K=}\exp _{\ast }$ $\mathbb{K}$. It is then
easy to see that $\mathbf{X}$ is a L\'{e}vy process in the sense of Hunt
(cf. Section \ref{sec:Levy}) with triplet $\left( \mathbf{a},\mathbf{b},%
\mathbf{K}\right) $. \ Conversely, given $\left( \mathbf{a},\mathbf{b},%
\mathbf{K}\right) $, one constructs a $\mathfrak{g}^{(2)}\left( \mathbb{R}%
^{d}\right) $-valued L\'{e}vy process $\left( X,Z\right) $ with triplet $%
\left( \left( a^{v,w}\right) ,\left( b^{v}\right) ,\log _{\ast }\mathbf{K}%
\right) $ and easily checks that the $\exp \left( X,\mathbb{A}+Z\right) $ is
the desired $G^{(2)}$-valued L\'{e}vy process.
\end{proof}

\bigskip

Recall that definition of the Carnot--Caratheodory (short:\ CC) norm on $%
G^{(2)}\left( \mathbb{R}^{d}\right) $ from Section \ref{sec:gRP}. The
definition below should be compared with the classical definition of
Blumenthal--Getoor (short: BG) index.

\begin{definition}
Given a L\'{e}vy measure $\mathbf{K}$ on the Lie group $G^{(2)}\left( 
\mathbb{R}^{d}\right) $, call
\begin{equation*}
\beta :=\inf \left\{ q >0 :\int_{G^{(2)}(\mathbb{R}^{d})}\left(
||g||_{CC}^{q}\wedge 1\,\right) \mathbf{K}\left( dg\right) \right\}  
\end{equation*}
the Carnot--Caratheodory Blumenthal--Getoor (short: CCBG) index.
\end{definition}

Unlike the classical BG index, the CCBG index is not restricted
to $[0,2]$.

\begin{lemma}
The CCBG index takes values in $\left[ 0,4\right] $.
\end{lemma}

\begin{proof}
Set $\log \left( g\right) =x+a\in \mathbb{R}^{d}\oplus \mathfrak{so}\left(
d\right) $. Then%
\begin{equation*}
||g||_{CC}^{q}\asymp \sum_{i}\left\vert x^{i}\right\vert
^{q}+\sum_{j<k}\left\vert a^{jk}\right\vert ^{q/2}.
\end{equation*}%
By the very nature of $\mathbf{K}$, it integrates $\left\vert
x^{i}\right\vert ^{2}$ and $\left\vert a^{jk}\right\vert ^{2}$ and hence $%
\beta \leq 4$. (The definition of CC Blumenthal--Getoor extends immediately
to $G^{(N)}\left( \mathbb{R}^{d}\right) $, in which case $\beta \leq 2N$.)
\end{proof}

\begin{theorem}
\label{theo:LRPChar}Consider a $G^{(2)}\left( \mathbb{R}^{d}\right) $-valued
L\'{e}vy process $\mathbf{X}$ with enhanced triplet $\left( \mathbf{a},%
\mathbf{b},\mathbf{K}\right) $. Assume \newline
(i) the sub-ellipticity condition 
\begin{equation*}
a^{v,w}\equiv 0\text{ unless }v,w\in I=\left\{ i:1\leq i\leq d\right\} ;
\end{equation*}%
(ii)\ the following bound on the CCBG index%
\begin{equation*}
\beta <3.
\end{equation*}%
Let $p\in (2,3)$. Then a.s. $\mathbf{X}$ is a L\'{e}vy $p$-rough path if $%
p>\beta $ and this condition is sharp.
\end{theorem}

\begin{proof}
Set $\log \left( g\right) =x+a\in \mathbb{R}^{d}\oplus \mathfrak{so}\left(
d\right) $. Then%
\begin{equation*}
||g||_{CC}^{2\rho }\asymp \sum_{i}\left\vert x^{i}\right\vert ^{2\rho
}+\sum_{j<k}\left\vert a^{jk}\right\vert ^{\rho }.
\end{equation*}%
Let $K$ denote the image measure of $\mathbf{K}$ under the projection map $%
g\mapsto x\in \mathbb{R}^{d}$. Let also $\mathbb{K}$ denote the image
measure under the map $g\mapsto a\in \mathfrak{so}\left( d\right) $. Since $%
\mathbf{K}$ is a L\'{e}vy measure on $G^{(2)}(\mathbb{R}^{d})$, we know that 
\begin{equation}
\int_{\mathfrak{so}\left( d\right) }\left( \left\vert a\right\vert ^{\rho
}\wedge 1\right) \mathbb{K}\left( da\right) <\infty .  \label{BCfora}
\end{equation}%
whenever $\beta <2\rho <3$. We now show that $\mathbf{X}$ enjoys $p$%
-variation. We have seen in the proof of Theorem \ref{theo:UnifiedLK_RP}
that any such L\'{e}vy process can be written as 
\begin{equation*}
\log \mathbf{X}=\left( X,\mathbb{A}+Z\right)
\end{equation*}%
where $X$ is a $d$-dimensional L\'{e}vy process with triplet%
\begin{equation*}
\left( \left( a^{i,j}\right) ,\left( b^{i}\right) ,K\right)
\end{equation*}%
with $\mathfrak{so}\left( d\right) $-valued area $\mathbb{A}=\mathbb{A}%
_{s,t} $ and a $\mathfrak{so}\left( d\right) $-valued L\'{e}vy process $Z$
with triplet%
\begin{equation*}
\left( 0,\left( b^{jk}\right) ,\mathbb{K}\right) .
\end{equation*}

We know that $\mathbb{E[}\left\vert X_{s,t}\right\vert ^{2}]\lesssim
\left\vert t-s\right\vert $ and $\mathbb{E[}\left\vert A_{s,t}\right\vert
^{2}]\lesssim \left\vert t-s\right\vert ^{2}$ and so, for $\left\vert
t-s\right\vert \leq h$, 
\begin{eqnarray*}
\mathbb{P}\left( \left\vert X_{s,t}\right\vert >a\right) &\leq &\frac{h}{%
a^{2}} \\
\mathbb{P(}\left\vert A_{s,t}\right\vert ^{1/2} &>&a)\leq \frac{h}{a^{2}}.
\end{eqnarray*}%
On the other hand,%
\begin{equation*}
\mathbb{P(}\left\vert Z_{s,t}\right\vert ^{1/2}>a)\leq \frac{1}{a^{2\rho }}%
\mathbb{E}\left( \left\vert Z_{s,t}\right\vert ^{\rho }\right) \sim \frac{h}{%
a^{2\rho }}\int_{\mathfrak{so}\left( d\right) }\left( \left\vert
a\right\vert ^{\rho }\wedge 1\right) \mathbb{K}\left( da\right)
\end{equation*}%
and so%
\begin{equation*}
\mathbb{P}\left( \left\vert \left\vert \mathbf{X}_{s,t}\right\vert
\right\vert _{CC}>a\right) \lesssim \frac{h}{a^{2\rho \vee 2}}.
\end{equation*}%
It then follows from Manstavicius' criterion, cf. Section \ref{sec:Man},
applied with $\beta =1,\gamma =2\rho \vee 2$, that $\mathbf{X}$ has indeed $%
p $-variation, for any $p>2\rho \vee 2$, and by taking the infimum, for all $%
p>\beta \vee 2$.

It remains to see that the {conditions are sharp. Indeed, if the
sup-ellipticity condition is violated, say }if $a^{v,w}\neq 0$ for some $%
v=jk $, say, this means (Brownian) diffusity (and hence finite $2^{+}$- but
not $2 $-variation) in direction $\left[ e_{j},e_{k}\right] \in \mathfrak{so}%
\left( d\right) $. As a consequence, $\mathbf{X}$ has $4^{+}$-variation (but
not $4$-variation), in particular, it fails to have $p$-variation for some $%
p\in \lbrack 2,3)$. Similarly, if one considers an $\alpha $-stable process
in direction $\left[ e_{j},e_{k}\right] $, with well-known finite $\alpha
^{+}$- but not $\alpha $-variation, we see that the condition $p>\beta $
cannot be weakened.
\end{proof}

\subsection{Expected signatures for L\'{e}vy rough paths}

Let us return to the Theorem \ref{lk_formula}, where we computed, subject to
suitable integrability assumptions of the L\'{e}vy measure, the expected
signature of a L\'{e}vy process, lifted by means of \textquotedblleft
Marcus" iterated integrals. There we found that the expected signature over $%
\left[ 0,T\right] $ takes L\'{e}vy--Kintchine form 
\begin{equation*}
\mathbb{E}[\mathbf{X}_{0,T}]=\exp \Bigl\{T\Bigl(b+\frac{a}{2}+\int_{\mathbb{R%
}^{d}}(\exp (y)-1-y\mathbb{I}_{|y|<1})K(dy)\Bigr)\Bigl\}
\end{equation*}%
for some symmetric, positive semidefinite matrix $a$, a vector $b$ and a L%
\'{e}vy measure $K$, provided $K\mathbb{I}_{|y|\geq 1}$ has moments of all
orders. In absence of a drift $b$ and jumps, the formula degenerate to
Fawcett's form, that is 
\begin{equation*}
\exp (T\frac{a}{2})
\end{equation*}%
for a \textit{symmetric} $2$-tensor $a$. Let us present two examples of L%
\'{e}vy rough paths, for which the expected signature is computable and 
\textit{different} from the above form.

\begin{example}
\label{Ex41} We return to the non-canonical Brownian rough path $\mathbf{B}^{%
\mathrm{m}}$, the zero-mass limit of physcial Brownain motion in a magnetic
field, as discucssed in Example \ref{ex:ncBRP}. The signature\textbf{\ }$%
S=S^{\mathrm{m}}$ is then given by Lyons' extension theorem applied to $%
\mathbf{B}^{\mathrm{m}}$, or equivalently, by solving the following rough
differential equation%
\begin{equation*}
dS_{t}=S_{t}\otimes d\mathbf{B}_{t}^{\mathrm{m}}\left( \omega \right)
,\,S_{0}=1
\end{equation*}%
In \cite{Friz--Gassiat--Lyons} it was noted that the expected signature
takes the Fawcett form, 
\begin{equation*}
\mathbb{E}[S_{0,T}^{\mathrm{m}}]=exp\Bigl\{T\frac{\tilde{a}}{2}\Big\}
\end{equation*}%
but now for a not necessarily symmetric $2$-tensor $\tilde{a}$, the
antisymmetric part of which depends on the charge of the particle and the
strength of the magnetic field.
\end{example}

\begin{example}
Consider the \textit{pure area Poisson process }from Example \ref{PACPRP}.
Fix some $\mathfrak{a}\in \mathfrak{so}\left( d\right) $ and let $\left(
N_{t}\right) $ be standard Poisson process, rate $\lambda >0$. We set%
\begin{equation*}
\mathbf{X}_{t}:=\otimes _{i=1}^{N_{t}}\exp ^{\left( 2\right) }\left( 
\mathfrak{a}\,\right) \in G^{(2)}(\mathbb{R}^{d});
\end{equation*}%
noting that the underlying path is trivial, $X=\pi _{1}\left( \mathbf{X}%
\right) \equiv 0$ and clearly $\mathbf{X}$ is a non-Marcus L\'{e}vy $p$%
-rough path, any $p\geq 2$. The signature of $\mathbf{X}$ is by definition
the minimal jump extension of $\mathbf{X}$ as provided by Theorem \ref%
{minimaljumpextension}. We leave it as easy exercise to the reader to see
that the signature $S$ is given by 
\begin{equation*}
S_{t}=\otimes _{i=1}^{N_{t}}\exp \left( \mathfrak{a}\,\right) \in T((\mathbb{%
R}^{d})).
\end{equation*}%
With due attention to the fact that computations take places in the
(non-commutative) tensor algebra, we then compute explicitly%
\begin{eqnarray*}
\mathbb{E}S_{T} &=&\sum_{k\geq 0}e^{\mathfrak{a}k}e^{-\lambda T}\left(
\lambda T\right) ^{k}/k! \\
&=&e^{-\lambda T}\sum_{k\geq 0}\left( \lambda Te^{\mathfrak{a}}\right)
^{k}/k! \\
&=&\exp [\lambda T(e^{\mathfrak{a}}-1\mathbf{)].}
\end{eqnarray*}%
Note that the jump is not described by a L\'{e}vy-measure on $\mathbb{R}^{d}$
but rather by a Dirac measure on $G^{(2)}$, assigning unit mass to $\exp 
\mathfrak{a}\in G^{(2)}$.
\end{example}

We now give a general result that covers all these examples. Indeed, Example %
\ref{Ex41} is precisely the case of $\tilde{a}=a+2\mathfrak{b}$ with
antisymmetric $\mathfrak{b}=(b^{j,k})\neq 0$, and symmetric $a=\left(
a^{i,j}\right) $. As for example (ii), everything is trivial but $\mathbf{K}$%
, which assigns unit mass to the element $\exp \mathfrak{a}$.)

\begin{theorem}
Consider a L\'{e}vy rough path $\mathbf{X}$ with enhanced triplet $\left( 
\mathbf{a},\mathbf{b},\mathbf{K}\right) $. Assume that $\mathbf{K}1_{\left\{
|g|>1\right\} }$ integrates all powers of $\left\vert g\right\vert
:=\left\vert \log g\right\vert _{\mathbb{R}^{d}\oplus \mathfrak{so}\left( d\right) }$%
. Them the signature of $\mathbf{X}$, by definition the minimal jump
extension of $\mathbf{X}$ as provided by Theorem \ref{minimaljumpextension},
is given by {\ } 
\begin{equation}
\mathbb{E}S_{0,T}=\exp \left[ T\left( \frac{1}{2}%
\sum_{i,j=1}^{d}a^{i,j}e_{i}\otimes
e_{j}+\sum_{i=1}^{d}b^{i}e_{i}+\sum_{j<k}b^{j,k}\left[ e_{j},e_{k}\right]
+\int_{G^{(2)}}\{\exp ({\log _{\left( 2\right) }g})-g1_{\{|g|<1\}}\}\mathbf{K%
}\left( dg\right) \right) \right]  \label{equ:LKgen}
\end{equation}
\end{theorem}

\begin{proof}
We saw in Corollary \ref{DEforminimaljumpextension} that $S$ solves 
\begin{equation*}
S_{t}=1+\int_{0}^{t}S_{s-}\otimes d\mathbf{X}_{s}\mathbf{+}\sum_{0<s\leq
t}S_{s-}\otimes \{\exp (\log ^{(2)}\Delta \mathbf{X}_{s})-\Delta \mathbf{X}%
_{s}\}.
\end{equation*}%
With notation as in the proof of Theorem \ref{theo:LRPChar}, 
\begin{eqnarray*}
\mathbb{X}_{s,t} &=&\pi _{2}\exp \left( X_{s,t}+\mathbb{A}%
_{s,t}+Z_{t}-Z_{s}\right) =\frac{1}{2}X_{s,t}\otimes X_{s,t}+\mathbb{A}%
_{s,t}+Z_{s,t} \\
\mathbb{X}_{s,t}^{\mathrm{I}} &=&\frac{1}{2}\left( X_{s,t}\otimes X_{s,t}-%
\left[ X,X\right] _{s,t}\right) +Z_{s,t}
\end{eqnarray*}%
where we recall that $\left( X,Z\right) $ is a $\mathbb{R}^{d}\oplus 
\mathfrak{so}\left( d\right) $ valued L\'{e}vy process. With $%
Z_{s,t}=Z_{t}-Z_{s}$, we note additivity of $\Xi :=\mathbb{X}-\mathbb{X}^{%
\mathrm{I}}$ given by%
\begin{eqnarray*}
\Xi _{s,t} &:&=\frac{1}{2}\left[ X,X\right] _{s,t}+Z_{s,t} \\
&=&\frac{1}{2}a\left( t-s\right) +\frac{1}{2}\sum_{r\in (s,t]}|\Delta
X_{r}|^{\otimes 2}+Z_{s,t}.
\end{eqnarray*}%
But then%
\begin{equation*}
\int_{0}^{t}S_{s-}\otimes d\mathbf{X}_{s}=\int_{0}^{t}S_{s-}\otimes d\mathbf{%
X}_{s}^{\mathrm{I}}+\int_{0}^{t}S_{s-}\otimes d\Xi
\end{equation*}%
and so, thanks to Theorem \ref{rough=ito} on consistency of It\^{o}- with
rough integration, we can express $S$ as solution to a proper It\^{o}
integral equation,%
\begin{eqnarray*}
S_{t} &=&1+\int_{0}^{t}S_{s-}\otimes dX_{s}\mathbf{+}\int_{0}^{t}S_{s-}%
\otimes d\Xi +\sum_{0<s\leq t}S_{s-}\otimes \{\exp (\log ^{(2)}\Delta 
\mathbf{X}_{s})-\Delta \mathbf{X}_{s}\} \\
&\equiv &1+\left( 1\right) +(2)+(3).
\end{eqnarray*}%
Let $M^{X}$ be the martingale part in the It\^{o}--L\'{e}vy decomposition of 
$X$, write also $N^{\mathbb{K}}$ for the Poisson random measure with
intensity $ds\mathbb{K}\left( dy\right) $. Then, with $\mathfrak{b}\equiv
\sum_{j<k}b^{jk}\left[ e_{j},e_{k}\right] $, 
\begin{eqnarray*}
X_{t} &=&M_{t}^{X}+bt+\int_{(0,t]\times \left\{ \left\vert y\right\vert
+\left\vert a\right\vert \geq 1\right\} }yN^{\mathbb{K}}\left( ds,d\left( y,%
\mathfrak{a}\right) \right) \,\,\,\,\,\in \mathbb{R}^{d} \\
Z_{t} &=&M_{t}^{Z}+\mathfrak{b}t+\int_{(0,t]\times \left\{ \left\vert
y\right\vert +\left\vert a\right\vert \geq 1\right\} }\mathfrak{a}N^{\mathbb{%
K}}\left( ds,d\left( y,\mathfrak{a}\right) \right) \in \mathfrak{so}\left(
d\right) \\
\Xi _{t} &=&\frac{1}{2}at+\frac{1}{2}\int_{(0,t]\times \left\{ \left\vert
y\right\vert \geq 1\right\} }y^{\otimes 2}N^{\mathbb{K}}\left( ds,d\left( y,%
\mathfrak{a}\right) \right) +Z_{t}\in \left( \mathbb{R}^{d}\right) ^{\otimes
2}.
\end{eqnarray*}%
Check (inductively) integrability of $S_{t}$ and note that $\int
S_{s-}dM_{s} $ has zero mean, for either martingale choice. It follows that 
\begin{eqnarray*}
\Phi _{t} &=&1+\int_{0}^{t}\Phi _{s}\otimes (C_{1}+C_{2}+C_{3})ds\text{
where } \\
C_{1} &=&b+\int_{g^{2}(\mathbb{R}^{d})}y1_{\left\{ \left\vert y\right\vert
+\left\vert a\right\vert >1\right\} }\mathbb{K}\left( y,\mathfrak{a}\right) ,
\\
C_{2} &=&\frac{1}{2}a+\frac{1}{2}\int_{g^{2}(\mathbb{R}^{d})}y^{\otimes
2}1_{\left\{ \left\vert y\right\vert +\left\vert a\right\vert >1\right\} }%
\mathbb{K}\left( y,\mathfrak{a}\right) +\mathfrak{b}+\int_{g^{2}(\mathbb{R}%
^{d})}\mathfrak{a}1_{\left\{ \left\vert y\right\vert +\left\vert
a\right\vert >1\right\} }\mathbb{K}\left( y,\mathfrak{a}\right) , \\
C_{3} &=&\int_{G^{(2)}\left( \mathbb{R}^{d}\right) }\{\exp ({\log }^{\left(
2\right) }{g})-g\}\mathbf{K}\left( dg\right) .
\end{eqnarray*}%
Recall $\mathbb{K}=\log _{\ast }^{(2)}\mathbf{K}$ so that the sum of the
three integrals over $g^{2}(\mathbb{R}^{d})$ is exactly%
\begin{equation*}
\int_{G^{(2)}}g1_{\left\{ |g|\geq 1\right\} }\}\mathbf{K}\left( dg\right)
\end{equation*}%
where $\left\vert g\right\vert =\left\vert \log g\right\vert =\left\vert
y\right\vert +\left\vert \mathfrak{a}\right\vert $. And it follows that%
\begin{equation*}
C_{1}+C_{2}+C_{3}=\frac{1}{2}a+b+\mathfrak{b+}\int_{G^{(2)}\left( \mathbb{R}%
^{d}\right) }\{\exp ({\log _{\left( 2\right) }g})-g1_{\left\{ |g|<1\right\}
}\}\mathbf{K}\left( dg\right)
\end{equation*}%
which concludes our proof.
\end{proof}

\subsection{The moment problem for random signatures}

Any L\'{e}vy rough path $\mathbf{X}\left( \omega \right) $ over some fixed
time horizon $\left[ 0,T\right] $ determines, via minimal jump exension
theorem, a random group-like element, say $S_{0,T}\left( \omega \right) \in
T((\mathbb{R}^{d}))$. What information does the expected signature really
carry? This was first investigated by Fawcett \cite{fawcett}, and more
recently by Chevyrev \cite{chevyrev}. Using his criterion we can show

\begin{theorem}
The law of $S_{0,T}\left( \omega \right) $ is uniquely determined from its
expected signature whenever%
\begin{equation*}
\forall \lambda >0:\int_{y\in G^{(2)}:\left\vert y\right\vert >1}\exp \left(
\lambda \left\vert y\right\vert \right) \mathbf{K}\left( dy\right) <\infty .
\end{equation*}
\end{theorem}

\begin{proof}
As in \cite{chevyrev}, we need to show that $\exp \left( C\right) $,
equivalently $C=\left( C^{0},C^{1},C^{2},...\right) \in T((\mathbb{R}^{d}))$%
, has sufficiently fast decay as the tensor levels grow. In particular, only
the the jumps matter. More precisely, by a criterion put forward in \cite%
{chevyrev} we need to show that 
\begin{equation*}
\sum \lambda ^{m}C^{m}<\infty
\end{equation*}%
where (for $m\geq 3$)$,$%
\begin{equation*}
C^{m}=\pi _{m}\left( \int_{G^{(2)}}\left( e_{\left( n\right) }^{\log
_{\left( 2\right) }g}-g\right) \mathbf{K}\left( dg\right) \right) \in \left( 
\mathbb{R}^{d}\right) ^{\otimes m}.
\end{equation*}%
We leave it as elementary exercise to see that this is implied by the
exponential moment condition on $\mathbf{K}$.
\end{proof}

\section{Further classes of stochastic processes}

\subsection{Markov jump diffusions}

\label{sec:Markov}

Consider a $d$-dimensional strong Markov with generator 
\begin{eqnarray*}
\left( \mathcal{L}f\right) \left( x\right) &=&\frac{1}{2}\sum_{i,j\in
I}a^{i,j}\left( x\right) \partial _{i}\partial _{j}f+\sum_{i\in
I}b^{i}\left( x\right) \partial _{i}f \\
&&+\int_{\mathbb{R}^{d}}\{f\left( x+y\right) -f\left( x\right) -1_{\left\{
y\leq 1\right\} }\sum_{i\in I}y^{i}\partial _{i}f\}K\left( x,dy\right) .
\end{eqnarray*}%
Throughout, assume $a=\sigma \sigma ^{T}$ and $\sigma ,b$ bounded Lipschitz, 
$K\left( x,\cdot \right) $ a L\'{e}vy measure, with uniformly integrable
tails. Such a process can be constructed as jump diffusion \cite{Jacod}, the
martingale problem is discussed in Stroock \cite{Stroock75}. As was seen,
even in the L\'{e}vy case, with (constant) L\'{e}vy triplet $\left(
a,b,K\right) $, showing finite $p$-variation in rough path sense is
non-trivial, the difficulty of course being the stochastic area 
\begin{equation*}
A_{s,t}\left( \omega \right) =\mathrm{Anti}\int_{(s,t]}(X_{r}^{-}-X_{s})%
\otimes dX\,\,\,\in \mathfrak{so}\left( d\right) ;
\end{equation*}%
where stochastic integration is understood in It\^{o} sense. In this section
we will prove

\begin{theorem}
\label{theo:MarkovRP} With probability one, $X\left( \omega \right) $ lifts
to a $G^{(2)}$-valued path, with increments given by%
\begin{equation*}
\mathbf{X}_{s,t}:=\exp ^{\left( 2\right) }\left( X_{s,t}+A_{s,t}\right) =%
\mathbf{X}_{s}^{-1}\otimes \mathbf{X}_{t}
\end{equation*}%
and $\mathbf{X}$ is a c\'{a}dl\'{a}g Marcus like, geometric $p$-rough path,
for any $p>2$. 
\end{theorem}

Note the immediate consequences of this theorem: the minimal jump extension
of the geometeric rough $\left( X,\mathbb{X}^{\mathrm{M}}\right) $ can be
identified with the Marcus lift, stochastic integrals and differential
equations driven by $X$ can be understood deterministically as function of $%
\mathbf{X}\left( \omega \right) $ and are identified with corresponding
rough integrals and canonical equations. As in the L\'{e}vy case discussed
earlier, we base the proof on the expected signature and point out some
Markovian aspects of independent interest. Namely, we exhibit the step-$N$
Marcus lift as $G^{(N)}$-valued Markov process and compute its generator. To
this end, recall (e.g. \cite[Remark 7.43]{FrizVictoir}) the generating
vector fields $U_{i}\left( g\right) =g\otimes e_{i}$ on $G^{(N)}$, with the
property that 
\begin{equation*}
\mathrm{Lie}\left( U_{1},\dots ,U_{d}\right) |_{g}=\mathcal{T}_{g}G^{(N)}.
\end{equation*}

\begin{proposition}
Consider a $d$-dimensional Markov process $\left( X\right) $ with generator
as above and the Marcus canonical equation $dS=S\otimes \diamond dX$,
started from%
\begin{equation*}
1\equiv \left( 1,0,\dots ,0\right) \in G^{\left( N\right) }\left( \mathbb{R}%
^{d}\right) \subset T^{\left( N\right) }\left( \mathbb{R}^{d}\right) .
\end{equation*}%
Then $S$ takes values in $G^{\left( N\right) }\left( \mathbb{R}^{d}\right) $
and is Markov with generator, for $f\in C_{c}^{2}$,%
\begin{eqnarray*}
(\mathcal{L}f)(x)=(\mathcal{L}^{\left( N\right) }f)\left( x\right) &=&\frac{1%
}{2}\sum_{i,j\in I}a^{i,j}\left( \pi _{1}\left( x\right) \right)
U_{i}U_{j}f+\sum_{i\in I}b^{i}\left( \pi _{1}\left( x\right) \right) U_{i}f
\\
&&+\int_{\mathbb{R}^{d}}\{f\left( x\otimes Y\right) -f\left( x\right)
-1_{\left\{ y\leq 1\right\} }\sum_{i\in I}y^{i}U_{i}f\}K\left( x,dy\right) ,
\\
\text{with }Y &\equiv &\exp ^{\left( n\right) }\left( y\right) .
\end{eqnarray*}
\end{proposition}

\begin{proof}
(Sketch) Similar to the proof of Theorem \ref{lk_formula}. Write $X=M+V$ for
the semimartingale decomposition of $X$. We have 
\begin{equation*}
dS=S\otimes \diamond dX=\sum_{i\in I}U_{i}\left( S\right) \diamond dX^{i}
\end{equation*}%
and easily deduce an evolution equation for $f\left( S_{t}\right) =f\left(
1\right) $. Taking the expected value leads to the form $(\mathcal{L}f)$.
\end{proof}

\bigskip

Since $N$ was arbitrary, this leads to the expected signature. We note that
in the (L\'evy) case of $x$-independent characteristics, $\Phi$ does not
depend on $x$ in which case the PIDE reduces to the ODE $\partial_t \Phi = C
\otimes \Phi$ which leads to the L\'evy--Kinthchine form $\Phi(t) = \exp (C
t )$ obtained previously. We also that the solution $\Phi =\left( 1,\Phi
^{1},\Phi ^{2},\dots \right) $ to the PIDE system given in the next theorem
can be iteratively constructed. In absence of jumps this systems reduces to
a system of PDEs derived by Ni Hao \cite{nihao, lyonsni}.

\begin{theorem}[PIDE for expected signature]
Assume uniformly bounded jumps, $\sigma ,b$ bounded and Lipschitz, $a=\sigma
\sigma ^{T}$, the expected signautre $\Phi \left( x,t\right) =E^{x}S_{0,t}$
exists. Set%
\begin{eqnarray*}
C\left( x\right) &:=&\sum_{i\in I}b^{i}\left( x\right) e_{i}+\frac{1}{2}%
\sum_{i,j\in I}a^{i,j}\left( x\right) e_{i}\otimes e_{j}+\int_{\mathbb{R}%
^{d}}\left( Y-1-\mathbb{I}_{\left\{ y\leq 1\right\} }\sum_{i\in
I}y^{i}e_{i}\right) K\left( x,dy\right) \\
& & \text{with }Y =\exp \left( y\right) \in T((\mathbb{R}^{d}))\text{.}
\end{eqnarray*}%
Then $\Phi \left( x,t\right) $ solves%
\begin{equation*}
\left\{ 
\begin{array}{l}
\partial _{t}\Phi =C\otimes \Phi +\mathcal{L}\Phi +\sum_{i,j\in
I}a^{i,j}\left( \partial _{j}\Phi \right) \left( x\right) e_{i} \\ 
\text{ \ }+\int_{\mathbb{R}^{d}}\left( Y-1\right) \otimes \left( \Phi \left(
x\otimes Y\right) -\Phi \left( x\right) \right) K\left( x,dy\right) \\ 
\Phi \left( x,0\right) =1.%
\end{array}%
\right.
\end{equation*}
\end{theorem}

\begin{proof}
It is enough to establish this in $T^{(N)}\left( \mathbb{R}^{d}\right) $,
for arbitrary integer $N$. We can see that 
\begin{equation*}
\mathbb{E}^{x}\mathbf{X}_{t}^{\left( N\right) }=:u\left( x,t\right) ,
\end{equation*}%
for $x\in G^{(N)}\left( \mathbb{R}^{d}\right) \subset T^{(N)}\left( \mathbb{R%
}^{d}\right) $ is well-defined, in view of the boundedness assumptions made
on the coefficients, and then a (vector-valued, unique linear growth)
solution to the backward equation%
\begin{eqnarray*}
\partial _{t}u &=&\mathcal{L}u, \\
u\left( x,0\right) &=&x\in T^{(N)}\left( \mathbb{R}^{d}\right) .
\end{eqnarray*}%
It is then clear that 
\begin{equation*}
E^{x}\mathbf{X}_{0,t}^{\left( N\right) }=x^{-1}\otimes u\left( x,t\right)
=:\Phi \left( x,t\right)
\end{equation*}%
also satisfies a PDE.\ Indeed, noting the product rule for second order
partial-integro operators, 
\begin{eqnarray*}
\left( \mathcal{L}\left[ fg\right] \right) \left( x\right) &=&\left( \left( 
\mathcal{L}\left[ f\right] \right) g\right) \left( x\right) +\left( f%
\mathcal{L}\left[ g\right] \right) \left( x\right) +\Gamma (f,g), \\
\Gamma (f,g) &=&\sum_{i,j\in I}a^{i,j}\left( U_{i}fU_{j}g\right) \left(
x\right) +\int_{G^{(2)}}\left( f\left( x\otimes Y\right) -f\left( x\right)
\right) \left( g\left( x\otimes Y\right) -g\left( x\right) \right) \mathbf{K}%
\left( dy\right)
\end{eqnarray*}%
and also noting the action of $U_{v}$ on $f\left( x\right) \equiv x$, namely 
$U_{i}f=x\otimes e_{v}$, we have 
\begin{eqnarray*}
\mathcal{L}x &=&x\otimes C:=x\otimes \left\{ \sum_{v\in J}b^{v}\otimes e_{v}+%
\frac{1}{2}\sum_{i,j\in I}a^{i,j}e_{i}\otimes e_{j}+\int_{G^{(2)}}\left( Y-1-%
\mathbb{I}_{\left\{ y\leq 1\right\} }\sum_{v\in J}Y^{v}\otimes e_{v}\right) 
\mathbf{K}\left( dy\right) \right\} , \\
\Gamma (x,g) &=&x\otimes \left\{ \sum_{i,j\in I}a^{i,j}\left( U_{j}g\right)
\left( x\right) e_{i}+\int_{G^{(2)}}\left( Y-1\right) \left( g\left(
x\otimes Y\right) -g\left( x\right) \right) \mathbf{K}\left( dy\right)
\right\} .
\end{eqnarray*}%
As a consequence, 
\begin{equation*}
x\otimes \partial _{t}\Phi =\partial _{t}u=\mathcal{L}u=\mathcal{L}\left(
x\otimes \Phi \right) =\left( \mathcal{L}x\right) \otimes \Phi +x\otimes 
\mathcal{L}\left[ \Phi \right] +\Gamma (x,\Phi )
\end{equation*}%
and hence%
\begin{equation}
\partial _{t}\Phi =C\otimes \Phi +\left\{ \mathcal{L}\left[ \Phi \right]
+\sum_{i,j\in I}a^{i,j}\left( U_{j}\Phi \right) \left( x\right)
e_{i}+\int_{G^{(2)}}\left( Y-1\right) \left( \Phi \left( x\otimes Y\right)
-\Phi \left( x\right) \right) \mathbf{K}\left( dy\right) \right\} .
\label{equ:PIDEsig}
\end{equation}
\end{proof}

\bigskip

We can now show rough path regularity for general jump diffusions.

\begin{proof}
(Theorem \ref{theo:MarkovRP}) Only $p$-variation statement requires a proof.
The key remark is that the above PIDE implies%
\begin{eqnarray*}
\Phi _{t} &=&1+\left( \partial _{t}|_{t=0}\phi \right) t+O\left(
t^{2}\right) = 1+Ct+O\left( t^{2}\right)
\end{eqnarray*}%
where our assumptions on $a,b,K$ guarantee uniformity of the $O$-term in $x$%
. We can then argue exactly as in the proof of Corollary \ref{cor:will3}.
\end{proof}

\subsection{Semimartingales}

\label{sec:semi}

In \cite{lepingle} L\'{e}pingle established finite $p$-variation of general
semimartingales, any $p>2$, together with powerful Burkholder--Davis--Gundy
type estimates. For \textit{continuous} semimartingales the extension to the
(Stratonovich=Marcus) rough path lift was obtained in \cite{FVsemi08}, see
also \cite[Chapter 14]{FrizVictoir}, but so far the general (discontinuous)
case eluded us. (By Proposition \ref{prop:semimartlift} it does not matter
if one establishes finite $p$-variation in rough path sense for the It\^o-
or Marcus lift.)

As it is easy to explain, let us just point to the difficulty in extending L%
\'{e}pingle in the first place: he crucially relies on Monroe's result \cite%
{Mo72}, stating that every (scalar!) c\'{a}dl\'{a}g semimartingale can be
written as a time-changed scalar Brownian motion for a (c\'{a}dl\'{a}g)
family of stopping times (on a suitably extended probability space). This,
however, fails to hold true in higher dimensions and not every (Marcus or
It\^o) lifted general semimartingale\footnote{%
... and certainly not every Markov jump diffusion as considered in the last
section ....} will be a (c\'{a}dl\'{a}g) time-change of some enhanced
Brownian motion \cite[Chapter 13]{FrizVictoir}, in which case the finite $p$%
-variation would be an immediate consequence of known facts about the
enhanced Brownian motion (a.k.a. Brownian rough path) and invariance of $p$%
-variation under reparametrization.

A large class of general semimartingales for which finite $p$-variation (in
rough path sense, any $p>2$) can easily be seen, consists of those with
summable jumps. Following Kurtz et al. \cite[p. 368]{kpp}, the Marcus
version" of such a s semimartingale, i.e. with jump replaced by straight
lines over stretched time, may be interpreted as \textit{continuous
semimartingale.} One can then apply \cite{FVsemi08, FrizVictoir} and again
appeal to invariance of $p$-variation under reparametrization, to see that
such (enhanced) semimartingales have a.s. $p$-rough sample paths, any $p>2$.

Another class of general semimartingales for which finite $p$-variation can
easily be seen, consists of time-changed L\'evy processes (a popular class
of processes used in mathematical finance). Indeed, appealing once more to
invariance of $p$-variation under reparametrization, the statement readily
follows from the corresponding $p$-variation regularity of L\'evy rough
paths.

\subsection{Gaussian processes}

\label{sec:gauss}

We start with a brief review of some aspects of the work of Jain--Monrad 
\cite{JM83}. Given a (for the moment, scalar) zero-mean, separable Gaussian
process on $\left[ 0,T\right] $, set $\sigma^2 \left( s,t\right) =\mathbb{E}%
X_{s,t}^{2}=|X_{t}-X_{s}|_{L^{2}}^{2}$.We regard the process $X$ as Banach
space valued path $\left[ 0,T\right] \rightarrow H=L^{2}\left( P\right) $
and assume finite $2\rho $-variation, in the sense of Jain--Monrad's condition
\begin{equation}
F\left( T\right) :=\sup_{\mathcal{P}}\sum_{\left[ u,v\right] \in \mathcal{P}%
}|\sigma^2 \left( u,v\right) |^{\rho }=\sup_{\mathcal{P}}\sum_{\left[ u,v%
\right] \in \mathcal{P}}|X_{t}-X_{s}|_{L^{2}}^{2\rho }<\infty  \label{JMcond}
\end{equation}%
with partitions $\mathcal{P}$ of $\left[ 0,T\right] $. It is elementary to
see that $p$-variation paths can always be written as time-changed H\"{o}%
lder continuous paths with exponent $1/p$ (see e.g. Lemma 4.3. in \cite%
{dudley}). Applied to our setting, with $\alpha ^{\ast }=1/(2\rho )$, $%
\tilde{X}\in C^{\alpha ^{\ast }\text{-H\"{o}l }}\left( [0,F(T)],H\right) $
so that%
\begin{equation*}
\tilde{X}\circ F=X\in W^{2\rho }\left( \left[ 0,T\right] ,H\right) .
\end{equation*}%
Now in view of the classical Kolmogorov criterion, and equivalence of
moments for Gaussian random variables, knowing 
\begin{equation*}
\left\vert \tilde{X}_{t}-\tilde{X}_{s}\right\vert _{L^{2}}\leq C  \left\vert t-s\right\vert ^{\alpha ^{\ast }}
\end{equation*}%
implies that $\tilde{X}$ (or a modification thereof) has a.s. $\alpha $-H%
\"{o}lder samples paths, any $\alpha <\alpha ^{\ast }$. But then, trivialy, $%
\tilde{X}$ has a.s. finite $p$-variation sampe paths, any $p>1/\alpha =2\rho 
$, and so does $X$ by invariance of $p$-variation under reparametrization.
(I should be noted that such $X$ has only discontinuities at deterministic times,
inherited from the jumps of $F$.)
In a nutshell, this is one of the main results of Jain--Monrad \cite{JM83},
as summarized in by Dudley--Norvai{\v{s}}a in \cite[Thm 5.3]{dudley}. We have the following
extension to Gaussian rough paths.

\begin{theorem}
\label{theo:Gauss} Consider a \thinspace $d$-dimensional zero-mean,
separabale Gaussian process $\left( X\right) $ with independent components.
Let $\rho \in \lbrack 1,3/2)$ and assume 
\begin{equation}
\sup_{\mathcal{P},\mathcal{P}^{\prime }}\sum_{\substack{ \left[ s,t\right]
\in \mathcal{P}  \\ \lbrack u,v]\in \mathcal{P}^{\prime }}}|\mathbb{E(}%
X_{s,t}\otimes X_{u,v})|^{\rho }<\infty  \label{FVcond}
\end{equation}%
Then $X$ has a c\'adl\'ag modification, denoted by the same letter, which
lifts a.s. to a random geometric c\'{a}dl\'{a}g rough path, with $\mathbb{A}=\mathrm{%
Anti}\left( \mathbb{X}\right) $ given as $L^{2}$-limit of Riemann--Stieltjes
approximations.
\end{theorem}

\begin{proof}
In a setting of continuous Gaussian processes, condition (\ref{FVcond}),
i.e. finite $\rho $-variation of the covariance, is well-known \cite%
{FrizVictoir, FrizHairer}. It plainly implies the Jain--Monrad condition (%
\ref{JMcond}), for each component $\left( X^{i}\right) $. With $F\left(
t\right) :=\sum_{i=1}^{d}F^{i}\left( t\right) $ we can then write 
\begin{equation*}
\tilde{X}\circ F=X
\end{equation*}%
for some $d$-dimensional, zero mean, (by Kolmogorov criterion: continuous) Gaussian process $\tilde{X}$, whose
covariance also enjoys finite $\rho $-variation. We can now emply standard
(continuous) Gaussian rough path theory \cite{FrizVictoir, FrizHairer} and construct a canoncial geometric
rough path lift of $\tilde{X}$. That is,%
\begin{equation*}
\mathbf{\tilde{X}=(}\tilde{X},\mathbb{\tilde{X}})\in \mathcal{C}^{\rho }
\end{equation*}%
with probability $1$. The desired geometric c\'{a}dl\'{a}g rough path lift is
then given by. 
\begin{equation*}
\left( X,\mathbb{X}\right) =\mathbf{X:=\tilde{X}}\circ F.
\end{equation*}%
The statement about $L^{2}$-convergence of Riemann--Stieltjes approximations
follows immediately for the corresponding statements for $\mathrm{Anti}(%
\mathbb{\tilde{X})}$, as found in \cite[Ch. 10.2]{FrizHairer}.
\end{proof}

%


\bibliographystyle{plain}
\bibliography{refs}

%
%
%
%
%

\end{document}